\documentclass[11pt]{amsart}
\usepackage[a4paper, margin=1.25in]{geometry}
\usepackage{amsthm}
\usepackage{xcolor,bm}
\usepackage{amssymb,amsmath,amsthm,amsfonts}
\usepackage{tikz-cd}
\usepackage{enumerate}
\usepackage{enumitem}
\usepackage{setspace}
\usepackage{tikz}
\usepackage{empheq}
\usepackage[all]{xy}
\usepackage{verbatim}
\usepackage[normalem]{ulem}
\usepackage{todonotes}
\usepackage{mathrsfs}

\usepackage{hyperref}  
\hypersetup{linkcolor=blue,citecolor=blue,filecolor=blue,urlcolor=blue} 

\numberwithin{equation}{section}
\numberwithin{figure}{section}

\newtheorem{theorem}{Theorem}[section]
\newtheorem{proposition}[theorem]{Proposition}

\newtheorem{lemma}[theorem]{Lemma}

\newtheorem{conjecture}[theorem]{Conjecture}

\theoremstyle{definition}
\newtheorem{definition}[theorem]{Definition}

\newtheorem{remark}[theorem]{Remark}

\usepackage[
  backend=biber,
  style=alphabetic,
  maxcitenames=99,
  maxbibnames=99,
  giveninits=false
]{biblatex}
\definecolor{myblue}{rgb}{0.6, 0.9, 1}

\makeatletter

\newcommand{\Rmnum}[1]{\expandafter\@slowromancap\romannumeral #1@}
\makeatother

\definecolor{myblue}{rgb}{0.6, 0.9, 1}
\definecolor{mygreen}{rgb}{0,0,1}
\definecolor{purple}{rgb}{0.6,0.2,1}
\definecolor{orange}{rgb}{0.8,0,0.2}

\newcommand{\bP}{\mathbb{P}}
\newcommand{\C}{\mathbb{C}}
\newcommand{\Q}{\mathbb{Q}}
\newcommand{\bU}{\mathbb{U}}

\newcommand{\R}{\mathbb{R}}
\newcommand{\bN}{\mathbb{N}}

\newcommand{\bH}{\mathbb{H}}
\newcommand{\bA}{\mathbb{A}}

\newcommand{\mcF}{\mathcal{F}}

\newcommand{\mcE}{\mathcal{E}}

\newcommand{\eps}{\varepsilon}

\newcommand{\mc}{\mathcal}

\newcommand{\Res}{\operatorname{Res}}

\newcommand{\supp}{\operatorname{supp}}

\newcommand{\Pic}{\operatorname{Pic}}

\newcommand{\Z}{\mathbb{Z}}
\newcommand{\mbf}{\mathbf}
\newcommand{\bb}{\mathbb}

\newcommand{\cal}{\mathcal}
\newcommand{\ovl}{\overline}

\newcommand{\h}{\widehat{h}}
\newcommand{\Spec}{\operatorname{Spec}}
\newcommand{\Rat}{\operatorname{Rat}}

\def\sH{\mathscr{H}}

\newcommand{\calJ}{\mathcal{J}}
\newcommand{\Prep}{\operatorname{Prep}}

\newcommand{\wh}{\widehat}
\newcommand{\Div}{\operatorname{Div}}
\newcommand{\scr}{\mathscr}
\newcommand{\vol}{\operatorname{vol}}

\newcommand{\PGL}{\mathrm{PGL}}
\newcommand{\Aut}{\mathrm{Aut}}
\newcommand{\an}{\mathrm{an}}
\newcommand{\can}{\mathrm{can}}
\theoremstyle{definition} 

\newcommand{\intt}{\operatorname{int}}

\begin{document}
\title[Non-Archimedean Rigidity and Uniformity]{Non-Archimedean Rigidity and Uniformity for Common Preperiodic Points}

\author{Chen Gong and Jit Wu Yap}

\date{today}

\begin{abstract}
 Let $k$ be an algebraically closed, complete non-Archimedean field of residue characteristic $0$. Let $f$ be a polynomial of degree at least $2$ over $k$ which does not have potential good reduction. We prove that if $g$ is any other polynomial with the same Julia set, then $f$ and $g$ must be dynamically related. 
 
 As a consequence, we show that for any two complex polynomials $f,g$  of degree at least $2$, either their sets of preperiodic points coincide, or the number of their common preperiodic points is uniformly bounded above by a constant depending only on the degrees,
thereby answering a conjecture of DeMarco--Krieger--Ye for polynomials. We also establish relative results, allowing us to prove special cases of the DeMarco--Mavraki conjecture. 
\end{abstract}

	\maketitle
\tableofcontents

\section{Introduction}
\textbf{Standing assumptions}. Throughout the paper, any metrized field $k$ will be assumed to be complete and non-trivially valued.

\subsection{Non-Archimedean rigidity of Julia sets} Let $f: \bb{P}^1(\bb{C}) \to \bb{P}^1(\bb{C})$ be a rational map of degree $d \geq 2$. One of the central objects in complex dynamics is the Julia set of $f$, denoted by $\calJ_f$ \cite{Mi06}, which captures the locus of instability for the dynamics. The Julia set $\calJ_f$ carries with it an equilibrium measure $\mu_f$ \cite{Bro65, Lyu83, FLM83}, which is the unique Radon probability measure satisfying $\supp \mu_f = \calJ_f$ and 
\begin{equation} \label{eq: IntroMixing1}
f_* \mu_f = \mu_f, \quad f^* \mu_f = d \cdot \mu_f.
\end{equation}
It is an old problem, dating back to Julia \cite{Jul22}, to determine when two rational maps $f,g$ share the same Julia set or even the same equilibrium measure. This problem has been extensively studied both in the polynomial setting \cite{Jul22,BE87,Bea90,Bea92,SS95} and also for general rational maps \cite{Ere89,Lev90,LP97,Din00,Ye15, DFG23,JX25, MS22}. These works reveal a rigidity phenomenon: two non-integrable\footnote{We use the term “integrable” by analogy with the notion of integrable systems in Hamiltonian dynamics, see~\cite{Ves91,Cas06}.} (also called non-exceptional in \cite{LP97}) rational maps share the same equilibrium measure if and only if they are dynamically related by an algebraic correspondence. These rigidity results, combined with arithmetic equidistribution \cite{BR06, CL06, FRL06, Yua08, Gau23, YZ24}, play an important role in arithmetic dynamics; see for example \cite{BD11,BD14,FG22,JX23b, DM24, MS22, DMY26}.
\par 
Our first goal in this paper is to establish the corresponding rigidity results for Julia sets of polynomials over a non-Archimedean field $k$ with residue characteristic $0$. The foundations of one-variable non-Archimedean dynamics were developed by Benedetto \cite{Ben98, Ben00, Ben01a, Ben01b} and Rivera-Letelier \cite{RL03}, leading to a theory of Fatou and Julia sets paralleling the complex case. Unlike their complex counterparts, however, these sets are not subsets of $\bb{P}^1(k)$, but instead of the Berkovich analytification $\bb{P}_k^{1,\an}$, a larger space introduced by Berkovich \cite{Ber90} to facilitate non-Archimedean geometry. The theory of non-Archimedean dynamics on $\bb{P}_k^{1,\an}$ 
arises
naturally from considering degenerations of families of rational maps over $\bb{C}$ \cite{Ki06, DF14, Fav20, Luo21, Luo22, FG25}. 

Similar to the complex case, to any rational map \( f \) of degree \( d \geq 2 \) over \( k \), one can associate an equilibrium measure \( \mu_f \), which is a Radon probability measure on \( \mathbb{P}^{1,\an}_k \) supported on the Julia set \( \mathcal{J}_f \) and satisfying \eqref{eq: IntroMixing1}. 
This measure was first constructed by Favre--Rivera-Letelier \cite{FRL04}, and independently by Baker--Rumely~\cite{BR06} and Chambert-Loir--Thuillier~\cite{Thu05,CL06}. Its ergodic and entropy-theoretic properties were further investigated by Favre--Rivera-Letelier \cite{FRL10, FRL26}.

We may now state our first main result. For a rational map $f$ of degree $d \geq 2$ over $k$, we say that $f$ does not have potential good reduction if its Julia set $\cal{J}_f$ is not a single point. For a Julia set $\cal{J}$, we define $\Aut(\cal{J})$ to be the set of affine linear transformations of $k$ that preserve $\cal{J}$.

\begin{theorem}\label{IntroRigidity1}
Let \( k \) be an algebraically closed, complete non-Archimedean field with residue characteristic \( 0 \). 
Let \( \mathcal{J} \) be the Julia set of a polynomial of degree at least \( 2 \) which does not have potential good reduction. 
Let \( g \) be a polynomial of minimal degree whose Julia set coincides with \( \mathcal{J} \).

For any polynomial $f$ of degree $d\ge 2$, the following conditions are equivalent:
\begin{enumerate}
\item There exist an integer $n \ge 1$ and $\sigma \in \Aut(\calJ)$ such that $f = \sigma \circ g^{n}.$
\item The Julia sets coincide:
$\calJ_f = \calJ.$
\item The equilibrium measures coincide: $
\mu_f = \mu_g.$
\item The Green functions coincide:
$G_f = G_g.$
\item The Böttcher coordinates at $\infty$ agree up to a root of unity:
$\phi_f = \lambda \phi_g$, where $\lambda$ is a root of unity.
\end{enumerate}
\end{theorem}

As a corollary of Theorem \ref{IntroRigidity1}, we are able to answer 
(Proposition \ref{prop: symmetryjulia})
a question raised by Favre--Gauthier in \cite[Remark 3.10]{FG22}. We note that our assumption of $\cal{J}$ being not a single point is necessary: for any monic polynomial $f$ with coefficients all having norm $\leq 1$, its Julia set $\cal{J}_f$ is always only the Gauss point $\{x_g\}$ and so we cannot possibly have rigidity. 
\par 
We note that (1) easily implies (2) to (5). The equivalence of (3) and (4) follows from \cite[Section 4]{FRL06} and the equivalence of (2) and (3) follows from \cite[Theorem~10.91]{BR10}. The main difficulty lies in showing that (2) implies (1). To overcome this, we construct a non-Archimedean analogue of DeMarco--McMullen's simplicial tree~\cite{DM08}, which serves as a combinatorial tool for studying the Julia set of a polynomial. The idea of using tree-like combinatorics to study dynamics of polynomials goes back to Branner--Hubbard \cite{BH88, BH92}. We show if $\mu_f = \mu_g$, then both $f$ and $g$ have the same associated tree, in particular having the same set of distinguished vertices, which allows us to study where the roots of $f^{n}(z) - \alpha$ and $g^n(z)-\alpha$ lie. We will explain this in detail in Section \ref{sec: Strategy}.

\subsection{Uniform bound on common preperiodic points} The second goal of this paper is to demonstrate how non-Archimedean rigidity results, such as Theorem~\ref{IntroRigidity1}, can be applied to establish uniformity results for unlikely intersection problems. We thank Niki Myrto Mavraki for enlightening us with this perspective. Let $f,g$ be two rational maps of degree $d \geq 2$ over $\bb{C}$. Then Baker--DeMarco \cite{BD11} showed that either 
$$\Prep(f) = \Prep(g) \text{ or } |\Prep(f) \cap \Prep(g)| < \infty$$
where $\Prep(f)$ denotes the set of preperiodic points for $f$ and similarly for $g$. A higher-dimensional version was proven by Yuan--Zhang \cite{YZ17}. In \cite{DKY21}, DeMarco--Krieger--Ye conjectured that there should exist a uniform constant $M = M(d) > 0$, depending only on $d$, such that if $f,g$ are of degree $d \geq 2$, then 
$$\Prep(f) = \Prep(g) \text{ or } |\Prep(f) \cap \Prep(g)| \leq M.$$
Our second main result not only proves their conjecture for polynomials, but also extends it to the case of polynomials of different degrees, using Theorem \ref{IntroRigidity1}.

\begin{theorem}
\label{thm: uniformprep}
Given any integers $d_1,d_2 \ge 2$, there exists a constant 
\( M = M(d_1,d_2) \), depending only on $d_1$ and $d_2$, such that for any two polynomials 
$f$ and $g$ of degrees $d_1$ and $d_2$ over $\mathbb{C}$, one of the following holds:
\[
\Prep(f) = \Prep(g) \quad \text{or} \quad 
|\Prep(f) \cap \Prep(g)| \le M.
\]
\end{theorem}

DeMarco--Krieger--Ye \cite{DKY20, DKY21} had previously shown their conjecture for rational maps belonging to either the quadratic unicritical family $\{z^2 + c\}_{c \in \bb{C}}$ or to the Lattès family associated to the Legendre family of elliptic curves. Poineau \cite{Poi25, Poi26} subsequently proved their conjecture for any pair of Lattès maps $(f,g)$, confirming a conjecture of Bogomolov--Fu--Tschinkel \cite{BFT18}. Mavraki--Schmidt \cite{MS22} proved the conjecture for any one-parameter family of rational maps and DeMarco--Mavraki \cite{DM24} showed the conjecture holds for a Zariski open $U \subset \Rat_d \times \Rat_d$ where $\Rat_d$ is the moduli space of all degree $d$ rational maps. The latter result was generalized by Gauthier~\cite{Gau24} for rational maps with different degrees.
\par 
Mavraki--Schmidt \cite{MS22} and DeMarco--Mavraki \cite{DM24} both rely on the theory of non-degeneracy, which is the key ingredient in the recent breakthroughs on the uniform Mordell--Lang conjecture by Dimitrov--Gao--Habegger \cite{DGH21}, Kühne \cite{Kuh21} and Gao--Ge--Kühne \cite{GGK21}. Our approach follows DeMarco--Krieger--Ye's \cite{DKY20, DKY21}, Gauthier's~\cite{Gau24} and Poineau's \cite{Poi25, Poi26}, where we study an adelic energy pairing, introduced by Favre--Rivera-Letelier \cite{FRL06} and further studied by Fili \cite{Fil17}, and prove uniform lower bounds. This is more in line with Yuan's proof of the uniform Bogomolov conjecture for curves \cite{Yua24}. To obtain uniform lower bounds, we combine Theorem \ref{IntroRigidity1} with a degeneration argument by ultrafilters following Favre--Gong \cite{FG25}, which we will explain in detail later.
\par 
Combining our uniform lower bounds on the energy pairing with the powerful machinery of Yuan--Zhang \cite{YZ24, Yua24}, we are able to establish relative versions of Theorem \ref{thm: uniformprep}. Let $\mathrm{Poly}^d$ and $\mathrm{Poly}_{\mathrm{mc}}^d$ denote  the moduli space of polynomials of degree $d$ and the space of monic centered polynomials of degree $d$ respectively. Let $X \subseteq \mathrm{Poly}_{\mathrm{mc}}^d \times \mathrm{Poly}^d$ be an irreducible subvariety defined over $\ovl{\bb{Q}}$. Let $\Delta \subseteq (\bb{P}^1)^2$ and consider the fibered power 
$$\Delta^{\dim X + 1} \times X \subseteq (\bb{P}^1)^{2(\dim X+1)} \times \mathrm{Poly}_{\mathrm{mc}}^d \times \mathrm{Poly}^d.$$

\begin{theorem} \label{IntroRelativeThm1}
Assume that as $(f,g)$ varies over $X(\ovl{\bb{Q}})$, the map $f$ is not always the monomial map $z^d$ and that $\Prep(f) \not = \Prep(g)$ for some $(f,g) \in X(\ovl{\bb{Q}})$. Then the set 
$$\left\{ (x_1,\ldots,x_{\dim X + 1}, f,g) \mid x_i \in \Prep(f) \cap \Prep(g) \right\} \subset (\Delta)^{\dim X +1} \times X$$
is not Zariski dense in $(\Delta)^{ \dim X +1} \times X$. 
\end{theorem}

In fact, we prove a Bogomolov version of Theorem \ref{IntroRelativeThm1}, which also handles points of small canonical height (Theorem \ref{RelativeManinMumford1}). One may deduce Theorem \ref{thm: uniformprep} from Theorem \ref{IntroRelativeThm1}, for example see \cite[Lemma 5.4]{DM24}. Theorem \ref{IntroRelativeThm1} can be viewed as a special case of the DeMarco--Mavraki conjecture \cite{DM26}, as named by Noytaptim--Zhong \cite{Zho25, NZ26}. The DeMarco--Mavraki conjecture may be viewed as a vast dynamical generalization of the Relative Manin--Mumford conjecture proven by Gao--Habegger \cite{GH23}. To state the conjecture, we have to introduce the notion of $\Phi$-special subvarieties for a family of dynamical systems. 
\par 
Let $S$ be a smooth and irreducible quasi-projective variety over $\bb{C}$ and let $\Phi: S \times \bb{P}^N \to S \times \bb{P}^N$ be an algebraic family of endomorphisms of degree $d$. Let $\bb{C}(S)$ be the function field of $S$ and let $\cal{X} \subseteq S \times \bb{P}^N$ be a closed irreducible subvariety that is flat over $S$. We say that $\cal{X}$ is $\Phi$-special if there exists a subvariety $\mathbf{Y} \subseteq \bb{P}^N_{\ovl{\bb{C}(S)}}$ containing the generic fiber $\mathbf{X}$ of $\cal{X}$ and a polarizable endomorphism $\Psi: \mathbf{Y} \to \mathbf{Y}$ over $\ovl{\bb{C}(S)}$ such that for some $n \in \bb{N}$, the following hold: 

\begin{enumerate}
 \item $\mathbf{\Psi}^n(\mbf{Y}) = \mbf{Y};$
 \item $\mathbf{\Psi}^n \circ \mathbf{\Phi} = \mathbf{\Phi} \circ \mathbf{\Psi}^n$ on $\mbf{Y}$ where $\mbf{\Psi}$ is the generic fiber of $\Phi$;
 \item $\mbf{X}$ is preperiodic for $\mbf{\Psi}$.
\end{enumerate}

We may now define the relative special dimension $r_{\Phi,\cal{X}}$ of $\cal{X}$ as 
$$r_{\Phi,\cal{X}}:= \min \{\dim \cal{Y} - \dim S \mid \cal{Y} \text{ is } \Phi-\text{special} \text{ and } \cal{X} \subset \cal{Y}\}.$$
Now let $\wh{T}_{\Phi}$ be the dynamical Green $(1,1)$-current associated to $\Phi$ \cite[Section 2.2]{GV25}. We can now state the DeMarco--Mavraki conjecture. 

\begin{conjecture} \label{IntroConjecture1} \textnormal{(DeMarco--Mavraki Conjecture)}
The following are equivalent:
\begin{enumerate}
    \item $\cal{X}$ contains a Zariski dense set of $\Phi$-preperiodic points; 
    \item $\wh{T}_{\Phi}^{\wedge r_{\Phi,\cal{X}}} \wedge [\cal{X}] \not = 0.$
\end{enumerate}
\end{conjecture}

DeMarco--Mavraki show the implication $(2)$ implies $(1)$ in their paper \cite{DM26}. As a consequence of Theorem \ref{IntroRelativeThm1}, we are able to show that the DeMarco--Mavraki conjecture holds for all fibered powers $(\Delta)^{k} \times X \subseteq (\bb{P}^1)^{2k} \times \mathrm{Poly}_{\mathrm{mc}}^d \times \mathrm{Poly}^d.$ We only handle the case of equal degrees $d_1 = d_2$ since the DeMarco--Mavraki conjecture was stated for polarized endomorphisms only. 

\begin{theorem} \label{IntroDeM1}
Let $X \subseteq \mathrm{Poly}_{\mathrm{mc}}^d \times \mathrm{Poly}^d$ be an irreducible subvariety defined over $\ovl{\bb{Q}}$. Assume that as $(f,g)$ varies over $X(\ovl{\bb{Q}})$, the map $f$ is not always $z^d$. Fix a positive integer $k$ and let 
$$\cal{X} = (\Delta)^k \times X \subseteq (\bb{P}^1)^{2k} \times X,$$
where we view $\cal{X}$ as a family over the base $X$. 
Then the DeMarco--Mavraki conjecture is true for $\cal{X}$.
\end{theorem}

\subsection{Strategy of the proof} \label{sec: Strategy}
\subsubsection{Rigidity of Julia Sets}
We first explain the proof of the implication (2) \( \Rightarrow \) (1) in Theorem~\ref{IntroRigidity1}.
We first reduce to the case where both $f,g$ have the same degree $d \geq 2$, which we will assume for the rest of this section. Let us sketch the proof when $k = \bb{C}$ is Archimedean. The equality of Julia sets $J_f = J_g$ implies their dynamical Green's 
functions
$G_f,G_g$ are  equal. Let $\phi_f,\phi_g: U \to D(0,1)$ be their associated Böttcher coordinates satisfying 
\begin{equation} \label{eq: IntroFunctional1}
\phi_f(f(z)) = \phi_f(z)^d, \quad \phi_g(g(z)) = \phi_g(z)^d,
\end{equation}
where $U$ is some neighbourhood of $\infty$. Then $G_f = G_g$ implies $\phi_f = c\phi_g$ for some explicit constant $c$, using the fact that $G_f = \log |\phi_f|$ and similarly for $g$. Baker--Eremenko \cite{BE87} and Schmidt--Steinmetz \cite{SS95} use the functional equations \eqref{eq: IntroFunctional1} to deduce that $f$ and $g$ must then be dynamically related.
\par 
When $k$ is non-Archimedean, the equality $J_f = J_g$ still implies that $G_f = G_g$. We still have the non-Archimedean Böttcher coordinates $\phi_f,\phi_g$ but it is no longer clear how to conclude that $\phi_f = c \phi_g$.  The main reason for this is that the norm function $| \cdot |: k \to \bb{R}$ contains less information for non-Archimedean fields than Archimedean fields because of the strong triangle inequality. As an example, if $\bb{L}$ denotes the completion of the field of Puiseux series $\bigcup_{n > 0} \bb{C}((t^{1/n}))$, then $| \cdot |$ maps $\bb{L}$ to $\exp(\bb{Q})$, which is much smaller than $\bb{R}$ and hence captures less information. This is the reason why $G_f$'s are all equal if $f$ has good reduction. 
\par 
We adopt a different approach to show that $f$ and $g$ must be dynamically related, exploiting the fact that $f$ has bad reduction which induces rich dynamics on $J_f$. We first illustrate the argument for $f,g$ of the form $z^2 + c$, where $|c| > 1$. We first pick a large $\alpha > 0$ so that $\{G_f < \alpha\} = \{G_g < \alpha\}$ is a disk $D(0,R)$ for some $R > 0$. Then 
$$\left\{G_f < \frac{\alpha}{2^n} \right\} = \left\{G_g < \frac{\alpha}{2^n} \right\} \implies f^{-n}(D(0,R)) = g^{-n}(D(0,R)).$$
Now it is easy to show that for any $\eps > 0$, we may find an $n$ large enough so that $f^{-n}(D(0,R))$ is the union of $2^n$ disjoint disks $\{D_1,\ldots,D_{2^n}\}$, each of radius $< \eps$. Hence the zeros of $f^n $ and $g^n$ can be paired up so that any pair $\{x,y\}$ satisfies $|x-y| < \eps$. Using Viète's formulas and taking $\eps \to 0$ allows us to deduce relations between the coefficients of $f$ and $g$ which lets us conclude. 
\par 
In general, the dynamics of $f$ will be much more complicated than the quadratic case. To help describe the dynamics, we construct a non-Archimedean analogue of the DeMarco--McMullen tree \cite{DM08}, which is a simplicial tree associated to a polynomial $f$ of degree $d \geq 2$ over $\bb{C}$. The set of ends of $T_f$ corresponds to the Julia set $\cal{J}_f$. For a tame polynomial \cite{Fab13a, Fab13b, Tru14, KN22} of degree $d \geq 2$ over a non-Archimedean field $k$, we associate to it a subset 
$$T_f = \bigcup_{t > 0} \partial \{G_f(z) < t\} \subseteq \bP_k^{1,\an},$$ 
whose closure in $\bP_{k}^{1,an}$ is exactly the convex hull of $\calJ_f\cup\{\infty\}.$

We endow it with a distinguished set of vertices and edges so that $f$ acts on $T_f$ as a simplicial self-map. The assumption that $f$ has bad reduction ensures that $T_f$ has infinitely many vertices and edges and is hence interesting. 
\par 
Given $f,g$ of degree $d \geq 2$ with $\mu_f = \mu_g$, one can show that $T_f = T_g$ as subsets of $\bb{P}_k^{1,\an}$, and with the same set of vertices and edges. However, $f$ and $g$ may act differently. We show that for any $\eps > 0$, we may find a finite subset $X \subseteq T_f$ such that 
$$f^{-1}(X) = g^{-1}(X) = Y\not=\emptyset$$
with every vertex of $Y$ representing a disk of radius $< \eps$ (Proposition \ref{prop: controldistance}). This allows us to replicate the argument used in the case of $z^2 + c$ and hence show that $f$ and $g$ must be dynamically related.

\subsubsection{Degenerations and Uniformity}
We now explain how to deduce Theorem \ref{thm: uniformprep} from Theorem \ref{IntroRigidity1}. By a specialization argument \cite[Section 10.2]{DKY21}, we may assume that $f,g$ are both defined over a number field $K$. Following Favre--Rivera-Letelier \cite{FRL06} (see also Fili \cite{Fil17}), for each place $\nu \in M_K$ we may define a local energy pairing 
$$\langle f, g \rangle_\nu = -\iint_{\bb{P}_\nu^{1,\an} \times \bb{P}_\nu^{1,\an}} \log |x-y|_\nu d(\mu_{f,\nu} - \mu_{g,\nu})(x) d(\mu_{f,\nu} - \mu_{g,\nu})(y)$$
where $\mu_{f,\nu},\mu_{g,\nu}$ are the equilibrium measures of $f,g$ over $\bb{C}_\nu$. We may then define a global adelic energy pairing by 
$$\langle f ,g \rangle = \sum_{\nu \in M_K} N_\nu \langle f, g \rangle_\nu$$
where $N_\nu = [K_\nu:\bb{Q}_\nu]/[K:\bb{Q}]$. 
By the work of DeMarco--Krieger--Ye \cite{DKY20, DKY21}, subsequently generalized by Gauthier \cite{Gau24} and Ang--Yap \cite{AY23}, proving Theorem \ref{thm: uniformprep} reduces to establishing uniform lower bounds on the global pairing $\langle f,g \rangle$. More precisely, one needs the existence of a uniform constant $c = c(d_1,d_2) > 0$, depending only on  $d_1$ and $d_2$, such that if $\Prep(f) \not = \Prep(g)$, then
\begin{equation} \label{eq: IntroUniform1}
\langle f ,g \rangle \geq c \max\{h(f)+h(g),1\}
\end{equation}
where $h$ is a Weil height on polynomials. 
\par 
To prove \eqref{eq: IntroUniform1}, we establish suitable uniform lower bounds on the local energy pairings 
$\langle f,g \rangle_\nu$ 
using a degeneration argument. This is the approach taken by DeMarco--Krieger--Ye \cite{DKY20} and Poineau \cite{Poi25} in the case of Lattes maps. We will sketch an example of such a uniform local lower bound and how to deduce it via degenerations. 
\par 
Let $V$ be a positive-dimensional subvariety of 
$\mathrm{Poly}_{\mathrm{mc}}^{d_1}$ 
defined over $\ovl{\bb{Q}}$. 
There exists a maximal Zariski open $U \subseteq V$ such that for all $f \in U(\ovl{\bb{Q}})$, one has  $\Aut(\calJ_f)=\bb{U}_m = \{\zeta^m = 1 \mid \zeta \in \bb{C}\}$.
Let $V' = V \setminus U$. 
One can verify that $V'$ is defined by the vanishing of finitely many 
polynomials $\{p_1,\ldots,p_{\ell}\}$ in the coefficients of $f$. 
\par 
Let $(k,| \cdot |)$ be a non-trivially valued field of characteristic $0$. If we write $f = \sum_{i=0}^{d_1} a_i z^i$ and $g = \sum_{i=0}^{d_2} b_i z^i$, we define $|f|=\max\{|a_i|\}$, $|(f,g)| = \max\{|a_i|,|b_j|,|b_{d_2}^{-1}|\}$ and 
$$\lambda_{V'}(f) = - \log \max \{|p_1(f)|,\ldots,|p_{\ell}(f)|\}.$$
The function $\lambda_{V'}$ measures how close $f$ is to the subvariety $V'$. In particular, we have $\lambda_{V'}(f) = + \infty$ if and only if $f \in V'$. We now state one of our uniform lower bounds.

\begin{theorem} \label{IntroUniform2}
Let \((k,|\cdot|)\) be a field of characteristic \(0\) equipped with a
nontrivial absolute value, and fix \(C_1,C_2>0\). Then there exist constants
\(C_3,C_4>0\) depending only on $k$, $d_1$ and $d_2$ such that the following holds:

Suppose that
\[
(f,g)\in (V\setminus V')(k)\times \operatorname{Poly}^{d_2}(k)
\]
satisfies
\[
\lambda_{V'}(f)\leq C_1\log |(f,g)|.
\]
If positive integers \(n_1,n_2\) exist with \(d_1^{n_1}=d_2^{n_2}\), assume
further that
\[
\log |f^{n_1}-\lambda g^{n_2}|
\geq -C_2\log |(f,g)|
\quad\text{for all } \lambda\in \mathbb{U}_m.
\]
Then
\[
\log |(f,g)|\geq C_3
\quad\Longrightarrow\quad
\langle f,g\rangle \geq C_4\log |(f,g)|.
\]
\end{theorem}

We also have the following variant of Theorem \ref{IntroUniform2} where the condition of $|(f,g)|$ being sufficiently large is replaced by $k$ having sufficiently large residue characteristic.

\begin{theorem} \label{IntroUniform3}
For every \(C_1,C_2>0\), there exist constants \(C_5,C_6>0\) depending only on $d_1$ and $d_2$ such that the following holds: Let \((k,|\cdot|)\) be a non-Archimedean field whose residue characteristic is at least \(C_5\), and let
\[
(f,g)\in (V\setminus V')\times \operatorname{Poly}^{d_2}(k)
\]
satisfy
\[
\lambda_{V',k}(f)\le C_1\log |(f,g)|.
\]
Suppose moreover that whenever there exist positive integers \(n_1,n_2\) with
\[
d_1^{n_1}=d_2^{n_2},
\]
the following additional condition is imposed:
\[
\log |f^{n_1}-\lambda g^{n_2}|
\ge -C_2\log |(f,g)|
\qquad\text{for every } \lambda\in \mathbb U_m.
\]
 Then
\[
\langle f,g\rangle \ge C_6\log |(f,g)|.
\]
\end{theorem}

Combining these two lower bounds gives us \eqref{eq: IntroUniform1}. 
\par 
As mentioned earlier, Theorem \ref{IntroUniform2} can be directly deduced from Theorem \ref{IntroRigidity1} using a degeneration argument. We follow the approach by Favre--Gong \cite{FG25} which combines rescaling limits with ultrafilters. The idea of rescaling limits in complex dynamics was first introduced by Kiwi \cite{Ki06}, where given a family of rational maps $(f_t)_{t \in \bb{D}^{*}}$ parametrized by the punctured disk $\bb{D}^{*}$, he constructs a rational map $f$ over the non-Archimedean field $\bb{C}((t))$ as a rescaled limit of the family. This has been expanded on by DeMarco--Faber \cite{DF14} and Favre \cite{Fav20}. 
\par 
In his thesis, Luo \cite{Luo21, Luo22}, using ultrafilters, generalizes Kiwi's construction to any sequence of degree $d$ rational maps $(f_n)$ over $\bb{C}$. In particular, the sequence $(f_n)$ need not be constrained in a one-parameter family and can vary in the entire moduli space $\Rat_d$. Favre--Gong then formalized Luo's construction using the theory of Berkovich spaces, which allows each $f_n$ to be defined over possibly distinct metrized fields $k_n$. The method of Favre--Gong has previously been applied by the second author to obtain uniformity results in both arithmetic dynamics \cite{Yap25} and Diophantine geometry with Looper \cite{LY26}.
\par 
We now sketch out the degeneration argument. Assume for contradiction that Theorem \ref{IntroUniform2} is false. We obtain a sequence 
$(f_n,g_n) \in (V \setminus V')(k) \times \mathrm{Poly}^{d_2}(k)$ such that 
$$\lambda_{V'}(f_n) \leq C_1 \log |(f_n,g_n)|,~ \langle f_n, g_n \rangle \leq \frac{1}{n} \log |(f_n,g_n)|;$$ 
$$\log |(f_n,g_n)| \to \infty,~\text{and}~ \log |f_n^{n_1} - \lambda g_n^{n_2}| \geq - C_2 \log |(f_n,g_n)| \text{ for all } \lambda \in \bb{U}_m.$$

We set $\eps_n = (\log |(f_n,g_n)|)^{-1} \to 0$. Fix a non-principal ultrafilter $\omega$. Then we may associate to $(f_n,g_n)$ an $\omega$-limit, which is a pair of degree $d$ polynomials $(f_{\omega},g_{\omega})$ defined over the metrized field 
$$\scr{H}(\omega) = \Bigl\{(z_n)\in k^{\bN}\ \bigm|\ \lim_{\omega}|z_n|^{\epsilon_n}<\infty\Bigr\}
\Big/
\Bigl\{(z_n)\in k^{\bN}\ \bigm|\ \lim_{\omega}|z_n|^{\epsilon_n}=0\Bigr\}.$$
The absolute value $| \cdot |_{\omega}$ on $\scr{H}(\omega)$ is given by $|(z_n)| = \lim_{\omega} |z_n|^{\eps_n}$ where $\lim_{\omega}$ denotes the $\omega$-limit \cite[Section 3]{FG25}. The field $\scr{H}(\omega)$ is algebraically closed, non-Archimedean and of residue characteristic $0$. By continuity, the pair $f_{\omega},g_{\omega}$ satisfies 
$$\langle f_{\omega}, g_{\omega}\rangle \leq \lim_{\omega} \frac{1}{n} \eps_n \log |(f_n,g_n)|  = 0$$
and furthermore, it is easy to reduce to the case where $f_{\omega}$ does not have potential good reduction. 
Since 
$$\lambda_{V'}(f_{\omega}) \leq \lim_{\omega} C_1 \eps_n \log  |(f_n,g_n)| = C_1 < +\infty,$$
we have $f_{\omega} \not \in V'(\scr{H}(\omega))$
, which implies that $\Aut(\calJ_{f_\omega})=\bb{U}_m$. 
Applying Theorem~\ref{IntroRigidity1}, one obtains $f_{\omega}^{n_1} = \lambda g_{\omega}^{n_2}$ for some $\lambda \in \bb{U}_m$.
But we also have 
$$\log |f_{\omega}^{n_1} - \lambda g_{\omega}^{n_2}| \geq - \lim_{\omega} C_2 \eps_n \log |(f_n,g_n)| = - C_2,$$
and hence $f_{\omega}^{n_1} \not = \lambda g_{\omega}^{n_2}$, 
giving us a contradiction. 
\par 
  Finally, we note that there is a theory of globally valued fields coming from the model theory point of view \cite{YH22,Sza23,YDHS24}. 
  This is related to the theory of adelic curves by Chen--Moriwaki \cite{CM19, CM21} and also to the theory of Yuan--Zhang \cite{YZ24}. Globally valued fields can be viewed as a global analogue of valued fields, where one deals with fields equipped with a global height function, such as the Weil height on $\ovl{\bb{Q}}$.
Using ultrafilters~\cite{Gol22} is a standard technique in model theory, allowing one to degenerate sequences of objects defined over globally metrized fields in a way analogous to degeneration over valued fields.
The interested reader may consult Hultberg's work \cite{Hul26} for applications to semiabelian varieties.  

\subsection{Further questions}
We now raise a few further questions. The first question is whether an extension of Theorem \ref{IntroRigidity1} to rational maps exists. More generally, we can ask for a non-Archimedean rigidity statement for a pair of rational maps $(f,g)$ and a curve $C \subseteq (\bb{P}^1)^2$. This will be addressed in a forthcoming work by Zhuchao Ji and the second author \cite{JY26}, with applications towards the uniform dynamical Manin--Mumford and Bogomolov conjecture on $(\bb{P}^1)^n$. 
\par 
Secondly, Theorem \ref{IntroRigidity1} only holds for residue characteristic $0$ and one may ask for analogues that hold for positive residue characteristic. Our proof relies crucially on residue characteristic being $0$ as we use the fact that $|n| = 1$ for all positive integers $n$. This leads us to formulate the following conjecture.

\begin{conjecture} 
Let $k$ be an algebraically closed and complete non-Archimedean field, possibly of positive characteristic. Let $\cal{J}$ be the Julia set of a tame polynomial of degree at least $2$ which does not have potential good reduction. Let $g$ be a polynomial of minimal degree whose Julia set coincides with $\cal{J}$. Then if $f$ is a polynomial of degree $d \geq 2$ with $\cal{J}_f = \cal{J}$, there exists an integer $n \geq 1$ and $\sigma \in \Aut(\cal{J})$ such that $f = \sigma \circ g^n$.
\end{conjecture}

When $\cal{J}$ is not the Julia set of a tame polynomial, there are counterexamples. Let $p\ge 3$ be a prime number. Consider $
f(z)=\frac{z^p - z}{p}$ and $g(z)=\frac{z^p - pz^2 - z}{p} \in \C_p[z].$
Then the Julia set of both maps equals $\Z_p$; see~\cite[Example 8.17]{Ben19} for an explanation in the case of $f$, and the argument for $g$ is analogous. However, there do not exist affine transformations $\sigma_1, \sigma_2$ such that
$f \circ \sigma_1 = \sigma_2 \circ g.$ Observe that $\Aut(\bb{Z}_p)$ is infinite, which suggests that it should play a similar role as the circle $S^1$ in the Archimedean case. We make a more ambitious conjecture that if $\Aut(\cal{J})$ is finite, then Conjecture \ref{IntroConjecture1} should be true.

\begin{conjecture} \label{IntroConjecture2}
Let $k$ be an algebraically closed and complete non-Archimedean field, possibly of positive characteristic. Let $\cal{J}$ be the Julia set of a polynomial of degree at least $2$ which does not have potential good reduction, such that $\Aut(\cal{J})$ is finite. Let $g$ be a polynomial of minimal degree whose Julia set coincides with $\cal{J}$. Then if $f$ is a polynomial of degree $d \geq 2$ with $\cal{J}_f = \cal{J}$, there exists an integer $n \geq 1$ and $\sigma \in \Aut(\cal{J})$ such that $f = \sigma \circ g^n$.
\end{conjecture}

Conjecture \ref{IntroConjecture2} naturally leads to the question of classifying Julia sets $\cal{J}$ with $|\Aut(\cal{J})| = + \infty$. We do not have a concrete conjecture for this question but we expect that over $\bb{C}_p$, such a Julia set should be similar to $\bb{Z}_p$ in nature.
\par 
Finally, it is interesting to ask whether one can deduce $\phi_f = c \phi_g$ directly from $G_f = G_g$, in analogy with the case over $\bb{C}$. At present, we do not know how to approach this problem.

\subsection*{Acknowledgements}
Both authors would like to thank Laura DeMarco, Charles Favre, Nuno Hultberg, Zhuchao Ji, and Myrto Mavraki for their thoughtful and valuable comments on this paper. We are especially grateful to Myrto Mavraki for her insightful suggestions highlighting the importance of rigidity of Julia sets in the non-Archimedean setting.
 
\section{Dynamics on the Berkovich projective line and potentials}
\subsection{Construction of the Berkovich projective line over a general Banach ring}We briefly recall the construction of the Berkovich projective line over a general Banach ring. We refer the reader to~\cite{LP24} for a detailed treatment.

A Banach ring $(B,\|\cdot\|)$ is, by convention, a commutative ring with unity endowed with a complete norm $
\|\cdot\| \colon B \to \R_+$ 
satisfying $\|a+b\| \le \|a\| + \|b\|$ and  $\|ab\| \le K \|a\| \|b\|$
for some constant $K>0$ and for all $a,b \in B$; see~\cite[A.1.2.1]{BGR84}.

The \emph{Berkovich spectrum} \( \mathcal{M}(B) \) is the set of all bounded multiplicative seminorms on \( B \), that is, maps
\( |\cdot| \colon B \to \R_{+} \) satisfying $
|a+b| \le |a| + |b|$, $|ab| = |a||b|$, and  $|\cdot| \le C\|\cdot\|$  for some $C>0$.
For any \( x \in \mathcal{M}(B) \) and \( b \in B \), we denote the value of \( b \) at \( x \) by \( |b|_x \) or \( |b(x)| \).

The space \( \mathcal{M}(B) \) is endowed with the weakest topology for which all functions of the form
$x \longmapsto |b(x)|$, $b \in B$,
are continuous. By \cite[Theorem~1.2.1]{Ber90}, the space \( \mathcal{M}(B) \) is compact.

To define the Berkovich projective line $\bP_B^{1,\an}$ over $B$, we first define the Berkovich affine line $\bA^{1,\an}_B$. It is the set of all multiplicative seminorms on $B[T]$
whose restriction to $B$ belongs to $\mathcal{M}(B)$.  
It is endowed with the weakest topology such that all functions $x \longmapsto |P(x)|$, for some $P \in B[T]$, are continuous.

For $x \in \bA^{1,\an}_B$, consider the prime ideal $\ker(x) := \{ P \in B[T] \mid |P(x)| = 0 \}.$
The \emph{complete residue field} $\sH(x)$ is defined as the completion of the fraction field of
$B[T]/\ker(x)$ with respect to the norm induced by $x$.

The Berkovich projective line $\bP^{1,\an}_B$ is obtained by gluing two copies $X_0$ and $X_1$ of
$\bA^{1,\an}_B$ along $\bA^{1,\an}_B \setminus \{0\}$ in the usual way.
This construction yields homogeneous coordinates $[z_0 : z_1]$ on $\bP^{1,\an}_B$.
For simplicity, we write $z = z_0$, $0 = [0:1]$, and $\infty = [1:0]$, so that
$\bP^{1,\an}_B = \bA^{1,\an}_B \sqcup \{\infty\}.$
By construction, the space $\bP^{1,\an}_B$ is compact.

There is a canonical continuous map
$\pi \colon \bP^{1,\an}_B \longrightarrow \mathcal{M}(B),$
sending a point $x \in \bP^{1,\an}_B$ to its restriction to $B$.
This map is proper, and for every $s \in \mathcal{M}(B)$, the fiber $\pi^{-1}(s)$ is canonically
identified with the Berkovich projective line $\bP^{1,\an}_{\sH(s)}$.

\subsection{Dynamics on $\bb{P}_k^{1,\an}$}
We briefly recall some basic notions concerning the dynamics of rational maps on the Berkovich projective line. For the case $k=\C$, see~\cite{Mi06}, and for the case $k\neq\C$, see~\cite{Jon15,Ben19}. 

\subsubsection{Berkovich projective line over a metrized field}Let $k$ be an algebraically closed complete metrized field, that is, a field endowed with a complete multiplicative norm $|\cdot|$.  

\smallskip

We say that $k$ is \emph{non-Archimedean} if its norm satisfies the strong triangle inequality $|x+y|\le \max\{|x|,|y|\}.$
By the Gelfand--Mazur theorem and Ostrowski’s theorem, if $k$ is Archimedean, then there exists
$0<\varepsilon\le 1$ such that $k$ is isometric to $\C_\varepsilon=(\C,|\cdot|^\varepsilon)$, where
$|\cdot|$ denotes the standard Euclidean norm on $\C$.  
In this case, $\bP_k^{1,\an}$ is homeomorphic to the Riemann sphere $\hat{\C}$.

If $k$ is non-Archimedean, then the Berkovich projective line $\bP_k^{1,\an}$ carries the structure
of a compact $\R$-tree: for any $x,y\in\bP_k^{1,\an}$, there exists a unique segment $[x,y]$
joining them. 

There are exactly four types of points on $\bP_k^{1,\mathrm{an}}$:

\begin{enumerate}
    \item type I: There exists a rational function $R$ such that $|R(x)|=0$. In this case, $x$ can be identified with a point of $\bP^1(k)$, and the associated
seminorm is given by the evaluation map $P\mapsto |P(x)|$;
\item type II: There exist $a\in k$ and $r\in |k^\times|$ such that for every polynomial $P(z)=\sum_{i=0}^{d}a_i(z-a)^i\in k[z]$, one has 
\[|P(x)|=\max_{z\in \bar D(a,r)} |P(z)|=\max_{0\le i\le d}|a_i|r^i,\]  where $\bar D(a,r)$ denotes the closed ball of radius $r$
centered at $a$;
\item type III: There exist $a\in k$ and $r>0$ and $r\notin |k^{\times}|$ such that for every polynomial $P(z)=\sum_{i=0}^{d}a_i(z-a)^i\in k[z]$, one has 
\[|P(x)|=\sup_{z\in \bar D(a,r)} |P(z)|=\max_{0\le i\le d}|a_i|r^i.\] 
\item type IV: There exists a sequence of decreasing closed balls $\bar{D}_{n}$  with $\bigcap \bar{D}_{n}=\emptyset$, such that \[|P(x)|=\inf_{n}\max_{z\in \bar{D}_{n}}|P(z)|\] for all $P\in k[z]$.
\end{enumerate}
The set of type~I and type~IV points forms the ends of the Berkovich tree.  
Type~II points correspond to branch points of the tree, while type~III points are regular points.


Moreover, if \( k \) is spherically complete, i.e., every decreasing sequence of closed balls has nonempty intersection, then \( \mathbb{P}^{1,\an}_k \) has no type~IV points.  
If \( |k^\times| = \mathbb{R}_+ \), then \( \mathbb{P}^{1,\an}_k \) does not contain type~III points.

\smallskip

We usually write $\zeta(a,r)$ for the type I, type~II or III point associated with the closed ball
$\bar D(a,r)$, and we denote the Gauss point $\zeta(0,1)$ by $x_g$. When $r=0$, $\zeta(a,r)$ is the type I point defined by  $a.$
\subsubsection{Dynamics of a rational map}
Recall that a rational map of degree $d\ge 2$ with coefficients in $k$ is given in homogeneous coordinates by
\[
f\colon [z_0:z_1]\longmapsto [P(z_0,z_1):Q(z_0,z_1)],
\]
where $P(z_0,z_1)=\sum_{i=0}^d a_i z_0^i z_1^{d-i}$, and $Q(z_0,z_1)=\sum_{i=0}^d b_i z_0^i z_1^{d-i}$
are homogeneous polynomials of degree $d$ in $k[z_0,z_1]$ without common factors.
We denote by $\Rat_d(k)$ the set of rational maps of degree $d$ over $k$.
It can be viewed as an affine subvariety of $\bP^{2d+1}$ defined by the nonvanishing locus of the normalized resultant
\[
|\Res(f)| \coloneqq \frac{|\Res(P,Q)|}{\max_{0\le i,j\le d}\{|a_i|,|b_j|\}^{2d}} \neq 0.
\]

A rational map $f$ extends naturally to the Berkovich projective line $\bP_k^{1,\an}$ by sending a point
$x$ in the hyperbolic space $\bH_k=\bP_k^{1,\an}\setminus\bP^{1}(k)$ to $f(x)$, defined by the rule
\[
|h(f(x))| = |(h\circ f)(x)| 
\quad \text{for all } h\in k(z).
\]

The induced endomorphism $f\colon \bP_k^{1,\an}\longrightarrow \bP_k^{1,\an}$
is continuous, finite, open, and surjective (see, for instance,~\cite[Proposition~4.3]{Jon15}).
One can define a local degree $\deg_x(f)\in\{1,\dots,d\}$ at each point $x\in\bP_k^{1,\an}$ such that
\begin{align}
\label{eqq: degree}
    \sum_{f(y)=x} \deg_y(f) = d
\quad \text{for all } x\in\bP_k^{1,\an}.
\end{align}

The exceptional set $\mcE$ of $f$ is the largest finite subset of $\bP^1(k)$ that is totally invariant under $f$.
It contains at most two points.
The Julia set \( \calJ_f \subset \bP_k^{1,\an} \) is the set of points \( x \) such that, for every open neighborhood \( U \) of \( x \),
$\bigcup_{n\ge 0} f^n(U) \supset \bP_k^{1,\an} \setminus \mc E .$
It is a closed and totally invariant subset of \( \bP_k^{1,\an} \).
Its complement
$\mcF(f) = \bP_k^{1,\an} \setminus \calJ_f$
is called the \emph{Fatou set}.
By \cite[Theorem 8.7]{Ben19}, the Julia set consists only of points of type~I and type~II.

\smallskip

A rational map \( f \) of degree \( d \ge 2 \) is said to have \emph{good reduction} if its Julia set reduces to the Gauss point \( x_g \). By~\cite[Theorem C]{FRL10}, this is equivalent to the condition that \( f \) has zero entropy, or equivalently that \( |\Res(f)| = 1 \). We say that \( f \) has \emph{potential good reduction} if it is conjugate to a rational map with good reduction. Otherwise,  we say \( f \) \emph{does not have potential good reduction}.
When \( k = \C \), we adopt the convention that \( f \) does not have good reduction.

\subsubsection{Tangent space and tangent map}

\label{sec: directiontagentmap}
\label{sec:directiontangentmap}

Let \(k\) be an algebraically closed non-Archimedean field. 

Recall that the Berkovich projective line \(\mathbb{P}_k^{1,\an}\) has a tree-like structure: for any two points \(x,y\in\mathbb{P}_k^{1,\an}\), there is a unique segment \([x,y]\) joining them. A \emph{direction} at a point \(x\in\mathbb{P}_k^{1,\an}\) is an equivalence class in
$\mathbb{P}_k^{1,\an}\setminus\{x\},$
where two points \(y,y'\neq x\) are equivalent if and only if
$(x,y]\cap(x,y']\neq\varnothing.$

The set of directions at \(x\) is called the \emph{tangent space} at \(x\) and is denoted by
$T_x\mathbb{P}_k^{1,\an}.$
For \(\vec{v}\in T_x\mathbb{P}_k^{1,\an}\), we denote by \(U(\vec{v})\) the corresponding connected component of
\(\mathbb{P}_k^{1,\an}\setminus\{x\}\). Thus, \(U(\vec{v})\) is a connected open subset of \(\mathbb{P}_k^{1,\an}\) whose boundary is the singleton \(\{x\}\).

\smallskip

Every nonconstant rational map \(f\in k(z)\) induces a map on tangent spaces,
\[
T_xf\colon T_x\mathbb{P}_k^{1,\an}
\longrightarrow
T_{f(x)}\mathbb{P}_k^{1,\an},
\]
called the \emph{tangent map} of \(f\) at \(x\). 
\subsection{Potential theory}
\subsubsection{Laplace operator}
We briefly recall some notions from potential theory on \( \mathbb{P}^{1,\an}_k \).
For the case \( k = \mathbb{C} \), see \cite{Ho94}, and for the non-Archimedean setting, see \cite{BR10,FJ04,Thu05}.

\smallskip

Let \( \sigma = \{P_i\}_{1 \le i \le l} \) be a finite collection of homogeneous polynomials of the same degree \( l(\sigma) \) with no common zero. Define
\[
\psi_{\sigma}
\coloneqq
\frac{1}{l(\sigma)} \log
\frac{\max\{|P_1|, \dots, |P_l|\}}{\max\{|z_0|, |z_1|\}^{l(\sigma)}} .
\]
Then \( \psi_{\sigma} \colon \mathbb{P}^{1,\mathrm{an}}_k \to \mathbb{R} \) is continuous.

Let \( X \) be the closure of \( \{\psi_{\sigma}\} \) in the space of continuous functions, endowed with the topology of uniform convergence. Denote by \( \mathcal{M}^{1}(\mathbb{P}^{1,\mathrm{an}}_k) \) the space of Radon probability measures on \( \mathbb{P}^{1,\mathrm{an}}_k \), endowed with the weakest topology such that, for every continuous function \( \varphi \) on \( \mathbb{P}^{1,\mathrm{an}}_k \), the map $\mu \mapsto \int \varphi\, d\mu$
is  a continuous function.

For each \( u \in X \), one can associate a signed measure \( \Delta u \) such that
$\mu_{\mathrm{can}} + \Delta u$
is a Radon probability measure, where \( \mu_{\mathrm{can}} \) denotes the Haar probability measure on the unit circle when \( k = \mathbb{C} \), and the Dirac measure supported at the Gauss point when \( k \) is non-Archimedean. This signed measure is called the \emph{Laplacian} of \( u \).

The Laplace operator defines a continuous map: 
$\mu_{\mathrm{can}}+\Delta \colon X \to \mathcal{M}^{1}(\mathbb{P}^{1,\mathrm{an}}_k)$, and satisfies
$\Delta u_1 = \Delta u_2$
if and only if
$u_1 - u_2$ is constant.

\subsubsection{Equilibrium measure}Let \( f=[P:Q] \) be a rational map of degree \( d \ge 2 \). 
Given any rational map, one can associate an equilibrium measure $\mu_f$ whose support is the Julia set. We briefly recall its construction below.

For any function $\varphi\colon \bP^{1,\an}_k \to \R$, define
\begin{equation}
\label{eq:pushforwardcontinuousfunction}
f_{*}\varphi(x)
\coloneqq
\sum_{y\in f^{-1}(x)} \deg_{y}(f)\,\varphi(y).
\end{equation}
Equation~\eqref{eqq: degree} implies that if $\varphi$ is continuous, then $f_{*}\varphi$ is also continuous, and its sup-norm is bounded by $d$ times the sup-norm of $\varphi$. Consequently, for any Radon measure $\rho$, the pullback of $\rho$ by $f$ can be defined dually by
$\int \varphi df^{*}\rho
=
\int f_{*}\varphi d\rho$ for any continuous function $\varphi.$

Set $\gamma_f
\coloneqq
\frac{1}{d}\log
\frac{\max\{|P|,|Q|\}}{\max\{|z_0|^d,|z_1|^d\}}.$
Iterating the identity for $f^*\mu_{\can}$ yields
\[
\frac{1}{d^n}(f^n)^*\mu_{\can}
=
\mu_{\can}
+
\Delta
\sum_{k=0}^{n-1}\frac{1}{d^k}\,\gamma_f \circ f^{k}.
\]
Since the series
$
\sum_{k=0}^{n-1}\frac{1}{d^k}\gamma_f \circ f^{k}$
converges uniformly in $X$ to
$\varphi_f(z)
\coloneqq
\sum_{k=0}^{\infty}\frac{1}{d^k}\,\gamma_f \circ f^{k},$
the limit of $
\frac{1}{d^n}(f^n)^*\mu_{\can}$
exists. We call this limit the \emph{equilibrium measure} of $f$, and denote it by $\mu_f$. Thus,
\[
\mu_f
=
\mu_{\can}
+
\Delta \varphi_f .
\]

This equilibrium measure satisfies the following equidistribution theorem. In the setting of complex polynomials, this result is due to Brolin~\cite{Bro65}, and for complex rational maps to Lyubich~\cite{Lyu83} and Freire-Lopes-Mañé~\cite{FLM83}. In the non-Archimedean setting, it was established by Favre--Rivera-Letelier~\cite{FRL04,FRL10}.

\begin{theorem}
\label{thm: equidistribution}
Let \( \rho \) be a Radon probability measure on \( \mathbb{P}^{1,\an}_{k} \). Then
\[
\lim_{n\to\infty}\frac{1}{d^n} (f^n)^*\rho
=
\mu_f ,
\]
if and only if \( \rho(\mcE)=0 \). Moreover,
$\frac{1}{d} f^{*}\mu_f
=
f_{*}\mu_f
=
\mu_f.$\end{theorem}
\subsubsection{Energy pairing}
\label{sec: energypairing}
Let $f$ and $g$ be two rational maps of degrees $d_1, d_2 \ge 2$, respectively.
We define an energy pairing between \( f \) and \( g \) using the mutual energy introduced by Favre--Rivera-Letelier~\cite{FRL06}.

We define the Hsia kernel $H \colon \mathbb{A}^{1,\mathrm{an}}_k \times \mathbb{A}^{1,\mathrm{an}}_k
\longrightarrow [-\infty,\infty)$
by
\[
H(x,y)=
\begin{cases}
|x-y|, & \text{if } k=\mathbb{C},\\[0.5em]
\sup\{x,y\}, & \text{if } k \text{ is non-Archimedean}.
\end{cases}
\]
In the non-Archimedean case, the quantity \( \sup\{x,y\} \) is defined as follows: 
Let \( x_0 \) be the unique point satisfying $
[x,\infty] \cap [y,\infty] = [x_0,\infty].$
Then we set $
\sup\{x,y\} \coloneqq \inf_{c \in k} |z-c|_{x_0}.$  In fact, when $x_0=\zeta(a,r)$ is of type I, II or III, then $\sup\{x,y\}$ is exactly the radius $r$.

 Since \( \varphi_f \) and \( \varphi_g \) are continuous functions,
by~\cite[Lemma~2.4 and Lemma~4.3]{FRL06} the function
\(\log H(x,y) \) belongs to
\( L^{1}\bigl(|\mu_f-\mu_g|\times|\mu_f-\mu_g|\bigr) \) on
\( \mathbb{P}^{1,\an}_k \times \mathbb{P}^{1,\an}_k \).
Therefore, one can define their \emph{energy pairing}, denoted by
\( \langle f,g\rangle \), by
\[
\langle f,g\rangle
=
-
\iint\log H(x,y)\, d(\mu_f-\mu_g)(x)\, d(\mu_f-\mu_g)(y).
\]
By~\cite[Proposition~2.6 and Proposition~4.5]{FRL06}, one has
\( \langle f,g\rangle \ge 0 \), and moreover
\( \langle f,g\rangle = 0 \) if and only if \( \mu_f=\mu_g \).

\begin{proposition}
\label{prop: computeE}
The energy pairing satisfies
\[
\langle f,g\rangle
=
-
\int (\varphi_f-\varphi_g)\, d(\mu_f-\mu_g).
\]
\end{proposition}

\begin{proof}
By Fubini’s theorem, we may write
\[
\langle f,g\rangle
=
-
\int \Bigl( \int \log H(x,y)\, d(\mu_f-\mu_g)(y) \Bigr)
\, d(\mu_f-\mu_g)(x).
\]
By~\cite[\S2.5 and \S4.5]{FRL06}, one has
\[
\int \log H(\cdot,y)\, d(\mu_f-\mu_g)
=
\varphi_f-\varphi_g + C
\]
for some constant \( C \in \mathbb{R} \).
Since \( \mu_f-\mu_g \) has total mass zero, the constant term integrates to zero,
and the desired identity follows.
\end{proof}

\section{Non-Archimedean degenerations by ultrafilters}



\subsection{Infinite product of Banach rings}
\subsubsection{Ultrafilters and Stone--Čech compactification}
We recall some basic notions concerning ultrafilters and the Stone--Čech compactification of the set \( \mathbb{N} \) of natural numbers. We refer to~\cite{CN74} for a detailed account.

\begin{definition}
A collection $\omega \subset \mathcal{P}(\bN)$ is called an \emph{ultrafilter} if it satisfies the following properties:
\begin{enumerate}
    \item $\emptyset \notin \omega$;
    \item if $E,F \in \omega$, then $E \cap F \in \omega$;
    \item if $E \in \omega$ and $E \subseteq F$, then $F \in \omega$;
    \item for every $E \subseteq \bN$, either $E \in \omega$ or $E^c = \bN \setminus E \in \omega$.
\end{enumerate}
\end{definition}

Indeed, an ultrafilter \( \omega \) may be viewed as a finitely additive
\( \{0,1\} \)-valued probability measure on \( \mathbb{N} \):
a subset has measure \(1\) if and only if it belongs to \( \omega \).
Such sets are called \emph{\(\omega\)-big}.
A property is said to hold \emph{\(\omega\)-almost surely} if it holds on an
\(\omega\)-big set.

For any \( n \in \mathbb{N} \), the family
\[
\omega_n = \{ F \subseteq \mathbb{N} \mid n \in F \}
\]
is an ultrafilter on \( \mathbb{N} \).
Such an ultrafilter is called \emph{principal}; it corresponds to the Dirac measure supported at \( n \).
An ultrafilter that is not principal is said to be \emph{non-principal}.
The existence of non-principal ultrafilters is guaranteed by Zorn’s lemma.

We denote by \( \beta\mathbb{N} \) the set of all ultrafilters on \( \mathbb{N} \).
Identifying each \( n \in \mathbb{N} \) with the corresponding principal ultrafilter
\( \omega_n \) yields an injection \( \mathbb{N} \hookrightarrow \beta\mathbb{N} \).

For each subset \( E \subseteq \mathbb{N} \), we define $
U_E = \{ \omega \in \beta\mathbb{N} \mid E \in \omega \}.$
For any \( E, F \subseteq \mathbb{N} \), one has \( U_E \cap U_F = U_{E \cap F} \).
Hence the family \( \{ U_E \}_{E \subseteq \mathbb{N}} \) forms a basis for a topology on
\( \beta\mathbb{N} \), whose open sets are precisely unions of such sets \( U_E \).

Equipped with this topology, the space \( \beta\mathbb{N} \), known as the
\emph{Stone--Čech compactification} of \( \mathbb{N} \), is compact and totally
disconnected, and contains \( \mathbb{N} \) as an open dense subset. 

\smallskip

The notion of an ultrafilter allows one to define generalized limits of sequences in any compact Hausdorff space \( X \). Namely, for any ultrafilter \( \omega \) on \( \mathbb{N} \) and any sequence \( (x_n) \subset X \), there exists a unique point \( x \in X \), called the \emph{\( \omega \)-limit} of \( (x_n) \), and denoted by
\[
x = \lim_{\omega} x_n,
\]
such that for every neighborhood \(V\) of \(x\), one has
\(x_n \in V\) \(\omega\)-almost surely.

Note that if \( \omega = \omega_m \) is the principal ultrafilter associated with
\( m \in \mathbb{N} \), then $
\lim_{\omega} x_n = x_m .$
Thus, in most situations of interest, we shall restrict attention to non-principal
ultrafilters.
\subsubsection{Product of Banach rings}
Let \( (k_n) \) be a sequence of complete metrized fields which are algebraically closed, and denote by
\( |\cdot|_n \) the given absolute value on \( k_n \).
A sequence of positive real numbers \( (\epsilon_n) \) is said to be
\emph{adapted} to \( (k_n) \) if, for each \( n \), one has
\( \epsilon_n \in (0,1] \) when \( k_n \) is Archimedean and
\( \epsilon_n > 0 \) when \( k_n \) is non-Archimedean.
Then, for every \( n \in \mathbb{N} \), the function
\( |\cdot|_n^{\epsilon_n} \colon k_n \to \mathbb{R}_+ \)
again defines a complete multiplicative norm on \( k_n \).

\smallskip

Let $\epsilon = (\epsilon_n)$ be a sequence adapted to $(k_n)$. We define the product Banach ring associated with $\epsilon$ as: 
\[
A^\epsilon
= \left\{ (z_n) \in \prod_{n\in\mathbb{N}} k_n \,\middle|\,
\sup_{n\in\mathbb{N}} |z_n|_n^{\epsilon_n} < +\infty \right\}.
\]
Then $A^\epsilon$, equipped with the norm $\|(z_n)\| := \sup_{n\in\mathbb{N}} |z_n|_n^{\epsilon_n},$
is a Banach ring.

By \cite[Proposition 1.2.3]{Ber90}, the Berkovich spectrum $\mathcal{M}(A^\epsilon)$ is homeomorphic to the Stone--Čech compactification $\beta\mathbb{N}$.
The homeomorphism is given by the map $\beta\mathbb{N} \longrightarrow \mathcal{M}(A^\epsilon),$
which sends an ultrafilter $\omega$ to the multiplicative seminorm $|\cdot|_\omega$ defined by
$|(z_n)|_\omega = \lim_{\omega} |z_n|_n^{\epsilon_n}.$

Note that the kernel of $|\cdot|_\omega$:
$
\ker(\omega) \coloneqq
\left\{ (z_n) \in \prod k_n \,\middle|\,
\lim_{\omega} |z_n|_n^{\epsilon_n} = 0 \right\}$
is a maximal closed ideal of $A^\epsilon$.
Therefore, the residue field of $A^\epsilon$ at $\omega$ is
$\sH(\omega) = A^\epsilon / \ker(\omega).$
Note that for any principal ultrafilter $\omega_n$, the residue field $\sH(\omega_n)$ is canonically identified with $k_n$ endowed with the rescaled norm $|\cdot|_n^{\epsilon_n}$.

When $\omega$ is a non-principal ultrafilter, the following proposition summarizes some main properties of the residue field $\sH(\omega)$.  

\begin{proposition}
\label{prop: prop-residuefield}
Fix a non-principal ultrafilter $\omega$. Then the following statements hold:
\begin{enumerate}
\item $\sH(\omega) = \mathbb{C}$ if and only if $k_n = \mathbb{C}$ $\omega$-almost surely and $\lim_{\omega} \epsilon_n > 0$;
\item the value group of $\sH(\omega)$ is $\R_+;$
\item $\sH(\omega)$ is algebraically closed and spherically complete.
\end{enumerate}
\end{proposition}
\begin{remark}
When $k_n=\mathbb{C}$ and $\epsilon_n=1$, we have $\sH(\omega)=\mathbb{C}$. 
By contrast, if $k_n=\mathbb{C}_p$ and $\epsilon_n=1$, then $\sH(\omega)\neq \mathbb{C}_p$, 
since $\sH(\omega)$ is spherically complete whereas $\mathbb{C}_p$ is not.
\end{remark}

\begin{proof}
If $k_n = \mathbb{C}$ $\omega$-almost surely, the assertions follow from \cite[Theorems~4.9 and~4.11]{FG25}.  
We now assume that $k_n$ is non-Archimedean $\omega$-almost surely.  

In this case, $\sH(\omega)$ is non-Archimedean, since for every $m \in \mathbb{N}$ one has
$|m|_\omega = \lim_{\omega} |m|_n^{\epsilon_n} \leq 1.$
The fact that $\sH(\omega)$ is algebraically closed follows from \cite[Theorem~4.1]{Yap25}. The spherical completeness follows from \cite[Lemma 1.14]{KW14}.

To prove that $|\sH(\omega)| = \mathbb{R}_+$, note first that $|k_n|_n$ is dense in $\mathbb{R}_+$ since $k_n$ is algebraically closed.  
Hence, for any $r \in \mathbb{R}_+$, one can choose $z_n \in k_n$ such that $\bigl||z_n|_n^{\epsilon_n} - r\bigr| < \frac{1}{n}.$
The sequence $(z_n)$ defines an element of $\sH(\omega)$, and we have $
|(z_n)|_\omega = \lim_{\omega} |z_n|_n^{\epsilon_n} = r.$
\end{proof}

\subsection{Continuity of dynamics on the hybrid space}
\subsubsection{Hybrid spaces}
Hybrid spaces provide a framework that interpolates between complex and non-Archimedean geometries.
The origin can be traced back to the compactification of $\mathrm{SL}_2$-character varieties by Morgan–Shalen~\cite{MS85}, and to later constructions using $\mathbb{R}$-trees in the work of DeMarco--McMullen~\cite{DM08}. The notion of hybrid space  was formalized by Berkovich~\cite{Ber09}, where the choice of Banach rings introduces substantial flexibility.

Hybrid spaces provide a natural framework for studying degenerations, particularly of measures. In holomorphic dynamics, this approach was introduced by Favre~\cite{Fav20}, building on a work of Boucksom–Jonsson~\cite{BJ17}. He showed that the equilibrium measures associated with a meromorphic family converge, in a suitable sense, to a non-Archimedean one. This result was subsequently extended by Poineau~\cite{Poi26} to more general Banach rings, and analogous statements were obtained by Favre–Gong~\cite{FG25} in the setting of product Banach rings arising from rescaled complex fields.

The hybrid space approach has since found applications in a variety of contexts, including character varieties~\cite{DF19,DM24}, automorphisms of algebraic varieties~\cite{Iro24,Iro25}, arithmetic dynamics~\cite{Yap25} and Diophantine geometry~\cite{LY26}.

\smallskip

We now construct hybrid spaces associated with product Banach rings over various base fields.

Let $\mathbb{P}^{1,\an}_{A^\epsilon}$ be the Berkovich projective line over the product Banach ring $A^\epsilon$.
It is a compact Hausdorff space and
there exists a canonical continuous projection
\[
\pi \colon \mathbb{P}^{1,\an}_{A^\epsilon} \longrightarrow \mathcal{M}(A^\epsilon) \simeq \beta\mathbb{N}.
\]
For any $\omega \in \beta\mathbb{N}$, the fiber $\pi^{-1}(\omega)$ is canonically identified with
$\mathbb{P}^{1,\an}_{\sH(\omega)}$.
In particular, for a principal ultrafilter $n \in \mathbb{N}$, the fiber $\pi^{-1}(n)$ is homeomorphic to
$\mathbb{P}^{1,\an}_{k_n}$. Thus, $\mathbb{P}^{1,\an}_{A^\epsilon}$ can be regarded as a fibration whose fibers are the spaces $\mathbb{P}^{1,\an}_{k_n}$ and their ultrafilter limits.


\subsubsection{Sequential hybridation}
A rational map $f\colon\bP^1_{A^\epsilon}\to \bP^1_{A^\epsilon}$ defined over $A^\epsilon$ is given in homogeneous coordinates by a pair of homogeneous polynomials $P,Q$ of degree $d$ whose resultant is a unit in $A^\epsilon$. When $A^\epsilon$ is a Banach ring, 
then $f$ induces a continuous map $f\colon\bP^{1,\an}_{A^\epsilon}\to \bP^{1,\an}_{A^\epsilon}$ over $\mathcal{M}(A^\epsilon)\cong\beta\bN$, and for any $\omega\in\beta\bN$, we denote by $f_\omega$
the restriction of $f$ on the fiber $\pi^{-1}(\omega)$. It is a rational map of degree $d$ defined over $\sH(\omega)$.

\smallskip 

We now provide a canonical way to encode a sequence of rational maps
\[ (f_n) \in \prod_{n\in\mathbb{N}} \mathrm{Rat}_d(k_n) \]
with bad reduction into a single dynamical system over the hybrid space.

To this end, we first associate to \( (f_n) \) the \emph{resultant sequence}
\( \epsilon(f) = (\epsilon(f_n)) \), defined by
\[
\epsilon(f_n)
=
\bigl(-\log(|\Res(f_n)|/C_d)\bigr)^{-1},
\]
where \( C_d = e \, \sup_{\mathrm{Rat}_d(\mathbb{C})} |\Res| \) if \( k_n=\mathbb{C} \),
and \( C_d = 1 \) otherwise.
This quantity is finite, since \( |\Res| \) is continuous on
\( \mathrm{Rat}_d(\mathbb{C}) \) and satisfies
\(
\lim_{f\to\partial \mathrm{Rat}_d(\mathbb{C})} |\Res(f)| = 0 .
\)

By construction, the resultant sequence \( \epsilon(f) \) is adapted to
\( (k_n) \).

\begin{theorem}
\label{thm: sequentialhybridation}
Let \( (f_n) \in \prod_{n\in\mathbb{N}} \mathrm{Rat}_d(k_n) \) be a sequence of rational maps with bad reduction, and let
\( \epsilon = (\epsilon_n) \) be a sequence adapted to \( (k_n) \).
Assume that there exists a constant \( C>0 \), independent of \( n \in \mathbb{N} \), such that
\( \epsilon_n \le C\, \epsilon(f_n) \) for all \( n \).
Then the following statements hold:
\begin{enumerate}
\item
There exists a rational map $f \in \Rat_d(A^\epsilon)$ inducing a continuous self-map
\[f \colon \mathbb{P}^{1,\an}_{A^\epsilon} \longrightarrow \mathbb{P}^{1,\an}_{A^\epsilon}\]
such that $\pi \circ f = \pi$, and for every $n \in \mathbb{N}$ one has
$f|_{\pi^{-1}(n)} = f_n$
under the canonical identification
$\pi^{-1}(n) \simeq \mathbb{P}^{1,\an}_{k_n}$;

\item
For  a fixed ultrafilter $\omega$, the following holds:
\begin{enumerate}
\item
If $\lim_{\omega} \frac{\epsilon_n}{\epsilon(f_n)} = 0,$
then $f_\omega$ has good reduction;

\item
If $\lim_{\omega} \frac{\epsilon_n}{\epsilon(f_n)} \neq 0,$
then $f_\omega$ does not have good reduction.
\end{enumerate}
\end{enumerate}
\end{theorem}

\begin{remark}
    If $k_n=\C$ and we suppose $|\Res(f_n)|=\sup_{M\in\PGL_{2}(\C)}|\Res(M_n\cdot f_n\cdot M_n^{-1})|$, then we get the sequential hybridation in \cite[Theorem 1.1]{FG25}.
\end{remark}
\begin{remark}
Similar results were obtained by Luo~\cite[Remark after Theorem 1.2]{Luo21} for renormalized hyperbolic spaces, and by Looper–Yap~\cite[\S4.1]{LY26} for degenerations of abelian varieties. 
\end{remark}
\begin{proof}  
We write $f_{n}=[P_{n}\colon Q_{n}]$, 
and normalize the coefficients of $P_n$ and $Q_n$ so that the maximum modulus of their coefficients is equal to $1$.   
We may suppose $C>1.$ Note that 
\begin{align*}
|\Res(P_n,Q_n)|_n^{\epsilon_n}
\in \left[e^{-(C+C|\log C_d|)},e^{(C+C|\log C_d|)}\right]    
\end{align*}
It follows that 
$P=(P_n)$ and $Q=(Q_n)$ are homogeneous polynomials with coefficients in $A^{\epsilon}$
and $\Res(P,Q)$ is a unit.
Consequently, 
$f=(f_n)$ defines an endomorphism of degree $d$ of $\bP^1_{A^\epsilon}$.
By construction, (1) holds.

If $\sH(\omega)=\C$, then Proposition~\ref{prop: prop-residuefield}(1) implies that
$\lim_{\omega}\epsilon_n>0.$
Hence only case~(b) can occur, and by convention $f_\omega$ does not have good reduction. Now assume that $\sH(\omega)$ is non-Archimedean. Note that
\[
|\Res(f_\omega)|
= \lim_{\omega} |\Res(f_n)|_n^{\epsilon_n}
= 1
\quad \Longleftrightarrow \quad
\lim_{\omega}\frac{\epsilon_n}{\epsilon(f_n)}=0.
\]
Therefore, $f_\omega$ has good reduction if and only if
$\lim_{\omega}\frac{\epsilon_n}{\epsilon(f_n)}=0.$
\end{proof}

\subsubsection{Continuity of energy pairings}
Note that for any \( f \in \mathrm{Rat}_d(A^\epsilon) \) and any \( \omega \in \beta\mathbb{N} \),
the equilibrium measure \( \mu_{f_\omega} \) can be regarded as a Radon probability measure on
\( \mathbb{P}^{1,\an}_{A^\epsilon} \) supported on \( \pi^{-1}(\omega) \). One can show that the equilibrium measure depends continuously on
\( \omega \in \beta\mathbb{N} \).
The proof follows the same strategy as in~\cite[Theorem~6.8]{FG25} in the case \( k_n=\mathbb{C} \).

\begin{theorem}
\label{thm: cvrgmeasure}
Let $f \in \mathrm{Rat}_d(A^\epsilon)$ for some $d \ge 2$.  
The family of equilibrium measures $\bigl(\mu_{f_\omega}\bigr)_{\omega \in \beta\mathbb{N}}$
defines a continuous family of positive measures on
$\mathbb{P}^{1,\an}_{A^\epsilon}$; that is, for every continuous function
$\varphi$ on $\mathbb{P}^{1,\an}_{A^\epsilon}$, the map
\[
\omega \longmapsto \int \varphi \, d\mu_{f_\omega}
\]
is continuous on $\beta\mathbb{N}$.
\end{theorem}

The continuity of equilibrium measures  implies the continuity of the rescaled energy pairings.

\begin{theorem}
\label{thm: cts-energy}

Let $d_1,d_2\ge2$ be two integers. Let \( f\in\Rat_{d_1}(A^\eps)\) and $g\in\Rat_{d_2}(A^\eps)$. For each \( \omega \in \beta\mathbb{N} \), define
\[
\epsilon(\omega)=
\begin{cases}
\epsilon_n, & \text{if } \omega \text{ is the principal ultrafilter } \omega_n,\\
1, & \text{if } \omega \text{ is non-principal}.
\end{cases}
\]
Then the rescaled energy pairings
\[
\epsilon(\omega)\,\langle f_\omega, g_\omega\rangle
\]
depend continuously on \( \omega \in \beta\mathbb{N} \).
\end{theorem}

\begin{proof}
Let \( f = [P:Q] \) and \( g = [U:V] \) be homogeneous representations of \( f \) and
\( g \), respectively. Define $
\gamma_f
\coloneqq
\frac{1}{d_1}\log
\frac{\max\{|P|,|Q|\}}{\max\{|z_0|^{d_1},|z_1|^{d_1}\}},$ and $\gamma_g
\coloneqq
\frac{1}{d_2}\log
\frac{\max\{|U|,|V|\}}{\max\{|z_0|^{d_2},|z_1|^{d_2}\}} .$
Associated to these functions, set
\[
\varphi_f
\coloneqq
\sum_{k=0}^{\infty} \frac{1}{d_1^{k}} \, \gamma_f \circ f^{k},
\qquad
\varphi_g
\coloneqq
\sum_{k=0}^{\infty} \frac{1}{d_2^{k}} \, \gamma_g \circ g^{k}.
\]
With these definitions, \( \varphi_f \) and \( \varphi_g \) are continuous functions on
\( \mathbb{P}^{1,\an}_{A^\epsilon} \), and they satisfy
\[
\mu_{f_\omega}
=
\delta_{x_g} + \frac{1}{\epsilon(\omega)} \,
\Delta \varphi_f \big|_{\pi^{-1}(\omega)},
\qquad
\mu_{g_\omega}
=
\delta_{x_g} + \frac{1}{\epsilon(\omega)} \,
\Delta \varphi_g \big|_{\pi^{-1}(\omega)} .
\]
Combining these identities with Proposition~\ref{prop: computeE}, we obtain
\[
\epsilon(\omega)\,\langle f_\omega, g_\omega \rangle
=
-
\int_{\pi^{-1}(\omega)}
(\varphi_f-\varphi_g)\, d(\mu_{f_\omega}-\mu_{g_\omega}) .
\]

On the other hand, by Theorem~\ref{thm: cvrgmeasure}, the families of measures
\( (\mu_{f_\omega}) \) and \( (\mu_{g_\omega}) \)
depend continuously on \( \omega \in \beta\mathbb{N} \).
Consequently, the quantity
\( \epsilon(\omega)\,\langle f_\omega, g_\omega \rangle \)
varies continuously with \( \omega \).\end{proof}

\section{Polynomial Trees}
Let $k$ be an algebraically closed, complete non-Archimedean field, possibly of positive characteristic. Given a polynomial $f$ of degree $d\ge2$, we will construct a subset $T_f \subseteq \bb{P}_k^{1,\an}$, inspired by DeMarco--McMullen's construction \cite{DM08}. 

When $f$ is a tame polynomial (e.g. when the residue characteristic of $k$ is $0$ or at least $d+1$), we will endow $T_f$ with a distinguished set of vertices so that $f$ acts on $T_f$ simplicially.


\subsection{Polynomial Dynamics}

\subsubsection{Image of disks}
\label{sec: imageofdisks}
We recall some basic facts concerning the images of disks under polynomial maps and refer the reader to \cite{Ben19} for a detailed discussion.

For $a \in k$ and $r > 0$, we define the open and closed disks in $\mathbb{P}_k^{1,\mathrm{an}}$ as $$D(a,r) := \{ x \in \mathbb{P}_k^{1,\mathrm{an}} \mid |z-a|_x < r \}
\quad \text{and} \quad
\bar{D}(a,r) := \{ x \in \mathbb{P}_k^{1,\mathrm{an}} \mid |z-a|_x \le r \}.
$$
If $r \in |k^\times|$, then $D(a,r)$ is said to be a rational open disk and $\bar{D}(a,r)$ is a rational closed disk.  
If $r \notin |k^\times|$, we have $D(a,r)=\bar{D}(a,r)$, and such a disk is called an irrational disk.

\smallskip

Let $D=D(a,r)$ or $\bar{D}(a,r)$ be a disk. We recall the notion  of \emph{Weierstrass degree} of $f$ on $D$:  
Write the expansion of $f$ around $a$ as
$f(z)=\sum_{i=0}^d a_i (z-a)^i .$
If $D$ is a rational closed disk (resp. a rational open or irrational disk), the Weierstrass degree of $f$ on $D$, denoted by $\mathrm{wdeg}_D f$, is the largest (resp. the smallest) integer
$0 \le \delta \le d$ such that
\[
|a_\delta| r^\delta = \max_{0 \le i \le d} \{ |a_i| r^i \}.
\]

\smallskip

By \cite[Corollary 3.18]{Ben19}, the image of a disk under $f$ is again a disk. More precisely, we have
\begin{align*}
    f\bigl(\bar{D}(a,r)\bigr)
=
\bar{D}\left(f(a), \max_{1 \le i \le d} \{ |a_i| r^i \} \right),\\
f\bigl(D(a,r)\bigr)
=
D\left(f(a), \max_{1 \le i \le d} \{ |a_i| r^i \} \right).
\end{align*}

In particular, $f$ maps a disk to a disk of the same type: it is open, rational closed, or irrational if and only if its image is of the same type.  
Moreover, by \cite[Theorem 3.15]{Ben19}, $f$ sends $D$ onto $f(D)$ with degree $\mathrm{wdeg}_D f$, counted with multiplicities.

\smallskip

To each disk \( D = D(a,r) \) or \( \bar{D}(a,r) \), one associates a type~II or type~III point
\( x = \zeta(a,r) \), according to whether \( r \in |k^\times| \) or not. By \cite[Proposition 7.6]{Ben19}, \( f \) sends this
type~II or type~III point to the point associated with the image disk, namely,
\begin{align}
\label{eqq: imagedisk}
    f(x)=\zeta\left(f(a),\max_{1\le i \le d}\{|a_i|r^i\}\right).
\end{align}
If \( D \) is open or irrational, then \( x \) is the unique boundary point of \( D \). 
If \( D \) is closed or irrational, the local degree \( \deg_x f \) coincides with
the Weierstrass degree of \( f \) on \( \bar{D}(a,r) \), see \cite[\S 7.4]{Ben19}.

\smallskip

Let $D=D(a,r)$ or $\bar{D}(a,r)$ be a disk. Then $f^{-1}(D)$ is a disjoint union of disks.  
More precisely, we have the following proposition: 
\begin{proposition}
\label{prop: preimage-disk}
    Let $x_i=\zeta(a_i,r_i)$ be the  preimages of the point $\zeta(a,r)$.

\begin{enumerate}
\item If $D=D(a,r)$ is open, let $D_{i,j}$ denote the connected components of $f^{-1}(D)$ with boundary point  $x_i=\zeta(a_i,r_i)$.  
Then, $D_{i,j}$ are disjoint disks and 
\begin{align*}
    f^{-1}(D)=\bigsqcup_{i,j} D_{i,j},\qquad \text{with } \sum_j \mathrm{wdeg}_{D_{i,j}} f = \deg_{x_i} f;
\end{align*}

    \item  If $D$ is a rational closed or an irrational disk  $\bar{D}(a,r)$, then
\[
f^{-1}(D)=\bigsqcup_i \bar{D}(a_i,r_i),
\qquad \text{with }  \mathrm{wdeg}_{\bar{D}(a_i,r_i)} f = \deg_{x_i} f .\]
\end{enumerate}
\end{proposition}
\begin{lemma}
\label{lem: increasedisk}
    If $D_1\subsetneq D_2$ are two rational closed disks or two irrational disks, then $f(D_1)\subsetneq f(D_2)$.
\end{lemma}
\begin{proof}
Write $D_i=\bar{D}(a,r_i)$ with $0<r_1<r_2,$ and $f(z)=\sum_{i=0}^d a_i (z-a)^i .$ Since $|k|$ is dense in $\R_+$, \[f(D_1)=f(\bar{D}(a,r_1))=D(f(a),\max_{1\le i\le d}|a_i|r_1^i)\subsetneq D(f(a),\max_{1\le i\le d}|a_i|r_2^i)=f(D_2),\]
which completes the proof.
\end{proof}
\begin{proof}[Proof of Proposition~\ref{prop: preimage-disk}]
(1)
Let $\vec v$ denote the tangent direction at $x$ represented by the disk $D$.  
For each $x_i$, let $\vec v_{i,j}$ be the tangent directions at $x_i$ such that $
T_{x_i}f(\vec v_{i,j})=\vec v,$
and let $D_{i,j}$ be the disk associated with $\vec v_{i,j}$.  
By \cite[Proposition~9.41]{BR10}, we have $f(D_{i,j})=D.$
Since $\partial D_{i,j}\subset f^{-1}(\partial D)$, it follows that $D_{i,j}$ is a connected component of $f^{-1}(D)$.  
Therefore, $
f^{-1}(D)=\bigsqcup_{i,j} D_{i,j}.$

(2)
Since $f(x_i)=x$, we have
$f\bigl(\bar D(a_i,r_i)\bigr)=\bar D(a,r).$
Hence,  $f^{-1}(D)\supset \bigsqcup_i \bar D(a_i,r_i).$
By Lemma~\ref{lem: increasedisk}, since the disks $\bar D(a_i,r_i)$ have the same image $\bar D(a,r)$ under $f$, they are pairwise disjoint.

Note that
$\deg_{x_i} f=\mathrm{wdeg}_{\bar D(a_i,r_i)} f,$
and $\sum_i \deg_{x_i} f = d.$ Since any point admits $d$ preimages counted with multiplicity, we conclude
$f^{-1}(D)=\bigsqcup_i \bar D(a_i,r_i).$
\end{proof}

\subsubsection{Dynamical Green's function}
\label{sec: greenfct}
We refer the reader to \cite[\S 2.3]{FG22} for a detailed discussion of the dynamical Green's function.  
Here we briefly recall some basic notions.

\smallskip

Let \( f(z)=a_d z^d+\cdots+a_0 \in k[z] \) be a polynomial of degree \( d\ge 2 \) with \( a_d\neq 0 \). It induces a continuous self-map $f:\mathbb A_k^{1,\mathrm{an}}\longrightarrow \mathbb A_k^{1,\mathrm{an}}$
on the Berkovich analytification of the affine line.

We define the \emph{filled-in Julia set} of $f$ as  
\[\mathcal{K}_f = \left\{ z \in \mathbb{A}^{1,\mathrm{an}}_k \,\middle|\, |f^n(z)| \text{ remains bounded as } n \to \infty \right\}.
\]
Set
\begin{equation}
\label{eqq: goodbound}
C_f = \max_{0\le i< d} \Bigl\{ 1,\;
\Bigl( \tfrac{|a_i|}{|a_d|} \Bigr)^{\!\frac{1}{d-i}},\;
\Bigl( \tfrac{2}{|a_d|} \Bigr)^{\!\frac{1}{d-1}} \Bigr\}.
\end{equation}
For any \( |z| > C_f \), one has
\( |a_d|\,|z|^{d} > |a_i|\,|z|^{i} \) for all \( i < d \).
By the strong triangle inequality, this implies
$|f(z)| = |a_d z^{d}| > 2|z|.$
Iterating this estimate yields
\begin{align}
\label{eqq: greencompute}
   |f^{n}(z)|
= |a_d|^{1 + d + \cdots + d^{n-1}} |z|^{d^{n}}>2^n|z|
\longrightarrow \infty . 
\end{align}
It follows that the \emph{basin of attraction of infinity}, defined as  
\[
\Omega_f = \left\{ z \in \mathbb{A}^{1,\mathrm{an}}_k \,\middle|\, |f^n(z)| \to \infty \text{ as } n \to \infty \right\},
\]
is an open set whose complement is precisely \( \cal{K}_f\). Note that all periodic points in $k$ are contained in $\cal{K}_f$. Consequently, \(\cal{K}_f \) is a non-empty compact subset of \( \mathbb{A}^{1,\mathrm{an}}_k \) contained in $\{|z|\le C_f\}$. Both \( \cal{K}_f \) and \( \Omega_f \) are totally invariant under \( f \), meaning \(f^{-1}(\cal{K}_f) = f(\cal{K}_f) = \cal{K}_f \) and likewise for \( \Omega_f \).
Moreover, the Julia set $\mathcal{J}_f$ is the boundary of $\mathcal{K}_f$.

\smallskip

Observe that the function $\frac{1}{d}\log^{+}|f(z)| - \log^{+}|z|$ extends to a continuous function on \( \mathbb{P}^{1,\an}_k \).
Consequently, there exists a constant \( C \ge 0 \) such that $
\left| \frac{1}{d}\log^{+}|f(z)| - \log^{+}|z| \right|
\le C$ on $\mathbb{P}^{1,\an}_k.$
It follows that the sequence of functions
$\frac{1}{d^{n}} \log^{+} |f^{n}(z)|$
converges uniformly on \( \mathbb{A}^{1,\an}_k \) to a continuous function
\[G_f(z)\coloneqq \lim_{n\to\infty}\frac{1}{d^n}\log^+|f^n(z)|,\] called the \emph{Green function} of \( f \).

We summarize some basic properties of the Green function in the following proposition.

\begin{proposition}\label{prop: basicgreen}
Let $f(z)=a_d z^d+\cdots$ be a polynomial of degree $d\ge2$. Then:
\begin{enumerate}
\item $G_f \circ f = d\, G_f$;
\item if $|z|>C_f$, one has
$G_f(z)=\frac{\log |a_d|}{d-1}+\log |z|;$
\item $G_f(z)=0$ if and only if $z$ belongs to the filled-in Julia set $\cal{K}_f$;
\item for any affine map $L\in \mathrm{Aff}(k)$, $
G_f(L(z)) = G_{L^{-1}\circ f \circ L}(z).$
\end{enumerate}
\end{proposition}

\begin{proof}
(1) By definition of the Green function,
\[
G_f(f(z))
= \lim_{n\to\infty}\frac{1}{d^n}\log^+|f^{n+1}(z)|
= d \lim_{n\to\infty}\frac{1}{d^{n+1}}\log^+|f^{n+1}(z)|
= d\, G_f(z).
\]

(2) This follows directly from \eqref{eqq: greencompute}.

(3) If $z\in \cal{K}_f$, then the sequence $|f^n(z)|$ is bounded, and hence $
G_f(z)=0.$
Conversely, suppose $G_f(z)=0$ and $z\notin \cal{K}_f$. Then $|f^n(z)|\to\infty$, so for $n$ sufficiently large,
\[
|f^n(z)|>\max\!\left\{C_f,\exp\!\left(-\frac{\log|a_d|}{d-1}\right)\right\}.
\]
By (1) and (2), this implies
$d^n G_f(z)=G_f(f^n(z))>0,$
which leads to a contradiction.

(4) Let $L(z)=az+b$. Note that $z$ belongs to the filled-in Julia set of $L^{-1}\circ f\circ L$ if and only if $L(z)$ belongs to the filled-in Julia set of $f$. By (3), this yields
$G_{L^{-1}\circ f\circ L}(z)=0$ iff 
$G_f(L(z))=0.$
Now assume $z\notin \cal{K}_{L^{-1}\circ f\circ L}$, so that
$|L^{-1}\circ f^n \circ L(z)|\to\infty.$
Note that $\log |L^{-1}(z)| = \log|z| + \log|a^{-1}|$ for $|z|$ large enough. Therefore,
\begin{align*}
G_{L^{-1}\circ f\circ L}(z)
&= \lim_{n\to\infty}\frac{1}{d^n}\log|L^{-1}\circ f^n\circ L(z)| \\
&= \lim_{n\to\infty}\frac{1}{d^n}\log|f^n(L(z))|
= G_f(L(z)),
\end{align*}
which completes the proof.
\end{proof}
\subsubsection{Sublevel set of the Green function}
Now set
$M_f=\max_{|z|\le C_f+1}\{ G_f(z),\frac{\log|a_d|}{d-1}+\log(C_f+1)\}$.
Then for any \( t>M_f \), Proposition~\ref{prop: basicgreen}(2) implies
\begin{align}
\label{eqq: greenbig-disk}
   \{G_f<t \}
=
\Bigl\{\log|z|<t -\frac{\log|a_d|}{d-1}\Bigr\},
\end{align}
which is an open disk. Moreover, 
we obtain the following proposition.
\begin{proposition}
\label{prop: predisk-green}
   For any \( t>0 \), the sublevel set
$\{G_f<t\}$
is a finite disjoint union of open disks. Moreover, if \( x \) is a boundary point of a connected component of \( \{G_f<t\} \), then
$G_f(x)=t.$
\end{proposition}
\begin{proof}
   Choose an integer \( n \) sufficiently large so that \( t d^n > M_f \). By~\eqref{eqq: greenbig-disk}, we have
\[
\{G_f<t d^n\}
=
\Bigl\{\log|z|< t d^n - \frac{\log|a_d|}{d-1}\Bigr\}.
\]
This set is an open disk whose boundary point is
$
x_0=\zeta\!\left(0,\exp\Bigl(t d^n-\frac{\log|a_d|}{d-1}\Bigr)\right).$
By construction, $
|x_0|=\exp\!\Bigl(t d^n-\frac{\log|a_d|}{d-1}\Bigr)>C_f,$
and it follows from Proposition~\ref{prop: basicgreen}(2) that
$G_f(x_0)=t d^n.$

Since $\{G_f<t\}=f^{-n}\{G_f<t d^n\},$ by Proposition~\ref{prop: preimage-disk}, 
 \( \{G_f<t\} \) is a finite disjoint union of open disks, and $x$ belongs to \( f^{-n}(x_0) \). Thus, we have 
\[
G_f(x)=\frac{1}{d^n}G_f(f^n(x))
=\frac{1}{d^n}G_f(x_0)=t,
\]
which completes the proof.
\end{proof}

\subsection{Basic properties of the convex hull}
\subsubsection{Convex hull of the Julia set with infinity} Fix a polynomial \(f\) of degree \(d \ge 2\) over \(k\). Let \(T_f\) denote the set of points in \(\mathbb{A}^{1,\mathrm{an}}_k \setminus \mathcal{K}_f\) that belong to the convex hull of \(\mathcal{J}_f \cup \{\infty\}\).
As a connected subset of $\mathbb{P}_k^{1,\mathrm{an}}$, the set $T_f$ carries a natural $\mathbb{R}$-tree structure. 

The following proposition suggests that one can describe a point in $T_f$ using the Green function.

\begin{proposition}
\label{prop: juliaconvexhull}
 Let $x \in \bA_k^{1,\mathrm{an}}$. Then $x\in T_f$  if and only if $x$ is a boundary point of the sublevel set $\{G_f < t\}$ for some $t>0$.
\end{proposition}
\begin{lemma}
\label{lem: increasinggreen}
Let $x_0=\zeta(z_0,r_0)$ be a point in the Julia set of $f$.  Then the following statements hold:
\begin{enumerate}
    \item $\bar D(z_0,r_0)$ is contained in the filled-in Julia set $\cal{K}_f$;
    
    \item the Green function \( G_f \) induces a strictly increasing, continuous, and surjective map
    \[
    [r_0,\infty)\longrightarrow[0,\infty),
    \qquad 
    r\longmapsto G_f(\zeta(z_0,r));
    \]
    
    \item for every \( r>r_0 \),  $D(z_0,r)$ is a connected component of 
    $\{G_f< G_f(\zeta(z_0,r))\}.$
\end{enumerate}
\end{lemma}

\begin{proof}
We first prove that \( G_f(\zeta(z_0,r)) \) is non-decreasing in \( r \).
Let \( x_i=\zeta(z_0,r_i) \) with \( r_1<r_2 \). Then for every integer \( n\ge1 \), one has
\[
|f^n(x_1)|
=\sup_{z\in \bar{D}(z_0,r_1)} |f^n(z)|
\le \sup_{z\in \bar{D}(z_0,r_2)} |f^n(z)|
=|f^n(x_2)|.
\]
Thus \( G_f(x_1)\le G_f(x_2) \). 
Since \( x_0 \in \mathcal{J}_f \), it follows that \( x_0 \in \mathcal{K}_f \). By Proposition~\ref{prop: basicgreen}(3), we have \( G_f(x_0) = 0 \). Consequently, for all \( z \in \bar{D}(z_0, r_0) \), it holds that \( G_f(z)\le G_f(\zeta(z_0,r_0))=G_f(x_0) = 0 \), and hence
$\bar{D}(z_0, r_0) \subset \mathcal{K}_f.$

To prove the second statement, we claim that $(x_0,\infty)\cap \mathcal K_f=\varnothing.$
Indeed, if there exists \( r_1>r_0 \) such that \( \zeta(z_0,r_1)\in \mathcal K_f \),
then the above argument implies \( D(z_0,r_1)\subset \mathcal K_f \).
In particular, \( x_0\in D(z_0,r_1) \) would belong to the interior of
\( \mathcal K_f \), contradicting the assumption that
\( x_0\in \mathcal J_f \).

Now let \( x_i=\zeta(z_0,r_i) \) with \( r_0<r_1<r_2<\infty \).
For any integer \( n\ge1 \), by~\eqref{eqq: imagedisk} we have
$f^n(x_i)=\zeta(f^n(z_0),r_{i,n})$ with $r_{1,n}<r_{2,n}.$
Since \( z_0 \) has bounded forward orbit while \( x_1 \) and $x_2$ escape to infinity,
there exists \( n \) sufficiently large such that $
|f^n(z_0)|\le C_f<r_{1,n}<r_{2,n},$ where $C_f$ is introduced in~\eqref{eqq: goodbound}. 
By~Proposition \ref{prop: basicgreen}(2), we obtain
\[
G_f(f^n(x_i))=\frac{\log|a_d|}{d-1}+\log r_{i,n}.
\]
Therefore, $
G_f(x_1)=\frac{1}{d^n}G_f(f^n(x_1))
<\frac{1}{d^n}G_f(f^n(x_2))=G_f(x_2),$
which proves strict monotonicity. Continuity follows from the continuity of \( G_f \),
and surjectivity follows from the fact that $\lim_{z\to\infty} G_f(z)=+\infty .$

Finally, we prove the last statement. Since $\partial D(z_0,r)=\zeta(z_0,r)$, it suffices to prove $D(z_0,r)\subset \{G_f<G_f(\zeta(z_0,r))\}.$ Let \( p\in D(z_0,r) \), by strict monotonicity, one has $
G_f(\zeta(z_0,|p-z_0|))<G_f(\zeta(z_0,r)).$
Since $p \in D(z_0,|p-z_0|)$, we have $
G_f(p)\le G_f\bigl(\zeta(z_0,|p-z_0|)\bigr)<G_f\bigl(\zeta(z_0,r)\bigr),$
which shows that $
p\in\{ G_f< G_f(\zeta(z_0,r))\}.$
\end{proof}
\begin{proof}[Proof of Proposition~\ref{prop: juliaconvexhull}]
     Let \( x=\zeta(a,r)\in \bA_k^{1,\an} \) be a point belonging to the convex hull.  
Then there exists \( z_0\in \calJ_f \) such that \( x \) lies on the segment
\(
I=[z_0,\infty).
\)
Write \( x=\zeta(z_0,r) \). By Lemma~\ref{lem: increasinggreen}(3), 
$D(z_0,r)$ is a connected component of $\{G_f<G_f(x)\}.$ Hence $x\in\partial\{G_f<G_f(x)\}.$

 To show the converse, it suffices to prove that the closed disk \( \bar{D}(a,r) \) contains a point in the Julia set.  

Choose an integer \( n\ge 1 \) such that $G_f(f^n(x))=d^n G_f(x)>M_f,$
where \( M_f \) is as in~\eqref{eqq: greenbig-disk}. Then the sublevel set 
$\{G_f< d^n G_f(x)\}=f^n(\{G_f<G_f(x)\})$
is a disk whose unique boundary point is \( f^n(x) \), which is defined by the disk $f^n(\bar{D}(a,r)).$ Thus, we obtain 
$\{G_f< d^n G_f(x)\}\subset f^n(\bar D(a,r)).$

Since every point in the Julia set satisfies \( G_f=0 \), we have
$\calJ_f\subset f^n(\bar D(a,r)).$
As the Julia set is backward invariant, it follows that there exists a point of \( \calJ_f \) contained in \( \bar{D}(a,r) \).  
Therefore, \( x \) belongs to the convex hull.
\end{proof}

\subsubsection{Restrict polynomial map to the convex hull}
Since $G_f\circ f=dG_f$, one has $f^{-1}(\{G_f<t\})=\{G_f<\frac{t}{d}\}$ and $f(\{G_f<t\})=\{G_f<td\}$. It then follows from Proposition~\ref{prop: juliaconvexhull} that $T_f$ is totally invariant under $f$:  $f^{-1}(T_f)=f(T_f)=T_f$.

Moreover, one can endow
\( T_f \)  with additional geometric data.
For any  \( x=\zeta(a,r)\in T_f \),  one can associate to it the closed disk $D_x\coloneqq\bar{D}(a,r).$ We define the \emph{mass} of $x$ by
\[
\mu_f(x):=\mu_f(D_x),
\]
where $\mu_f$ is the equilibrium measure of $f$.

Given a segment \( e=[x_1,x_2]\subset T_f \) with \( D_{x_1}\subset D_{x_2} \), let \( r_i \) be the radius of \( D_{x_i} \) for \( i=1,2 \).
We define the \emph{length} of \( e \) by
\[
l(e)=\log\frac{r_2}{r_1}.
\]

The following proposition summarizes how the mass and the length behave under \( f \).

\begin{proposition}
\label{prop: mass-length}
Let $x\in T_f$. The following statements hold:
\begin{enumerate}
\item One has $f(D_x)=D_{f(x)}$. 
\item For all sufficiently large $m\in\mathbb N$, one has $\mu_{f}(f^m(x))=1.$
    \item For every vertex \( x\in T_f \), one has
    \[
    \mu_f(f(x))=\frac{d\mu_f(x)}{\deg_xf}.
    \]
     \item Let $e = [x_1, x_2]$ such that $D_{x_1} \subset D_{x_2}$. If the interval $(x_1, x_2)$ contains no branch points of $T_f$, then:
    \[
    l\bigl(f(e)\bigr)=\deg_{x_1} f \cdot l(e).
    \]
    In particular,
    $l(e)\le l\bigl(f(e)\bigr)\le d\cdot l(e).$
\end{enumerate}
\end{proposition}
\begin{proof}
(1) The first statement follows from \S\ref{sec: imageofdisks}.

(2) Note that for a sufficiently large $m \in \mathbb{N}$, equation \eqref{eqq: greenbig-disk} implies that the set $\{G_f < G_f(f^m(x))\}$ is the disk  $D_{f^m(x)}$. Since the measure $\mu_f$ is supported on the Julia set, and Proposition \ref{prop: basicgreen}(3) establishes that the filled-in Julia set is the set $\{G_f = 0\}$, it follows that the support of $\mu_f$ is entirely contained within $D_{f^m(x)}$. Consequently, we have $\mu_f(D_{f^m(x)}) = 1$.

(3) We  pick \( \epsilon>0 \), and
choose a compact subset \( K_\epsilon\subset D_{f(x)} \) satisfying $
\mu_f\left(D_{f(x)}\setminus K_\epsilon\right)<\epsilon.$
Let $K \coloneqq f^{-1}(K_\epsilon)\cap  D_x.$ Since \( f_*\mu_f=\mu_f \), we obtain
\[
\mu_f(D_x\setminus K)
\le \mu_f\left(f^{-1}(D_{f(x)}\setminus K_\epsilon)\right)
= \mu_f\left(D_{f(x)}\setminus K_\epsilon\right)
<\epsilon .
\]

Let \( 0\le \varphi_\epsilon \le 1 \) be a continuous function supported in
\( D_x \) such that \( \varphi_\epsilon \equiv 1 \) on \( K \).
Then \( f_*\varphi_\epsilon=\deg_x f \) on \( K_\epsilon \).
Hence
\begin{align*}
\mu_f(D_x)-\epsilon
&\le \int \varphi_\epsilon\, d\mu_f
= \frac{1}{d}\int f_*\varphi_\epsilon\, d\mu_f
\le \frac{\deg_x f}{d}\mu_f(K_\epsilon)+\epsilon,\\
\mu_f(D_x)
&\ge \int \varphi_\epsilon\, d\mu_f
= \frac{1}{d}\int f_*\varphi_\epsilon\, d\mu_f
\ge \frac{\deg_x f}{d}\mu_f(K_\epsilon).
\end{align*}
Letting \( \epsilon\to 0 \), we obtain the desired result.

(4) 
By construction, there exists a point $x_0=\zeta(z_0,r_0)$ in the Julia set such that $
x_1\in [x_0,x_2].$
Write $x_i=\zeta(z_0,r_i)$ with $r_0<r_1<r_2$.

Since $(x_1,x_2)$ does not contain any branch point of $T_f$, it follows that $x_1$ is the unique preimage of $f(x_1)$ in $D(z_0,r_2)$. Hence $\deg_{x_1}f$ equals the Weierstrass degree of $f$ on $D(z_0,r_2)$.
It then follows from \cite[Theorem~3.33]{Ben19} that $f$ maps the annulus
$D(z_0,r_2)\setminus \bar D(z_0,r_1)$
onto $f(D(z_0,r_2))\setminus f(\bar D(z_0,r_1)),$
and that $
l\bigl(f(e)\bigr) = \deg_{x_1} f \cdot l(e).$
In particular, since $1 \le \deg_{x_1} f \le d$, we obtain $
l(e) \le l\bigl(f(e)\bigr) \le d \cdot l(e).$
\end{proof}
\subsection{Simplicial tree for tame polynomials}
Recall that a simplicial tree $T$ is a nonempty, connected, locally finite,
$1$-dimensional simplicial complex without cycles.
We denote by $V$ the set of vertices and by $E$ the set of (unoriented, closed) edges
(see e.g.,~\cite[\S 2]{DM08}).

\smallskip

In this section, we restrict our attention to tame polynomials. We endow the $\mathbb{R}$-tree $T_f$ with a simplicial tree structure such that the map $f \colon T_f \to T_f$ is simplicial; that is, $f$ maps each edge $e = [v_1, v_2]$ onto the edge $[f(v_1), f(v_2)]$. 
\begin{remark}
  The tameness assumption is imposed to guarantee that the grand orbit of the branch points of $T_f$
 is discrete, which is necessary for defining the edges. It is not clear whether this property holds for arbitrary polynomials.
\end{remark}

\subsubsection{Tame polynomials}
The notion of a tame polynomial was introduced by Trucco~\cite{Tru14} for polynomials for which wild ramification does not occur.

Recall that the \emph{ramification locus} of a polynomial $f$ of degree $d\ge 2$ is defined by
$\mathcal{R}_f \coloneqq \{\, x \in \bA_k^{1,\mathrm{an}} \mid \deg_x f \ge 2\}.$
Observe that the set of finite critical points $
\mathrm{Crit}(f) \coloneqq \{\, x \in k \mid f'(x)=0 \}$
is contained in $\mathcal{R}_f$.
For a thorough study of the ramification locus, we refer the reader to
\cite{Fab13a,Fab13b}.

\smallskip

Let $\bH_k=\bP_k^{1,an}\setminus\bP^{1}(k)$ be the hyperbolic space over $k$. If $\mathcal{R}_f \cap \bH_k$ is a locally finite open subtree, then we say that $f$ is a
\emph{tame polynomial}. Let $p$ denote the residue characteristic of $\tilde{k}$.
By~\cite[Proposition~7.11]{Fab13a}, the polynomial $f$ is tame if and only if either $p=0$ or
$p>0$ and $
p \nmid \mathrm{wdeg}_D f$ for every disk  $D \subseteq \bA_k^{1,\mathrm{an}}.$
In particular, if $p=0$ or $p \ge d+1$, then $f$ is tame.

\smallskip

Moreover, the following proposition holds; see~\cite[\S\,2.6]{Tru14} and~\cite[\S\,2.2]{KN22}.

\begin{proposition}
\label{prop:tamecharacterization}
Let $f$ be a tame polynomial. Then for every disk $D$, the number of critical points in $D$
(counted with multiplicity) is equal to
$\mathrm{wdeg}_D f - 1.$
In particular, $\mathcal{R}_f \cup \{\infty\}$ is the convex hull of
$\mathrm{Crit}(f) \cup \{\infty\}$.
\end{proposition}

\begin{proof}
By~\cite[Proposition~3.9]{Ben19}, one has
\[
\mathrm{wdeg}_D f = \mathrm{wdeg}_D f' + 1.
\]
Moreover, by~\cite[Theorem~3.13(a)]{Ben19}, there are exactly $\mathrm{wdeg}_D f - 1$ critical
points in $D$, counted with multiplicity. Moreover, $\deg_xf\geq 2$ if and only if the associated closed ball contains a critical point. Thus $\cal{R}_f\cup \{\infty\}$ is the convex hull of $\mathrm{Crit}(f) \cup \{\infty\}$.
\end{proof}
\subsubsection{The simplicial tree structure on $T_f$}
The construction of the simplicial tree relies on the following characterization of branch points of $T_f$  using the values of Green function at critical points.

\begin{proposition}
\label{prop: char-branch}
A point $x$ is a branch point of $T_f$ if and only if there exists a point 
$c'\in D_x$ lying in the backward orbit of a critical point of $f$ such that
$G_f(x)=G_f(c').$
\end{proposition}
\begin{proof}
Up to replacing $f$ by an iterate, we may assume that $
D := \{ G_f < d\, G_f(x) \}$ is an open disk. Let $D_i$, $1 \le i \le m$, be the connected components of
$\{ G_f < G_f(x) \}$ whose boundary is equal to $x$.

Denote by $d_i$ the Weierstrass degree of $f$ on $D_i$.
By Proposition~\ref{prop: preimage-disk}(1), we have $
\deg_x f = \sum_{i=1}^m d_i.$
Moreover, by Proposition~\ref{prop:tamecharacterization}, the number of critical points in
$D_x \setminus \bigcup_{i=1}^m D_i$
is equal to
\[
(\deg_x f - 1) - \sum_{i=1}^m (d_i - 1) = m - 1.
\]
Since $x$ is a branch point if and only if $m \ge 2$, this is equivalent to the existence of a critical point
$c$ in $D_x$ such that $
G_f(x) = G_f(c).$
\end{proof}

We define $V_f$ to be the set of points $x\in \bP_k^{1,\mathrm{an}}$ that belong to the grand orbit of a branch point of $T_f$. Since every branch point of $T_f$ is of type~II and $f$ preserves point types, it follows that every $v\in V_f$ is also a type~II point.

Since $G_f\circ f=dG_f$, Proposition~\ref{prop: char-branch} implies that
\[
G_f(V_f)\subset \bigcup_{\substack{n\in\mathbb{N}\\ f'(c)=0}} d^n G_f(c),
\]
which is a discrete subset of $\mathbb{R}_+$.

It follows that $T_f$ admits a decomposition into segments between successive vertices. Denote by $E_f$ the collection of closed intervals $[v_1,v_2]$ such that $v_1,v_2\in V_f$ and
\[
(v_1,v_2)\cap V_f=\varnothing.
\]

By construction, $V_f$ contains all branch points of $T_f$; thus, the interiors of distinct intervals in $E_f$ are disjoint. Furthermore, Proposition~\ref{prop: predisk-green} implies that each sublevel set $\{G_f < t\}$ has only finitely many boundary points. It follows that every vertex $v \in V_f$ is adjacent to only finitely many intervals in $E_f$. Therefore, the $\mathbb{R}$-tree $T_f$, equipped with vertex set $V_f$ and edge set $E_f$, is a simplicial tree.

 Moreover,  the induced map  \( f\colon T_f\to T_f \) is simplicial by construction.

\smallskip

\begin{proposition}
\label{prop:treestructure}
For any $l \ge 1$, the polynomial tree $T_f$ coincides with $T_{f^l}$ as a simplicial tree.
\end{proposition}

\begin{proof}
Since the Julia set is invariant under iteration, the underlying \(\mathbb{R}\)-trees \(T_f\) and \(T_{f^l}\) coincide. It remains only to prove that their vertex sets \(V_f\) and \(V_{f^l}\) are the same. Recall that \(G_f=G_{f^l}\). Thus \(T_f\) and \(T_{f^l}\) have the same set \(S\) of branch points. From the definition of the vertex set, the inclusion $V_{f^l}\subset V_f$ is immediate, since any relation involving iterates of \(f^l\) is also a relation involving iterates of \(f\). Conversely, let \(v\in V_f\). By definition, there exists a branch point \(x\in S\) and integers \(m,n\geq 1\) such that $f^m(v)=f^n(x).$ Applying a sufficiently high iterate of \(f\) to both sides if necessary, we may assume that \(m\) is divisible by \(l\). Thus we may write $m=lk_1$ for some \(k_1\geq 1\). Similarly, write $n=lk_2-r,$ where \(k_2\geq 1\) and \(0\leq r\leq l-1\). Choose \(x'\in f^{-r}(x)\). Since preimages of branch points are again branch points of \(T_f\), we have \(x'\in S\). Therefore \[ (f^l)^{k_1}(v) = f^{lk_1}(v) = f^{lk_2-r}(x) = f^{lk_2}(x') = (f^l)^{k_2}(x'). \] By the definition of the vertex set for \(f^l\), this implies \(v\in V_{f^l}\). Hence \(V_f\subset V_{f^l}\). Combining the two inclusions, we obtain $ V_f=V_{f^l}. $
\end{proof}







\section{Non-Archimedean Rigidity}
In this section, we fix a non-Archimedean field $k$ with residue characteristic $0$.  We will fix two polynomials \(f\) and \(g\) of degree at least \(2\) such that their Julia sets coincide. 
By~\cite[\S4]{FRL06} and \cite[Theorem~10.91]{BR10}, this condition is equivalent to the equality of their equilibrium measures,
\(\mu_f=\mu_g\), and also equivalent to the equality of their Green functions,
\(G_f=G_g\).

\subsection{Comparison of polynomials with identical  trees}
\label{sec: compareDMtree}
Since the residue characteristic of $k$ is $0$, $f$ and $g$ are automatically tame. 

To establish the rigidity of Julia sets, we first assume that $f$ and $g$ have the same degree $d\ge 2.$ Our strategy is to compare the induced actions of \(f\) and \(g\) on their associated polynomial trees.

The first step is the following proposition, which shows the induced simplicial trees coincide with each other.

\begin{proposition}
Let $f,g$ be two polynomials of degree $d\ge2$ such that $\calJ_f=\calJ_g$. Then the associated polynomial trees \(T_f\) and \(T_g\) coincide as simplicial trees.
\end{proposition}

\begin{proof}
By construction, $T_f=T_g$ as $\R$-trees. In particular, they have the same branch points. It therefore suffices to show that their vertex sets agree.
Let \(v \in V_f\), then there exist integers
\(m,n \in \mathbb{N}\) and a branch point \(v_\star\) such that $f^m(v) = f^n(v_\star).$ Choosing \(m\) sufficiently large, we may also suppose that
$D=\{ G_f < G_f(f^m(v)) \}$
is a disk. Since $\calJ_f=\calJ_g$, one obtains \(G_f = G_g\). The equality of degrees implies $G_f\circ f^l=d^lG_f=d^lG_g=G_g\circ g^l.$ Therefore, \[G_g\bigl(g^m(v)\bigr) =G_f\bigl(f^m(v)\bigr)=G_{f}(f^n(v_\star)) = G_g\bigl(g^n(v_\star)\bigr)\]
and $D=\{ G_g < G_g(g^m(v)) \}.$
Therefore,
 $g^m(v) = g^n(v_\star)=\partial D.$
Hence \(v\) belongs to the grand orbit of \(v_\star\) under \(g\), and therefore \(v \in V_g\). By symmetry, one obtains \(V_g \subset V_f\). Thus $V_f=V_g.$
\end{proof}

To compare the coefficients of \( f \) and \( g \), we study their preimages of a given fixed element in $k$. This requires comparing the preimages of vertices under \( f \) and \( g \). However, $f$ and $g$ may act differently on the tree and hence it is not necessarily true that $
f^{-1}(v)=g^{-1}(v).$
\par 
To overcome this difficulty, we associate to each vertex $v$ a subset
$X_v \subset \partial\{G_f<G_f(v)\}$
containing $v$, which is minimal with respect to the invariance property
\[
f^{-1}(X_v)=g^{-1}(X_v).
\]

The construction of $X_v$ is inductive.  
We begin by setting $X_0=\{v\}$ and $Y_0=\varnothing$.  
Define a sequence of maps $(h_s)_{s\ge 0}$ by letting $h_s=f$ if $s$ is even and $h_s=g$ if $s$ is odd.  
For each $s\ge1$, set
\[
X_s = X_{s-1} \cup h_s(Y_{s-1}),
\qquad
Y_s = Y_{s-1} \cup h_s^{-1}(X_{s-1}).
\]
Finally, define
\begin{align*}
X_v &= \bigcup_{s\ge0} X_s, \\
Y_v &= \bigcup_{s\ge0} Y_s.
\end{align*}

The  following proposition is a key ingredient in the proof of Theorem~\ref{IntroRigidity1}.

\begin{proposition}
\label{prop: controldistance}
Let $f,g$ be two polynomials of degree $d\ge2$ such that $\calJ_f=\calJ_g$. Then the following statements hold.
    \begin{enumerate}
        \item  $Y_v=f^{-1}(X_v)=g^{-1}(X_v)$;
        \item \(w \in Y_v\) if and only if there exist \(s \in \mathbb{N}\), distinct vertices
\(v_0,\dots,v_s \in X_v\) with \(v_s=v\), and points
\(w_0,\dots,w_s \in Y_v\) with \(w_0=w\), such that
\[
h_s(w_i)=v_i \quad (0\le i\le s), \qquad
h_{s-1}(w_{i+1})=v_i \quad (0\le i\le s-1).
\]
The diagram below illustrates how the vertices \(v_i\) and points \(w_i\) are selected:
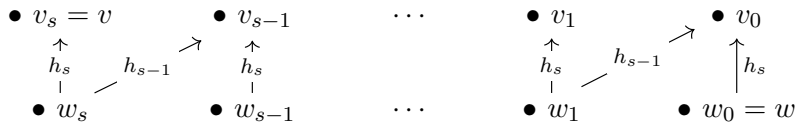
\[\begin{tikzcd}
	{\bullet~v_s=v} & {\bullet~v_{s-1}} & \cdots & {\bullet~v_{1}} & {\bullet~v_0} \\
	{\bullet~w_s} & {\bullet~w_{s-1}} & \cdots & {\bullet~w_{1}} & {\bullet~ w_0=w}
	\arrow["{h_s}"{description}, from=2-1, to=1-1]
	\arrow["{h_{s-1}}"{description}, from=2-1, to=1-2]
	\arrow["{h_{s}}"{description}, from=2-2, to=1-2]
	\arrow["{h_{s}}"{description}, from=2-4, to=1-4]
	\arrow["{h_{s-1}}"{description}, shift left, from=2-4, to=1-5]
	\arrow["{h_{s}}"', from=2-5, to=1-5]
\end{tikzcd}\]
        \item Let $\omega\in Y_v$.  
Let $e_w=[w,w']$ and $e_v=[v,v']$ be edges such that $D_w\subset D_{w'}$ and $D_v\subset D_{v'}$.  
Then we have $w' \in Y_{v'};$
        \item There exists a non-empty finite set
\[
A\subset \bigcup_{v_i\in X_v}(D_{v_i}\cap k)
\]
such that, for every \(w\in Y_v\), \(A\) has the same number of preimages in \(D_w\) under \(f\) and \(g\), counting multiplicities.
    \end{enumerate}
\end{proposition}

\begin{proof}[Proof of Proposition~\ref{prop: controldistance}]
(1) 
First, let $w\in Y_v$, and let $s$ be the minimal integer such that $w\in Y_s$.
If $s$ is even, then $
Y_s = Y_{s-1}\cup f^{-1}(X_{s-1})$.  By assumption, $w\in f^{-1}(X_{s-1}).$
Hence $f(w)\in X_{s-1}\subset X_v$.
If $s$ is odd,  then $
X_{s+1}=X_s\cup f(Y_{s}).$ Hence $w\in Y_s\subset  f^{-1}(X_{s+1})\subset f^{-1}(X_v).$
The same argument shows that $Y_v\subset g^{-1}(X_v).$

Conversely, let $w\in f^{-1}(X_v)$, and let $s$ be minimal such that
$w\in f^{-1}(X_s)$. Note that $X_s\subset X_{s+1}$. If $s$ is even, then  $
w\in f^{-1}(X_s)\subset Y_{s+1}\cup f^{-1}(X_{s+1})=Y_{s+2}$. If $s$ is odd, then $
w\in Y_{s+1}=Y_{s}\cup f^{-1}(X_{s}).$ 
In either case, we conclude that $w\in Y_v$.
Therefore $f^{-1}(X_v)\subset Y_v$.
The same argument shows that
$g^{-1}(X_v)\subset Y_v$.

Combining both inclusions yields $f^{-1}(X_v)=g^{-1}(X_v)=Y_v,$
as claimed.

\smallskip

(2) The if direction follows directly from (1). To show the only-if direction, let $s$ be the minimal integer such that $w=w_0\in Y_s$, where $
Y_s = Y_{s-1} \cup h_s^{-1}(X_{s-1}).$
Then necessarily $w_0 \in h_s^{-1}(X_{s-1})$.  

Set  $
v_0 = h_s(w_0) \in X_{s-1}=X_{s-2}\cup h_{s-1}(Y_{s-2}).$  We claim that $v_0\notin X_{s-3}$, otherwise $w_0\in h_{s}^{-1}(X_{s-3})=h_{s-2}^{-1}(X_{s-3})\subseteq Y_{s-2}$, which contradicts the minimality of $s.$ Therefore, $v_0\notin X_{s-2}.$

Choose $w_{1} \in Y_{s-2}$ such that
$h_{s-1}(w_{1}) = v_0.$
We claim that $s-2$ is the least integer $t$ such that $w_{1} \in Y_{t}$.  
Otherwise, we would have
$w_{1} \in Y_{s-3} = Y_{s-4} \cup h_{s-3}^{-1}(X_{s-4}).$
Since $
h_{s-3}(w_{1}) = v_0 \notin X_{s-3},$
it follows that $w_{1} \in Y_{s-4}$.  
Consequently, $
v_0 = h_{s-3}(w_{1}) \in h_{s-3}(Y_{s-4}) \subset X_{s-3},$
which contradicts the previous claim.

Iterating this procedure, we successively construct $v_0,\cdots,v_s$ and $w_0,\cdots,w_s$
 satisfying the required conditions.

(3) By (2), there exist vertices
$v_0 ,  v_1, \dots, v_s=v \in X_v$ and $
w_0=w, \dots, w_s  \in Y_v$
such that for all $0\le i\le s$,
$h_{s}(w_i) = v_i$
and $h_{s-1}(w_{i+1}) = v_{i}.$
Then the corresponding edges $e_{w_i}=[w_i,w_i']$ and $e_{v_i}=[v_i,v_i']$ satisfy
$
h_{s}(e_{w_i}) = e_{v_i}$
 and $h_{s-1}(e_{w_{i+1}}) = e_{v_{i}}.$
In particular, $
h_{s}(w_i') = v_i'$ and $
h_{s-1}(w_{i+1}') = v_{i}'.$
Applying (2) once again, we conclude that
$w' \in Y_{v'}.$

\smallskip

(4) Let  $X_v=\{v_i\}_{1\le i\le l},$ and $
Y_v=\{w_j\}_{1\le j\le l'}.$ By Proposition~\ref{prop: mass-length}(2), there exists $m\in\bN$ large enough such that $\mu_f(f^m(v_i))=\mu_f(f^{m+1}(w_j))=1$  for all $1\le i \le l$ and $1\le j\le l'$. Note that $\mu_f=\mu_g$.
By Proposition~\ref{prop: mass-length}(3), we may define
\begin{align*}
 d_i&\coloneqq \deg_{v_i}f^m=\deg_{v_i}g^m=d^m\mu_f(v_i)\\
    \delta_j&\coloneqq \deg_{w_j}f^{m+1}=\deg_{w_j}g^{m+1}=d^{m+1}\mu_f(w_j)
\end{align*}
For each $j$, there exist indices $i_j$ and $k_j$ such that
$f$ (resp.\ $g$) maps $D_{w_j}$ onto $D_{v_{i_j}}$
(resp.\ $D_{v_{k_j}}$). We then have
$\deg_{w_j}f
 = \frac{\delta_j}{d_{i_j}},$ and 
$\deg_{w_j}g
 = \frac{\delta_j}{d_{k_j}} .$

Now choose a finite set
$A \subset \bigcup_{i=1}^{l} (D_{v_i}\cap k)$
such that $\#(A\cap D_{v_i})=d_i$ for each $1\le i\le l$.
Then, for every $w_j\in Y_v$, the number of preimages of $A$
under $f$ in $D_{w_j}$ is $
d_{i_j}\cdot \frac{\delta_j}{d_{i_j}} = \delta_j,$
and similarly the number of preimages under $g$ in $D_{w_j}$ is $
d_{k_j}\cdot \frac{\delta_j}{d_{k_j}} = \delta_j.$ 
\end{proof}

\subsection{Rigidity of Julia sets for polynomials of the same degree}

Recall that for the Julia set $\mathcal J$ of a polynomial of degree at least $2$, $\Aut(\mathcal J)$ denotes the group of affine transformations that preserve $\mathcal J$.

\smallskip

\begin{lemma}
\label{lem: charautojulia}
Let $f$ be a polynomial of degree at least $2$, and let $\sigma$ be an affine transformation. Then the following statements are equivalent:
\begin{enumerate}
    \item $\sigma \in \Aut(\calJ_f)$;
      \item $\mathcal J_{f \circ \sigma} = \calJ_f$;
    \item $\mathcal J_{\sigma \circ f} = \calJ_f$.
\end{enumerate}
\end{lemma}

\begin{proof}
Assume first that $\sigma \in \Aut(\calJ_f)$. Since $\sigma(\calJ_f)=\calJ_f$, the measures
$\mu_{\sigma^{-1}\circ f\circ \sigma}$ and $\mu_f$ are equilibrium measures with the same support. Hence,
$\sigma^{*}\mu_f
=
\mu_{\sigma^{-1}\circ f\circ \sigma}
=
\mu_f.$
Therefore, \[
(f\circ \sigma)^{*}\mu_f
=
(\sigma \circ f)^{*}\mu_f
=
d\mu_f.\]
Applying Theorem~\ref{thm: equidistribution} to $f\circ \sigma$ (or  to $\sigma \circ f$)  and
measure $\rho=\mu_f$, we obtain
$\mu_{f\circ \sigma}=\mu_{ \sigma\circ f}=\mu_f.$
Consequently, $
\mathcal \calJ_{f\circ \sigma}
=
\mathcal \calJ_{\sigma \circ f}
=
\mathcal \calJ_f.$

Assume that $\mathcal J_{f \circ \sigma}=\calJ_f$. Since $\calJ_f$ is totally invariant under $f$ and $f\circ \sigma$, we obtain
\[
\calJ_f = (f \circ \sigma)^{-1}(\calJ_f) =  \sigma^{-1}(\calJ_f).
\]
Thus $\sigma \in \Aut(\calJ_f)$, which completes the proof.
\end{proof}
We now prove the following rigidity result, which is the core of the proof of Theorem~\ref{IntroRigidity1}:  
\begin{theorem}
\label{thm: rigidity-samedegree}
  Let $k$ be a non-Archimedean field with residue characteristic $0.$  Let $f$ and $g$ be two polynomials of degree $d\ge 2$ which do not have potential good reduction. Suppose $\calJ=\calJ_f=\calJ_g$. Then there exists $\sigma\in\Aut(\calJ)$ such that $f=\sigma\circ g.$
  \end{theorem}
\begin{remark}
   When both $f$ and $g$ have good reduction, the statement is no longer valid.  
For example, consider
$f(z)=z^d$ and $g(z)=z^d+z.$
Both maps have good reduction, and therefore their Julia sets reduce to the Gauss point $x_g$.  
However, there is no affine transformation $\sigma$ such that
$f=\sigma\circ g.$
\end{remark}

Before beginning the proof of Theorem \ref{thm: rigidity-samedegree}, let us explain our argument when $f = z^2 + c_1, g = z^2 + c_2$ with $|c_i| > 1$. In this case, choose some large $\alpha > 0$ so that $\{G_f(z) < \alpha\} = \{G_g(z) < \alpha\}$ is given by a disk $D(0,R)$. Taking preimages, we obtain that 
$$f^{-n}(D(0,R)) = \left\{G_f(z) < \frac{\alpha}{2^n} \right\} = \left\{G_g(z) < \frac{\alpha}{2^n} \right\} = g^{-n}(D(0,R)).$$
Fix an $\eps > 0$. We may then take $n$ large enough so that $f^{-n}(D(0,R)) = g^{-n}(D(0,R))$ is a disjoint union of $2^n$ disks, each of radius $< \eps$. Now let $\alpha_1,\ldots,\alpha_{2^n}$ be the roots of $f^n(z) = 0$ and similarly let $\beta_1,\ldots,\beta_{2^n}$ be those for $g^n(z) = 0$. Then we may arrange them so that $|\alpha_i - \beta_i| < \eps$ for all $i$. In particular, we obtain 
$$\sum_{1 \leq i < j \leq 2^n} |\alpha_i \alpha_j - \beta_i \beta_j| \leq \sum_{1 \leq i < j \leq 2^n} \left|(\alpha_i - \beta_i) \alpha_j + (\alpha_j - \beta_j) \beta_i \right| \leq \max\{|c_1|,|c_2|\} \eps.$$
By Viète's formulas, this implies that the coefficients of $z^{2^n - 2}$ for $f^n(z)$ and $g^n(z)$ differ by at most $\max\{|c_1|,|c_2|\} \eps$. But this coefficient can easily be computed to be $2^{n-1} c_1$ and $2^{n-1} c_2$. Hence 
$$|2^{n-1}| |c_1 - c_2| < \max\{|c_1|,|c_2|\} \eps \implies |c_1 - c_2| < \max\{|c_1|, |c_2|\} \eps,$$
where we crucially used the fact that $k$ is of residue characteristic $0$. Letting \(\epsilon\to 0\) gives \(c_1=c_2\), proving rigidity in this case.
\par 
For general polynomials, we will use the subset $X_v$ for a suitable vertex $v$, as introduced in Proposition \ref{prop: controldistance} as a replacement for $f^{-n}(D(0,R))$ and $g^{-n}(D(0,R))$. Keeping the notation introduced in \S\ref{sec: compareDMtree}, one has:

\begin{lemma}
\label{lem: smallradiusvertex}
Let $v^0$ be a vertex such that the sublevel set $\{G_f < G_f(v^0)\}$ is a disk. Then for all $\epsilon>0$, there exists a vertex $v \in f^{-n}(v^0)$ for some $n \in \mathbb{N}$, such that  for any $w\in Y_v$,  the radius of the associated ball satisfies $\mathrm{diam}(D_w)<\epsilon.$ 
\end{lemma}

\begin{proof}

By \cite[Theorem 3.1]{Fav25}, $f$ admits a repelling type I periodic point $p$, which necessarily belongs to the Julia set. (Note: Although the original statement is formulated for polynomials over formal Laurent series, the proof remains valid for 
fields with residue characteristic $0$.
)

Up to iteration, we may assume $p$ is a fixed point. Note that $G_f(p) = 0$. Therefore, for any $n\in\bN$, there exists a unique connected component $D_n$ of $\{G_f < \frac{1}{d^n}G_f(v^0)\}$ containing $p$. 

We claim that there exists $n$ such that $\mathrm{diam}(D_n)<\epsilon.$ Suppose for contradiction that $\mathrm{diam}(D_n) \geq \epsilon$ for all $n$. Then  $D(p, \epsilon)$ is contained in the intersection $\bigcap_{n \in \mathbb{N}} D_n$. Since $f(\bigcap_{n \in \mathbb{N}} D_n) = \bigcap_{n \in \mathbb{N}} D_n$, we conclude that $D(p, \epsilon)$ is contained in the Fatou set of $f$, which contradicts the fact that $p$ belongs to the Julia set. Let $v=\partial D_n$.
Then, $v\in f^{-n}(v^0)$ and $\mathrm{diam}(D_v)<\epsilon.$

Let $w \in Y_v$. We claim that $v^0=f^n(v) = f^{n+1}(w)$. This follows from the fact that $\{G_f < G_f(f^n(v))\}$ and $\{G_f < G_f(f^{n+1}(w))\}$ are the same disk with boundary point $v^0 = f^n(v) = f^{n+1}(w)$.

Let $h_n=f$ if $n$ is even and let $h_n=g$ if $n$ is odd. By Proposition~\ref{prop: controldistance}(2), for any $w \in Y_v$, there exist $s \in \mathbb{N}$, distinct vertices $v_0, v_1, \dots, v_s=v \in X_v$, and $w_0=w, \dots, w_s \in Y_v$ such that $h_{s}(w_i) = v_i$ and $h_{s-1}(w_{i+1}) = v_{i}$. Let $e_{w_i}$ and $e_{v_i}$ denote the unique edges with lower vertices $w_i$ and $v_i$, respectively. According to Proposition~\ref{prop: mass-length}(4), we have the length relations $l(e_{v_i}) = \deg_{w_i}h_s \cdot l(e_{w_i})$ and $l(e_{w_{i+1}}) \le l(e_{v_{i}})$ . Combining these yields:
\[
l(e_v) \le \left(\prod_{i=0}^s \deg_{w_i} h_s\right)\cdot l(e_w).
\]

Note that the disks $D_{w_i}$ are disjoint disks. By Proposition~\ref{prop:tamecharacterization}, each disk $D_{w_i}$ contains exactly $\deg_{w_i} h_s - 1$ critical points (counted with multiplicities). Since the total number of critical points of $h_s$ is $d-1$, there are at most $d-1$ indices $i$ such that $\deg_{w_i} h_s \ge 2$. We therefore obtain
\[
l(e_v) \le d^{d-1}\, l(e_w).
\]

Let $v^0, v^1, \dots, v^m = v$ be the sequence of vertices obtained by descending along the tree to $v$, and let $v^0 = w^0, w^1, \dots, w^{m+j} = w$ be the sequence descending to $w$. Since $w\in Y_v$ by assumption, it follows from Proposition~\ref{prop: controldistance}(3) that $w^{i+j} \in Y_{v^i}$ for each $1 \le i \le m$. Therefore,  $l(e_{v^i}) \le d^{d-1} l(e_{w^{i+j}})$. Summing these lengths, we obtain:
\[
\log \frac{\mathrm{diam}(D_{v^0})}{\mathrm{diam}(D_v)} = \sum_{i=1}^{m} l(e_{v^i}) \le d^{d-1} \sum_{i=1}^{m} l(e_{w^{i+j}}) \le d^{d-1} \log \frac{\mathrm{diam}(D_{v^0})}{\mathrm{diam}(D_w)}.
\]
Consequently, as $\epsilon \to 0$, it follows that $\mathrm{diam}(D_w) \to 0$, which completes the proof.
\end{proof}

\begin{proof}[Proof of  Theorem~\ref{thm: rigidity-samedegree}]
By Lemma~\ref{lem: charautojulia}, it suffices to prove that there exists an affine map $\sigma$ such that $f=\sigma\circ g$.

Fix a vertex $v^0$ such that  $\{G_f < G_f(v^0)\}$ is a disk. By Lemma~\ref{lem: smallradiusvertex}, for all $\epsilon>0$, there exists a vertex \(v\in f^{-n}(v^0)\) such that for every
\(w\in Y_v\), one has
$\operatorname{diam}(D_w)<\epsilon.$ Write $Y_v=\{w_{j}\}_{1\le j\le l}.$
By Proposition~\ref{prop: controldistance}(4), there exists a finite set $A\subset \bigcup_{v_i\in X_v} D_{v_i}\cap k$ such that, for each \(1\le j\le l\), the sets \(f^{-1}(A)\) and \(g^{-1}(A)\)
have the same cardinality in \(D_{w_j}\), counted with multiplicity.
Since \(\operatorname{diam}(D_{w_j})<\epsilon\), any 
\(x,y\in D_{w_j}\cap k\) satisfy \(|x-y|<\epsilon\).

Let \(N=\#A\), and write $
f^{-1}(A)=\{x_1,\dots,x_{dN}\}$ and $g^{-1}(A)=\{y_1,\dots,y_{dN}\}$,
counted with multiplicity.
After reindexing, we may assume that
$|x_i-y_i|<\epsilon$ for all  $1\le i\le dN$. Since \(x_i,y_i\in D_{v^0}\) by construction, there exists a constant \(C>0\),
independent of \(\epsilon\), such that for every \(1\le s\le d\), one has
\[
\left|\sum_{i=1}^{dN} x_i^{s}-\sum_{i=1}^{dN} y_i^{s}\right|
\le C\epsilon.
\]

Write $
f(z)=a_d z^d+a_{d-1}z^{d-1}+\cdots+a_0,$ and $g(z)=b_d z^d+b_{d-1}z^{d-1}+\cdots+b_0.$ Assume that $g$ cannot be written as the composition of an affine map with $f$.
Let \( i \) be the largest integer such that $
\frac{a_i}{a_d} \neq \frac{b_i}{b_d}.$
Then \( i \ge 1 \).

Note that for $i \ge 1$, by Viète's theorem, the $(d-i)$-th elementary symmetric function of $f^{-1}(c)$ (resp.\ $g^{-1}(c)$) is independent of the choice of $c \in k$ and is equal to $\frac{a_i}{a_d}$ (resp.\ $\frac{b_i}{b_d}$).
Denote by $x_i^c$ (resp. $y_i^c$) the preimages for $c\in A$ under $f$ (resp. $g$).
Applying Newton’s identities and Viète’s formulas to preimages of all $c \in A$ under $f$ and $g$,  we obtain\begin{align*}
(-1)^{d-i}(d-i)
\left(\frac{a_i}{a_d}\right)
&=
(-1)^{i-1}
\sum_{s=1}^{d-i}
\frac{a_{i+s}}{a_d}
\left(
\sum (x_i^{c})^{s}
\right);\\
(-1)^{d-i}(d-i)
\left(\frac{b_i}{b_d}\right)
&=
(-1)^{i-1}
\sum_{s=1}^{d-i}
\frac{b_{i+s}}{b_d}
\left(
\sum (y_i^{c})^{s}
\right).
\end{align*} 
Since we assume that $\frac{a_{i+s}}{a_d}=\frac{b_{i+s}}{b_d}$, summing over all $c \in A$ and taking the difference, we obtain
\begin{align*}
(-1)^{d-i}(d-i)N
\left(\frac{a_i}{a_d}-\frac{b_i}{b_d}\right)
&=
(-1)^{i-1}
\sum_{s=1}^{d-i}
\frac{a_{i+s}}{a_d}
\left(
\sum_{j=1}^{dN}(x_j^{s}-y_j^{s})
\right).
\end{align*}
Since the residue characteristic is assumed to be \(0\), we have \(|(d-i)N|=1\). It then follows from the strong triangle inequality that

\[
\left|\frac{a_i}{a_d}-\frac{b_i}{b_d}\right|
\le
C\epsilon
\max_{1\le s\le d-i}
\left|\frac{a_{i+s}}{a_d}\right|.
\]
Letting \(\epsilon\to 0\), we obtain
\(\frac{a_i}{a_d}=\frac{b_i}{b_d}\), which leads to a contradiction.

Therefore, 
\(\frac{a_i}{a_d}=\frac{b_i}{b_d}\) for all \(i\ge 1\). Let $\lambda=\frac{
a_d}{b_d}$ and $\mu=a_0-\frac{b_0a_d}{b_d}$, then $f=\lambda g+\mu$, which proves the claim. 
\end{proof}

\subsection{Group of dynamical symmetries}
We recall from~\cite[\S3]{FG22} the notion of the dynamical symmetry group of a polynomial over a field $k$ with characteristic $0$. 

We begin with monic and centered polynomials.  
A \emph{reduced presentation} of \( f \in k[z] \) is a decomposition
\[
f(z)=z^{\mu} f_0(z^{m}),
\]
where \( \mu,m \ge 0 \), \( f_0(0)\neq 0 \), and \( f_0 \) is not of the form \( h(z^{l}) \) for any polynomial \( h \) with \( \deg h \ge 1 \) and any \( l \ge 2 \).  
Such a presentation is unique.

The \emph{group of dynamical symmetries} of \( f \), denoted by \( \Sigma_f \), is defined by
\[
\Sigma_f=
\begin{cases}
\mathbb{U}_m, & \text{if } f_0\neq 1,\\
\bigcup_{m=1}^{\infty}\bU_m, & \text{if } f_0=1,
\end{cases}
\]
where \( \mathbb{U}_m \) denotes the cyclic group of \( m \)-th roots of unity.

Now let $f$ be an arbitrary polynomial of degree at least $2$, and let $L$ be an affine transformation such that $L^{-1}\circ f \circ L$ is monic and centered.
We define the \emph{group of dynamical symmetries of $f$} by
\[
\Sigma_f
=
\bigl\{L^{-1}\circ \sigma \circ L \big| \sigma \in \Sigma_{L^{-1}\circ f \circ L}\bigr\}.
\]
This definition is independent of the choice of the affine transformation $M$.

\smallskip

A priori, the dynamical symmetry group $\Sigma_f$ is defined in purely algebraic terms.  
However, as a consequence of rigidity phenomena for complex Julia sets, Favre--Gauthier~\cite[Proposition~3.9]{FG22} proved  $\Aut(\mathcal{J}_f)=\Sigma_f$ when $f$ is not a monomial.

\begin{proposition}
\label{prop: cplxsymmetrygroup}
Let $f$ be a complex polynomial of degree $d\ge 2$. Then:
\begin{enumerate}
    \item The group $\Aut(\mathcal{J}_f)$ of affine transformations preserving $\mathcal{J}_f$ coincides with the group of all $\sigma\in \mathrm{Aff}(\C)$ such that the Green function satisfies
    $G_f(\sigma(z)) = G_f(z).$
    
    \item If $f$ is not conjugate to the monomial $M_d(z)=z^d$, then
    $ \Aut(\mathcal{J}_f)=\Sigma_f .$
    Otherwise, $\Aut(\mathcal{J}_{M_d})=\mathbb{S}^1$, whereas $\Sigma_{M_d}$ is the group of roots of unity.
\end{enumerate}
\end{proposition}

Following the same philosophy, we establish an analogous comparison between $\Aut(\mathcal{J}_f)$ and $\Sigma_f$ in the non-Archimedean setting. This will follow from Theorem~\ref{thm: rigidity-samedegree}.

\begin{proposition}
\label{prop: symmetryjulia}
Let $k$ be a non-Archimedean field with residue characteristic $0$. Let $f(z)\in k[z]$ be a polynomial of degree at least $2$.
\begin{enumerate}
    \item  The group $\Aut(\calJ_f)$ coincides with the group of $\sigma\in\mathrm{Aff}(k)$ such that the Green function satisfies $G_f(\sigma(z))=G_f(z)$;
    \item Suppose $f$ has potential good reduction. Let $L\in\mathrm{Aff}(k)$ such that $L(\calJ_f)=x_g$, then \[\Aut(\calJ_f)=\{L^{-1}\circ(az+b)\circ L\mid |a|=1,|b|\le 1 \}.\]
    \item If $f$ does not have potential good reduction, then \[\Aut(\calJ_f)=\Sigma_f.\]
\end{enumerate}
\end{proposition}

\begin{proof}
(1) Let $\sigma \in \Aut(\calJ_f)$. Then
$\calJ_{\sigma^{-1}\circ f\circ \sigma}
= \sigma^{-1}(\calJ_f)
= \calJ_f.$
By \cite[Theorem~10.91]{BR10}, we have
$G_{\sigma^{-1}\circ f\circ \sigma}(z)=G_f(z).$
It follows from Proposition~\ref{prop: basicgreen}(4) that
\[
G_f(\sigma(z))
= G_{\sigma^{-1}\circ f\circ \sigma}(z)
= G_f(z).
\]
Conversely, suppose $\sigma \in \mathrm{Aff}(k)$ satisfies $G_f(\sigma(z))=G_f(z)$ for all $z$.
Then $\sigma$ preserves the filled-in Julia set $\cal{K}_f$, and hence also its boundary, $
\calJ_f=\partial \cal{K}_f.$
Therefore, $\sigma \in \Aut(\calJ_f)$.

\smallskip

(2) The second statement follows from the observation that the group of affine transformations fixing $x_g$ is precisely $\{az+b\mid |a|=1,|b|\le 1\}$.

\smallskip

(3) Up to conjugation, we may assume that $f$ is monic and centered, and write its reduced presentation as
$f(z)=z^{\mu} f_0(z^m).$

By Lemma~\ref{lem: charautojulia}, for any $\sigma \in \Aut(\calJ_f)$ we have
$\calJ_{f\circ \sigma}=\calJ_f.$
By Theorem~\ref{thm: rigidity-samedegree}, there exists an element
\(
\rho(\sigma)\in \Aut(\calJ_f)
\)
such that
\begin{equation}\label{eqq: rigid}
f\circ \sigma=\rho(\sigma)\circ f.
\end{equation}
Moreover, the map $
\rho \colon \Aut(\calJ_f) \longrightarrow \Aut(\calJ_f)$
is a group homomorphism.

We claim that $\Aut(\calJ_f)$ is finite.  
Assuming this claim, it then follows from \cite[Proposition~3.6]{FG22} that
$\Aut(\calJ_f) \subseteq \Sigma_f.$

Conversely, let $\lambda \in \bU_m=\Sigma_f$. A direct computation shows that
\[
(\lambda f)^k(z)
= \lambda^{1+\mu+\cdots+\mu^{k-1}} f^k(z)
\quad \text{for all } k\ge1.
\]
Hence the associated Green functions satisfy $G_{\lambda f}=G_f,$
which implies $
\calJ_{\lambda f}=\calJ_f.$
By Lemma~\ref{lem: charautojulia}, we conclude that $\lambda z \in \Aut(\calJ_f)$.
Therefore, $
\Aut(\calJ_f)=\Sigma_f.$

It remains to prove the claim that $\Aut(\calJ_f)$ is finite.
By~\eqref{eqq: rigid}, each $\sigma\in \Aut(\calJ_f)$ permutes the critical points of $f$.
Since an affine transformation is uniquely determined by the images of two distinct points,
it follows that if $f$ has at least two critical points, then $\Aut(\calJ_f)$ is finite.

If instead $f$ is unicritical, then, since $f$ is monic, centered, and has bad reduction,
it can be written as $
f(z)=z^d+C$ with $C\in k^\times.$
In this case, $\sigma$ fixes the unique critical point $0$, and hence must be of the form
$\sigma(z)=\lambda z$ for some $\lambda\in k^\times$.
Similarly, we may write $\rho(\sigma)(z)=\lambda' z$.
Then~\eqref{eqq: rigid} gives
$(\lambda z)^d + C = \lambda'(z^d + C).$
It follows that $\lambda$ is a root of unity and $\lambda'=1$.
Hence $\Aut(\calJ_f)$ is finite in this case as well.
\end{proof}

\subsection{General rigidity of Julia sets}
 In order to prove Theorem~\ref{IntroRigidity1},  we need the following rigidity theorem for complex polynomials. 

 \smallskip

 Recall that $k$ is an algebraically closed field of characteristic $0$, and let
\[
f(z)=a_dz^d+a_{d-1}z^{d-1}+\cdots
\]
be a polynomial of degree $d\ge 2$ over $k$. A formal series $\phi_f\in k[[z^{-1}]][z]$ is called a Böttcher coordinate if it satisfies
\[
\phi_f\circ f=\phi_f^d,
\]
see~\cite[\S 9]{Mi06} and~\cite[\S 2.5]{FG22}. 

In fact, $\phi_f$ can be written in the form
\[
\phi_f=\alpha\left(z+\frac{a_{d-1}}{d a_d}\right)+\sum_{j\ge1}\alpha_j z^{-j},
\]
where $\alpha^{\,d-1}=a_d$ and $\alpha_j\in \Q(a_i,\alpha)$. Moreover,  $\phi_f$ is uniquely determined up to multiplication by a $(d-1)$-st root of unity.

    \begin{theorem}
    \label{thm: cplxrigidity}
Assume that $\calJ$ is the Julia set of a complex polynomial of degree at least $2$ which is not integrable. Let $g$ be a polynomial of minimal degree
whose Julia set equals $\calJ$.

For any complex polynomial $f$ of degree $d\ge 2$, the following conditions are equivalent:
\begin{enumerate}
\item There exist an integer $n \ge 1$ and $\sigma \in \Aut(\calJ)$ such that $f = \sigma \circ g^{n}.$
\item The Julia sets coincide:
$\calJ_f = \calJ.$
\item The equilibrium measures coincide: $
\mu_f = \mu_g.$
\item The Green functions coincide:
$G_f = G_g.$
\item The Böttcher coordinates at $\infty$ agree up to a root of unity:
$\phi_f = \lambda \phi_g$, where $\lambda$ is a root of unity.
\end{enumerate}
\end{theorem}

\begin{proof}
 The equivalence of (1) and (2) was proved by Schmidt--Steinmetz~\cite{SS95}.  
For the equivalence of (2), (3), and (4), we refer the reader to~\cite{Ran95}.  
The implication (5) $\Rightarrow$ (4) follows from the identity
$G_f=\log|\phi_f|=\log|\phi_g|=G_g,$
see~\cite[\S9]{Mi06}.

It remains to prove that (1) and (4) imply (5).  
Since $\log|\phi_f|=\log|\phi_g|$, there exists $\lambda\in\mathbb{S}^1$ such that
$\phi_f=\lambda\phi_g.$
Up to conjugation, we may assume that $f$ is monic and centered.  
Write
\[
\phi_f(z)=z+\cdots, \qquad \phi_g(z)=\frac{1}{\lambda}z+\cdots .
\]
By Proposition~\ref{prop: cplxsymmetrygroup}, $\sigma$ is given by multiplication by a root of unity. 
Moreover, $\sigma(z)=\lambda^{d-1}z$: indeed, the leading coefficient of 
$g^n=\sigma^{-1}\circ f$ is the $(d-1)$st power of the leading coefficient 
$\lambda$ of the Böttcher coordinate $\phi_g=\phi_{g^n}$. 
It follows that $\lambda^{d-1}$ is a root of unity, and hence $\lambda$ itself is a root of unity.
\end{proof}

\begin{lemma}
\label{lem: reducedform}
Let $k$ be a non-Archimedean field with residue characteristic $0.$ Let $f\in k[z]$ be a polynomial of degree at least $2$ which does not have potential good reduction. Suppose
$\Aut(\calJ_f)=\{\lambda z\in\mathrm{Aff}(k)\mid\lambda\in\bU_m\}$. Then 
there exist an integer $\mu\ge 0$ and a polynomial
$f_0\in k[z]$ such that $
f(z)=z^{\mu} f_0(z^{m}).$
\end{lemma}
\begin{proof} If $m=1$, the conclusion is immediate. We therefore assume that $m\neq 1$.

Let $L$ be an affine transformation such that $L^{-1}\circ f\circ L$ is monic and centered. 
Let $L^{-1}\circ f\circ L (z)= z^{\mu} h_0(z^{m_1})$ be a reduced representation, where $h_0(z)=\sum_{i=0}^{\delta}a_i z^i.$

Applying Proposition~\ref{prop: symmetryjulia} to $L^{-1}\circ f\circ L$, we obtain \[
\Aut(\calJ_f)=\{L\circ \sigma\circ L^{-1}\in\mathrm{Aff}(k)\mid\text{
$\sigma(z)=\lambda z$ for some 
$\lambda\in\bU_{m_1}$}\}.\] Since $\#\Aut(\calJ_f)=m$ by assumption, it follows that $m=m_1$. Thus $L\circ \bU_m\circ L^{-1}=\bU_m.$ Hence $L$ is of the form $L(z)=az.$
Define $f_0(z)=\sum_{i=0}^{\delta} a_ia^{i m+\mu-1} z^i.$ 
Then we have
\[
f(z)= z^{\mu} f_0(z^{m}),
\]
which proves the desired decomposition.
\end{proof}
\begin{lemma}
\label{lem: autjalgebraic}
  Let \(k\) be a non-Archimedean field of residue characteristic \(0\), and let
\(K \subseteq k\) be a subfield. Let \(f \in K[z]\) be a polynomial of degree
at least \(2\). Suppose that \(f\) does not have potential good reduction.

Then for any embedding
$\jmath : K/\bb{Q} \hookrightarrow \mathbb{C}/\bb{Q},$
we have
\[
\jmath\big(\mathrm{Aut}(\mathcal{J}_f)\big)
=
\mathrm{Aut}\big(\mathcal{J}_{\jmath(f)}\big).
\]
\end{lemma}
\begin{proof}
Since $f$ does not have potential good reduction, $\jmath(f)$ is not conjugate to a monomial. Therefore, by Proposition~\ref{prop: cplxsymmetrygroup} and~\ref{prop: symmetryjulia}, one obtains
\[\jmath(\Aut(\mathcal J_f))=\jmath(\Sigma_f)=\Sigma_{\jmath(f)}=\Aut(\mathcal J_{\jmath(f)}),\] 
which completes the proof.
\end{proof}
\newtheorem*{theorem1.1}{Theorem 1.1}
\begin{theorem1.1}
Let \( k \) be an algebraically closed, complete non-Archimedean field with residue characteristic \( 0 \). 
Let \( \mathcal{J} \) be the Julia set of a polynomial of degree at least \( 2 \) which does not have potential good reduction. 
Let \( g \) be a polynomial of minimal degree whose Julia set coincides with \( \mathcal{J} \).

For any polynomial $f$ of degree $d\ge 2$, the following conditions are equivalent:
\begin{enumerate}
\item There exist an integer $n \ge 1$ and $\sigma \in \Aut(\calJ)$ such that $f = \sigma \circ g^{n}.$
\item The Julia sets coincide:
$\calJ_f = \calJ.$
\item The equilibrium measures coincide: $
\mu_f = \mu_g.$
\item The Green functions coincide:
$G_f = G_g.$
\item The Böttcher coordinates at $\infty$ agree up to a root of unity:
$\phi_f = \lambda \phi_g$, where $\lambda$ is a root of unity.
\end{enumerate}
\end{theorem1.1}

\begin{proof}
The equivalence of (3) and (4) follows from \cite[Section 4]{FRL06} and the equivalence of (2) and (3) follows from \cite[Theorem~10.91]{BR10}.  Assume (1) holds, then Lemma~\ref{lem: charautojulia} implies that $\mathcal J_f=\mathcal J_g.$ Also, by~Lemma~\ref{lem: autjalgebraic}, one has
\[
\jmath(f)=\jmath(\sigma)\circ \jmath(g)^n 
\]
with $\jmath(\sigma)\in \Aut(\cal{J}_{\jmath(g)}).$
By Theorem~\ref{thm: cplxrigidity}, the Böttcher coordinates $\phi_{\jmath(f)}$ and $\phi_{\jmath(g)}$
coincide up to multiplication by a root of unity.  
Consequently the same holds for $\phi_f$ and $\phi_g$.

\smallskip

Next we prove (5) implies (1).
Let $K$ be a field finitely generated over $\mathbb Q$ such that
$f,g,\phi_f,$ and $\phi_g$ are defined over $K$.  
Let $\jmath\colon K/ \bb{Q}\hookrightarrow \mathbb C/\bb{Q}$ be an embedding.  
Since $g$ does not have potential good reduction, $\jmath(g)$ is not conjugate to a monomial.

If $\phi_f=\lambda\phi_g$ for some root of unity $\lambda$, then
$\phi_{\jmath(f)}=\lambda\phi_{\jmath(g)}$.  
By Theorem~\ref{thm: cplxrigidity}, there exists a polynomial $Q\in\mathbb C[z]$ such that
\[
\jmath(f)=\sigma_1\circ Q^{l_1}, \qquad 
\jmath(g)=\sigma_2\circ Q^{l_2},
\]
for some $\sigma_i\in\Aut(\mathcal J_{\jmath(g)})$ and integers $l_i\ge1$. Therefore, 
\[
f=\jmath^{-1}(\sigma_1)\circ \jmath^{-1}(Q)^{l_1}, \qquad 
g=\jmath^{-1}(\sigma_2)\circ \jmath(Q)^{l_2},
\]

By Lemma~\ref{lem: autjalgebraic}, one has $j^{-1}(\sigma_i)\in \Aut(\calJ).$
It then follows from Lemma~\ref{lem: charautojulia} that
\[
\mathcal J_f=\mathcal J_g=\mathcal J_{\jmath^{-1}(Q)} .
\]

Since $g$ was chosen to have minimal degree, we must have $\deg g=\deg Q$.  
Consequently $\deg f=\deg(g^{\,l_1})$.  
Applying Theorem~\ref{thm: rigidity-samedegree} to $f$ and $g^{l_1}$ yields
\[
f=\sigma\circ g^{\,l_1}
\]
for some $\sigma\in\Aut(\mathcal J)$.

\smallskip

It remains to prove (3) implies (5).
We may assume that $f$ is monic and centered.  Let $f(z)=z^{\mu}f_0(z^m)$ be a reduced representation.
By Proposition~\ref{prop: symmetryjulia}(3),
$\Aut(\mathcal J)=\{\lambda z\in \mathrm{Aff}(k)\mid \lambda\in \mathbb U_m\}$
for some $m\in\mathbb Z$.

Note that
$(f\circ g)^*(\mu_f)=\deg(f\circ g)\mu_f.$
By Theorem~\ref{thm: equidistribution}, we obtain $\mu_{f\circ g}=\mu_g$, and hence
$\mathcal J_{f\circ g}=\mathcal J.$ Similarly, $\mathcal J_{g\circ f}=\mathcal J$.  
Applying Theorem~\ref{thm: rigidity-samedegree} to $f\circ g$ and $g\circ f$, we deduce that
there exists $\zeta\in\mathbb U_m$ such that
\begin{equation}
\label{eqq: quaspermute}
f\circ g=\zeta g\circ f .
\end{equation}

By Lemma~\ref{lem: reducedform}, we may write $g(z)=z^{\nu}g_0(z^m).$
Define $\hat f(z)=z^{\mu}(f_0(z))^m, $ and $\hat g(z)=z^{\nu}(g_0(z))^m .$
Then \eqref{eqq: quaspermute} implies
\[
\hat f\circ \hat g=\hat g\circ \hat f .
\]

Thus $\jmath(\hat f)$ and $\jmath(\hat g)$ commute.  
By the classical results of Julia~\cite{Jul22}, Fatou~\cite{Fat24}, Ritt~\cite{Rit20}, one of the following occurs.

\smallskip

\noindent
(1) $\jmath(\hat f)$ is a monomial, or $\pm \jmath(\hat f)$ is conjugate to a Chebyshev polynomial.  
In this case $f$ would also be conjugate to a Chebyshev polynomial, which has potential good reduction, contradicting our assumption.

\smallskip

\noindent
(2) There exist integers $n_1,n_2\ge1$ such that $
\jmath(\hat f^{n_1})=\jmath(\hat g^{n_2}).$
Since $\deg\hat f=\deg f$ and $\deg\hat g=\deg g$, this implies
\[
\deg(f^{n_1})=\deg(g^{n_2}).
\]
Applying Theorem~\ref{thm: rigidity-samedegree} to the iterates
\(f^{n_1}\) and \(g^{n_2}\), we obtain some
\(\zeta'\in\mathbb U_m\) such that
$f^{n_1}=\zeta' g^{n_2}.$
Applying the implication \((1)\Rightarrow(5)\) to \(\zeta' g^{n_2}\), we then get
\[
\phi_f=\phi_{\zeta' g^{n_2}}=\lambda \phi_g
\]
for some root of unity \(\lambda\). This proves the statement.
\end{proof}

\section{Uniform bounds on common preperiodic points} \label{sec: UniformBound}
In this section, we will prove Theorem \ref{thm: uniformprep}. 

We will follow the strategy of DeMarco--Krieger--Ye \cite{DKY20, DKY21}.

Let \(f,g\in K[z]\) be polynomials of degrees \(d_1,d_2\ge 2\), respectively, defined over a number field \(K\). Suppose further that \(f\) is monic and centered. For each place $\nu \in M_K$, we may define a local energy pairing 
$$\langle f,g \rangle_{\nu} = -\iint_{\bb{P}_{\nu}^{1,\an} \times \bb{P}_{\nu}^{1,\an}} \log |x-y|_{\nu} d (\mu_{f,\nu} - \mu_{g,\nu}) (x) d (\mu_{f,\nu} - \mu_{g,\nu})(y).$$
We can then define a global energy pairing by 
$$\langle f , g \rangle = \sum_{\nu \in M_K} N_{\nu} \langle f , g \rangle_{\nu}$$
where $N_{\nu} = \frac{[K_{\nu}:\bb{Q}_p]}{[K:\bb{Q}]}$. Let  $h(f),h(g)$ be the standard Weil heights on the coefficients of $f$ and $g$. Then to establish Theorem \ref{thm: uniformprep}, we have to prove the existence of constants $c_1,c_2 >0$, depending only on $d_1$ and $d_2$, so that
\begin{equation} \label{eq: Global1}
\langle f , g \rangle \geq c_1( h(f)+h(g)) + c_2.
\end{equation}
This would follow from suitable uniform lower bounds on the local energy pairing $\langle f,g \rangle_{\nu}$, which in turn would follow from Theorem \ref{IntroRigidity1} using a degeneration argument.

\subsection{Uniform bounds for local pairings}
We will  begin by proving uniform lower bounds for $\langle f , g \rangle_{\nu}$.

Write $f = \sum_{i=0}^{d_1} a_i z^i$ and $g = \sum_{i=0}^{d_2} b_i z^i$. Recall the following notation: $|f|=\max\{|a_i|\}$; $|(f,g)| = \max\{|a_i|,|b_j|,|b_{d_2}|^{-1}\}$;
 \[|\Res(f)| =\frac{|a_d|^d}{\max_{0\le i\le d}\{1,|a_i|\}^{2d}};\]
 and
\[
\epsilon(f_n)
=
\bigl(-\log(|\Res(f_n)|/C_d)\bigr)^{-1},
\]
where \( C_d = e \, \sup_{\mathrm{Rat}_d(\mathbb{C})} |\Res| \) if \( k_n=\mathbb{C} \),
and \( C_d = 1 \) otherwise.

\begin{lemma}
\label{lem: mcminres}
Let \(k\) be a non-Archimedean field with residue characteristic \(0\).
Let \(f\) be a monic centered polynomial of degree \(d\ge 2\).
Then \(f\) has good reduction if and only if \(f\) has potential good reduction.
\end{lemma}

\begin{proof}
The implication ``good reduction \(\Rightarrow\) potential good reduction'' is obvious.
We prove the converse.

Write $f(z)=z^{d}+a_{d-2}z^{d-2}+\cdots+a_0.$ Assume that \(f\) has potential good reduction.
Then there exists a  type II fixed point $x=\zeta(a,R)$
with local degree $d$. Let \(M(z)=\lambda z+a\) with \(|\lambda|=R\).
Define $g=M^{-1}\circ f\circ M.$
A direct computation shows that
$g(z)
=
\lambda^{d-1}z^{d}
+
d a \lambda^{d-2} z^{d-1}
+\cdots .$

Since $\tilde{g}$ has degree $d$, we conclude  $|\lambda|^{d-1}=1,$ and 
$|d a \lambda^{d-2}|\le 1.$
Because the residue characteristic of \(k\) is \(0\), we have \(|d|=1\), and therefore $
|a|\le 1.$ It follows that \(x=\zeta(0,1)\), and hence \(f\)  has good reduction.
\end{proof}


Let $V$ be a Zariski-closed subvariety of $\mathrm{Poly}_{\mathrm{mc}}^{d_1}$ of positive dimension. Then by \cite[Proposition 3.9]{FG22}, there exists a positive integer $m \geq 1$ and a proper subvariety $V' \subset V$ such that for all $f \in (V \setminus V')(\ovl{\bb{Q}})$, the group of dynamical symmetries is $\mathbb{U}_m$. We can describe explicit equations cutting out $V'$ as follows. 
\par 
For any positive integer $m$, let $V_m$ be the set of monic centered polynomials $f(z) = z^{d_1} + \sum_{i=0}^{d_1-2} a_i z^i$ such that $m\mid \#\Sigma_f.$
Then
\[
V_m 
= \bigcup_{m\mid d_1-\mu } 
\left\{ a_i = 0 \;\middle|\; m \nmid (
i-\mu) \right\}
\] 
defines a Zariski closed subvariety of $\mathrm{Poly}_{\mathrm{mc}}^{d_1}.$
Set $V' = (\bigcup_{m'\nmid m} V_{m'})\cup (\bigcup_{i \geq 1} V_{im}).$
Then $V'$ is a Zariski-closed subvariety defined by finitely many monomials $p_1,\dots,p_\ell$.

For any metrized field $k$ of characteristic $0$ and any $f\in k[z]$,  we define
\[
\lambda_{V',k}(f) 
= - \log \max \{ |p_1(f)|, \dots, |p_\ell(f)| \}.
\]
Then $\lambda_{V',k}=+\infty$ if and only if $f\in V'(k)$. 

By construction, we have that $f \in (V\setminus V')(k)$ if and only if
$\Sigma_f = \mathbb{U}_m.$ 

\newtheorem*{theorem1.6}{Theorem 1.6}
\begin{theorem1.6} 
Let \((k,|\cdot|)\) be a field of characteristic \(0\) equipped with a
nontrivial absolute value, and fix \(C_1,C_2>0\). Then there exist constants
\(C_3,C_4>0\) depending only on $k$ such that the following holds:

Suppose that
\[
(f,g)\in (V\setminus V')(k)\times \operatorname{Poly}^{d_2}(k)
\]
satisfies
\[
\lambda_{V'}(f)\leq C_1\log |(f,g)|.
\]
If positive integers \(n_1,n_2\) exist with \(d_1^{n_1}=d_2^{n_2}\), assume
further that
\[
\log |f^{n_1}-\lambda g^{n_2}|
\geq -C_2\log |(f,g)|
\quad\text{for all } \lambda\in \mathbb{U}_m.
\]
Then
\[
\log |(f,g)|\geq C_3
\quad\Longrightarrow\quad
\langle f,g\rangle \geq C_4\log |(f,g)|.
\]
\end{theorem1.6}


\begin{proof}
We argue by contradiction. Suppose that there exists a sequence
\((f_n,g_n) \in V \setminus V'\times \mathrm{Poly}^{d_2}(k)\) such that, upon setting
\[
\epsilon_n := \Bigl(\log|(f_n,g_n)|\Bigr)^{-1},
\]
the following properties hold:
\begin{enumerate}
\item $\lambda_{V',k}(f_n) \le C_1\epsilon_n^{-1}$;
    \item $\log |f_n^{n_1} - \lambda g_n^{n_2}| \ge -C_2\epsilon_n^{-1}$ for all $\lambda \in \bU_m$;
    \item $\epsilon_n^{-1} \ge n$;
    \item $\langle f_n, g_n \rangle \le \frac{1}{n}\epsilon_n^{-1}$.
\end{enumerate}
By construction, $\lim_{n\to\infty}\epsilon_n = 0$. 

Fix a non-principal ultrafilter $\omega$.  
By Proposition~\ref{prop: prop-residuefield}, the field
\[
\sH(\omega)
=
\Bigl\{(z_n)\in k^{\bN}\bigm| \lim_{\omega}|z_n|^{\epsilon_n}<\infty\Bigr\}
\Big/
\Bigl\{(z_n)\in k^{\bN}\bigm|\lim_{\omega}|z_n|^{\epsilon_n}=0\Bigr\}
\]
is a non-Archimedean field. Moreover, for any integer $m$ we have
$\lim_{\omega}|m|^{\epsilon_n}=1.$
It follows that the residue field of $\sH(\omega)$ has characteristic $0$.

By construction, the sequences \((f_n)\) and
\((g_n)\) induce polynomial maps \(f_\omega\) and \(g_\omega\) of degrees
\(d_1\) and \(d_2\), respectively, over \(\sH(\omega)\).
Moreover, \begin{align}
\label{eqq: equalitycoeff}
    \max\{|a_{i,\omega}|,|b_{j,\omega}|,|b_{d_2,\omega}|^{-1}|\}=e.
\end{align}

Passing to the ultralimit, we obtain
$\lambda_{V',\sH(\omega)}(f_\omega) \le C_1<\infty$  and $\log |f_{\omega}^{n_1} - \lambda g_\omega^{n_2}| \ge -C_2>-\infty$.
Thus we have \begin{align}
\label{eqq: nosymmetry}
  \Sigma_f=\bU_m,~\text{and $f_\omega^{n_1} \neq \lambda g_\omega^{n_2}$ for all
$\lambda\in\bU_m$.}    
\end{align}
Furthermore, by Theorem~\ref{thm: cts-energy}, one has $\langle f_\omega, g_\omega \rangle
= \lim_{\omega} \epsilon_n\langle f_n, g_n \rangle
= 0.$
It then follows from \cite[Proposition 4.5]{FRL06} that $\mu_{f_\omega}=\mu_{g_\omega}$. Hence $\calJ_{f_\omega}=\calJ_{g_\omega}.$
 In particular, $f_\omega$ has good reduction if and only if $g_\omega$ has good reduction. 
If so, $\max\{|a_{i,\omega},|b_{j,\omega}|,|b_{d_2,\omega}|^{-1}|\}=1$, which contradicts \eqref{eqq: equalitycoeff}. Hence, both $f_{\omega}$ and $g_{\omega}$ does not have good reduction.
  
 Since $f_\omega$ is monic and centered, Lemma~\ref{lem: mcminres} implies that $f_\omega$ does not have potential good reduction. Therefore, by Theorem~\ref{thm: rigidity-samedegree} and Proposition~\ref{prop: symmetryjulia}(3), there exists $\lambda\in\bU_m$ such that
$f_\omega^{n_1}=\lambda g_\omega^{n_2},$
which contradicts~\eqref{eqq: nosymmetry}.
\end{proof}
The proof of Theorem~\ref{IntroUniform3} follows the same lines as that of Theorem~\ref{IntroUniform2}, with only minor modifications. 

\newtheorem*{theorem1.7}{Theorem 1.7}
\begin{theorem1.7} 
For every \(C_1,C_2>0\), there exist constants \(C_5,C_6>0\) depending only on $d_1$ and $d_2$ such that the following holds: Let \((k,|\cdot|)\) be a non-Archimedean field whose residue characteristic is at least \(C_5\), and let
\[
(f,g)\in (V\setminus V')\times \operatorname{Poly}^{d_2}(k)
\]
satisfy
\[
\lambda_{V',k}(f)\le C_1\log |(f,g)|.
\]
Suppose moreover that whenever there exist positive integers \(n_1,n_2\) with
\[
d_1^{n_1}=d_2^{n_2},
\]
the following additional condition is imposed:
\[
\log |f^{n_1}-\lambda g^{n_2}|
\ge -C_2\log |(f,g)|
\qquad\text{for every } \lambda\in \mathbb U_m.
\]
 Then
\[
\langle f,g\rangle \ge C_6\log |(f,g)|.
\]
\end{theorem1.7}

\begin{proof}[Proof of Theorem~\ref{IntroUniform3}]
For the sake of contradiction, we obtain a sequence of fields $(k_n,|\cdot|_n)$ with residue characteristic at least $n$, and a sequence
$(f_n,g_n) \subset V \setminus V'\times \mathrm{Poly}^{d_2}(k_n)$. Setting \[
\epsilon_n := \Bigl(\log |(f_n,g_n)|_n\Bigr)^{-1},
\]
one has
\begin{enumerate}
\item $\lambda_{V',k_n}(f_n)
\le C_1\epsilon_n^{-1}$;
    \item $\log |f_n^{n_1}-\lambda g_n^{n_2}|_n
\ge -C_2\epsilon_n^{-1}$ for all $\lambda \in \bU_m$;

\item $\langle f_n,g_n\rangle< \frac{1}{n}\epsilon_n^{-1}$.
\end{enumerate}
By~Proposition~\ref{prop: prop-residuefield},
the field
 \[\sH(\omega)=\left\{(z_n)\in \prod_{n\in\bN} k_n\mid \lim_{\omega}|z_n|_n^{\epsilon_n}<\infty\right\}/\left\{ (z_n)\in \prod_{n\in\bN} k_n\mid\lim_{\omega}|z_n|_n^{\epsilon_n}=0\right\}\]
is a non-Archimedean field. Since we have assumed $\mathrm{char}(\tilde{k}_n)\ge n$,  for any fixed integer $a$, $\lim_\omega|a|_n^{\epsilon_n}=1$. Therefore, $\sH(\omega)$ has residue characteristic \(0\).

As in the proof of Theorem~\ref{IntroUniform2}, \(f_\omega\) and \(g_\omega\)
are maps of degrees \(d_1\) and \(d_2\), respectively; at least one of them does not
have good reduction, and
\[
\Aut(f_\omega)=\bU_m,
\qquad
f_\omega^{n_1}\neq \lambda g_\omega^{n_2}~ \text{for all }\lambda\in\bU_m,
\qquad
\mu_{f_\omega}=\mu_{g_\omega}.
\]

The equality of equilibrium measures implies that neither \(f_\omega\) nor \(g_\omega\)
has good reduction. By Lemma~\ref{lem: mcminres}, it follows that neither map has potential good reduction. Applying Theorem~\ref{thm: rigidity-samedegree}, we obtain a contradiction.
\end{proof}

\subsection{Uniform bounds for global energy pairings}
We now begin the proof of \eqref{eq: Global1}.
\subsubsection{Places and  height}We refer the reader to~\cite{HS00} for a general reference. 

Let $M_{\mathbb{Q}}$
denote the set consisting of all prime numbers together with the symbol
$\infty$. We write $|\cdot|_{\infty}$ for the usual absolute value on
$\mathbb{Q}$, and for each prime number $p$ we denote by $|\cdot|_{p}$ the
$p$-adic norm, normalized so that $|p|_{p}=p^{-1}$.

Let $K$ be a number field, that is, a finite extension of
$\mathbb{Q}$. We define $M_{K}$ to be the set of all multiplicative
norms on $K$ extending one of the absolute values
$|\cdot|_{\nu}$ with $\nu \in M_{\mathbb{Q}}$. Elements of $M_{K}$ are
called \emph{places} of $K$. For each $\nu \in M_{K}$, we
 denote by $|\cdot|_{\nu}$ the corresponding multiplicative norm on
$K$. A place $\nu$ is said to be \emph{Archimedean} if its restriction
to $\mathbb{Q}$ coincides with $|\cdot|_{\infty}$, and
\emph{non-Archimedean} otherwise.

Each place $\nu$ gives rise to a local field. More precisely, let
$K_{\nu}$ be the completion of $K$ with respect to
$|\cdot|_{\nu}$, and let $\mathbb{Q}_{\nu}$ be the closure of $\mathbb{Q}$ inside
$K_{\nu}$. We denote by $\mathbb{C}_{\nu}$ the completion of an
algebraic closure of $K_{\nu}$. The field $\mathbb{C}_{\nu}$ is then
complete and algebraically closed. Moreover, for every Archimedean place
$\nu$ (denoted by $\nu\mid\infty $), one has an isomorphism $\mathbb{C}_{\nu} \simeq \mathbb{C}$; and for every non-Archimedean place
$\nu$ over $p$ (denoted by $\nu\mid p $), one has an isomorphism $\mathbb{C}_{\nu} \simeq \mathbb{C}_p$.

Equivalently, a place $\nu \in M_{K}$ extending a place
$v \in M_{\mathbb{Q}}$ can be described by an embedding $\sigma_{\nu} \colon K \hookrightarrow \mathbb{C}_{v}$
such that $|a|_{\nu} = |\sigma_{\nu}(a)|_{v},$ for all $a \in K$.

\smallskip

We associate to each place $\nu \in M_{K}$ the weight $N_{\nu} := \frac{[K_{\nu} : \mathbb{Q}_{\nu}]}{[K : \mathbb{Q}]}.$
A fundamental fact is that for every $\alpha \in K^{\ast}$,
all but finitely many places $\nu \in M_{K}$ satisfy
$|\alpha|_{\nu} = 1$, and the following product formula holds: \[
\prod_{\nu \in M_{K}} |\alpha|_{\nu}^{N_{\nu}} = 1 .\]

This leads naturally to the notion of height. For $\alpha \in K$, we
define its  height by
\[
h(\alpha) := \sum_{\nu \in M_{K}} N_{\nu} \log^{+} |\alpha|_{\nu}.
\]
By the product formula, one has $h(\alpha) = h(\alpha^{-1})$ for
$\alpha \neq 0$. Moreover, this definition is independent of the choice of the
number field $K$ containing $\alpha$.

\smallskip

Recall that $|f|_{\nu}=\max\{|a_i|_{\nu}\}$ and $ |(f, g)|_{\nu}
=
\max
\bigl\{ |a_i|_{\nu}, |b_j|_{\nu},|b_{d_2}|_{\nu}^{-1} \bigr\}.$
We define the height for polynomials and polynomial pairs: 
\begin{align*}
h(f)
&:=
\sum_{\nu \in M_{K}}
N_{\nu}\,\log^+ |f|_{\nu}.\\
    h(f,g)
&:=
\sum_{\nu \in M_{K}}
N_{\nu}\,\log^+ |(f, g)|_{\nu}.
\end{align*}

Similarly, the definitions do not depend on the choice of $K.$

\smallskip

We shall use the following lemma in the sequel. 
\begin{lemma}
\label{lem: height}
Let $f,g$ be two polynomials (possibly constant) defined over $\overline{\Q}$. Then 
\begin{enumerate}
    \item $h(f,g)\le h(f)+2h(g)\leq 3h(f,g)$;
    \item  $h(f+g)\le h(f)+h(g)+\log 2$
    \item  Fix two positive integers $n_1$ and $n_2$. Then there exists a constant $R=R(n_1,n_2)>2$ such that 
    \[h(f^{n_1},g^{n_2})\le Rh(f,g)+R\]
\end{enumerate}
\end{lemma}

\begin{proof}
   Note that
\begin{align*}
h(f,g)
&=\sum_{\nu \in M_{K}}
N_{\nu}\,\log \max\{|a_i|_{\nu},|b_j|_{\nu},|b_{d_2}|_{\nu}^{-1}\}\\
&\le
\sum_{\nu\in M_K}N_{\nu}\,\log \max\{|a_i|_{\nu}\}
+\sum_{\nu\in M_K}N_{\nu}\,\log \max\{|b_i|_{\nu}\}
+\sum_{\nu\in M_K}N_{\nu}\,\log \max\{|b_{d_2}|_{\nu}^{-1}\}\\
&= h(f)+h(g)+h(b_{d_2}^{-1})\\
&= h(f)+h(g)+h(b_{d_2})
\le h(f)+2h(g).
\end{align*}
On the other hand, by definition, $h(f)\le h(f,g)$ and $
h(g)\le h(f,g).$
Therefore,
\[
h(f)+2h(g)\le 3h(f,g),
\]
which proves the first assertion.

The second and third assertions then follow immediately from
\cite[Proposition~5]{HS11} together with the first assertion.
\end{proof}

\subsubsection{Proof of the lower bounds for global energy pairing}
Fix two integers $d_1,d_2\ge 2$. Let $f$ and $g$ be two polynomials of degree $d_1$ and $d_2$ respectively over $K$,
$f(z) = \sum_{i=0}^{d} a_i z^i,$ and $
g(z) = \sum_{i=0}^{d} b_i z^i.$
For any place $\nu$ of $K$, the polynomials $f$ and $g$ can be viewed as polynomials over $\C_\nu$.

For all but finitely many places $\nu$, one has
$\lvert \operatorname{Res}(f) \rvert_\nu
=
\lvert \operatorname{Res}(g) \rvert_\nu
=
1,$
so that both $f_\nu$ and $g_\nu$ have good reduction at $\nu$. In this case, their
associated equilibrium measures coincide, and consequently the  local energy pairing
\(\langle f, g \rangle_\nu\) introduced in \S\ref{sec: energypairing} vanishes.

This allows us to define a \emph{global energy pairing} by summing the local
contributions:
\[
\langle f, g \rangle
:=
\sum_{\nu \in M_{K}}
N_{\nu}\langle f, g \rangle_{\nu}.
\]

As before, the definition does not depend on the choice of the number field
$K$. We are now ready to prove the first uniform lower bound on our global energy pairing. 
\begin{theorem}
\label{thm: glbenergybd1}
   Let \(d_1,d_2\ge 2\) be fixed integers. Then there exist constants
\(A=A(d_1,d_2)>0\) and \(B=B(d_1,d_2)>0\) such that, for every pair
$
(f,g)\in \mathrm{Poly}_{\mathrm{mc}}^{d_1}
\times \mathrm{Poly}^{d_2}(\overline{\mathbb Q}),$
if
\[
f^{n_1}\neq \lambda\, g^{n_2}
\quad
\text{for all } \lambda\in \Sigma_f
\text{ and all } n_1,n_2\in \mathbb Z_{>0},
\]
then
\[
\langle f,g\rangle \ge A\,h(f,g)-B .
\]
\end{theorem}

\begin{proof}[Proof of Theorem~\ref{thm: glbenergybd1}]
If $f$ is monomial, the result follows from \cite[Corollary 1.4]{Obe24}. 

Let \( V_m \subset \mathrm{Poly}_{\mathrm{mc}}^{d_1}(\overline{\mathbb{Q}}) \) denote the locus of polynomials whose group of dynamical symmetries is equal to \( \bU_m \). Note that
\[
\mathrm{Poly}_{\mathrm{mc}}^{d_1}(\overline{\mathbb{Q}})\setminus \{z^{d_1}\}
= \bigsqcup_{m=1}^{d_1} \left( V_m \setminus \bigcup_{i=2}^{d_1} V_{im} \right).
\]
It therefore suffices to treat the case where
$f \in V \setminus V'$,
with  $V = V_m$ and $V' = \bigcup_{i=2}^{d_1} V_{im}.$ By construction, for any \( f \in V \setminus V' \), we have \( \Sigma_f = \bU_m \).

Recall that \( V' \) is defined by a finite collection of monic monomials \( p_1, \ldots, p_\ell \). Write $f(z)=\sum_{n=0}^{d_1}a_nz^n$ and $g(z)=\sum_{n=0}^{d_2}b_nz^n$.
Let $K$ be the number field generated by $a_i$ and $b_j$ for all $0\le i\le d_1$ and $0\le j \le d_2$.  For any $\nu\in M_K$, we define 
\[\lambda_{V',\nu}(f)= - \log \max \{|p_1(f)|_\nu,\ldots,|p_{\ell}(f)|_\nu\}.\]
Set
$M := \max_{1 \le i \le \ell} \deg p_i.$ 

\medskip

We divide the proof into the following steps.

\noindent\emph{Step 1.} Suppose there exist positive integers $n_1,n_2$ such that $D=d_1^{n_1}=d_{2}^{n_2}.$
Let $R$ be the constant appearing in  Lemma~\ref{lem: height}. Let $R_1=2R+\log 2.$

We say a place $\nu\in M_{K}$ is good if the following statements hold:
\begin{enumerate}
  \item $\lambda_{V',\nu}(f)\leq 4mR_1(D+1)M\log|(f,g)|_\nu$;
    \item $\min_{\lambda\in\bU_m}\log|f^{n_1}-\lambda g^{n_2}|_\nu\geq -4mR_1(D+1)\log|(f,g)|_\nu$.
\end{enumerate}

We  prove that
\begin{align}
\label{eqq: ineq0}
    \sum_{\nu\text{~good}}N_\nu\log|(f,g)|_\nu\geq\frac{1}{2}h(f,g)-\frac{1}{2}.
\end{align}

\smallskip

Otherwise, since $\sum_{\nu\in M_K}N_\nu\log|(f,g)|_\nu=h(f,g)$, we obtain
\begin{align}
\label{eqq: badmany}
    \sum_{\nu\text{~bad}}N_\nu\log|(f,g)|_\nu\geq\frac{1}{2}h(f,g)+\frac{1}{2}.
\end{align}

 Write $f^{n_1}(z)=\sum_{n=0}^{D}a_n^{(D)}z^n$ and $g^{n_2}(z)=\sum_{n=0}^{D}b_n^{(D)}z^n$. Let $\nu$ be a bad place. Then either (1) or (2) does not hold. Suppose (1) does not hold. Note that there exists $p_i$, such that \[\lambda_{V',\nu}(f)=-\log|p_i(f)|_\nu\leq M\max_{a_i\not=0}\log^+\{|a_i|^{-1}_\nu\}.\]
Therefore, 
\[4mR_1(D+1)\log|(f,g)|_{\nu}<\max_{a_i\not=0}\log^+\{|a_i|_\nu^{-1}\}.\] If (2) does not hold, then \[4mR_1(D+1)\log|(f,g)|_{\nu}<\max_{\lambda\in\bU_m}\log\min_{0\le i\le D}\{|a_i^{(D)}-\lambda b_i^{(D)}|^{-1}\}.\]

Using~\eqref{eqq: badmany} and summing up over all bad places, one obtains
\begin{align*}
2m(D+1)(R_1h(f,g)+R_1)&\leq\sum_{\nu~\text{bad}} N_\nu\max_{\lambda\in\bU_m,~a_j\not=0}\max\{\log^+\min_{0\le i\le D}\{|a_i^{(D)}-\lambda b_i^{(D)}|_\nu^{-1}\},\log^{+}\{|a_j|_\nu^{-1}\}\}\\
&< \sum_{\nu\in\mathrm{M}_K} N_\nu\max_{\lambda\in\bU_m,~a_j\not=0}\max\{\log^+\min_{0\le i\le D}\{|a_i^{(D)}-\lambda b_i^{(D)}|_\nu^{-1}\},\log^{+}\{|a_j|_\nu^{-1}\}\},
\end{align*}
which is at most 
\[\left(\sum_{\lambda\in\bU_m,a_j\not=0}\sum_{\nu\in M_K}N_\nu\log^+\min_{0\le i\le D}\{|a_i^{(D)}-\lambda b_i^{(D)}|_\nu^{-1}\}\right)+\left(\sum_{a_j\not=0}m N_\nu\log^{+}\{|a_j|_\nu^{-1}\}\right).\]
By the pigeonhole principle,  either there exist $\lambda\in\bU_m,$ $a_i^{(D)}\not= \lambda b_i^{(D)}$ such that 
\[R_1h(f,g)+R_1< \sum_{\nu\in\mathrm{M}_K}N_\nu\log^+\{|a_i^{(D)}-\lambda b_i^{(D)}|_\nu^{-1}\}=h(a_i^{(D)}-\lambda b_i^{(D)})\]
or there exists $a_i\not=0$, 
such that
\[R_1h(f,g)+R_1< \sum_{\nu\in\mathrm{M}_K}N_\nu\log^+\{|a_i|_\nu^{-1}\}=h(a_i).\]
 By Lemma~\ref{lem: height}, we have
$R_1h(f,g)+R_1\ge h(a_i)$
and 
\[h(a_i^{(D)}-\lambda b_i^{(D)})\le h(a_i^{(D)})+h(b_i^{(D)})+\log 2 \le 2h(f^{n_1},g^{n_2})+\log 2\le 
R_1h(f,g)+R_1.\]
This yields a contradiction.

\medskip

\noindent\emph{Step 2.} Suppose there do not exist positive integers $n_1$ and $n_2$ such that $d_1^{n_1}=d_2^{n_2}$.

We say a place $\nu\in M_{K}$ is good if
  $\lambda_{V',\nu}(f)\leq 2(d+1)M\log|(f,g)|_\nu,$ where $d=\max\{d_1,d_2\}$.

We  prove that
\begin{align}
\label{eqq: ineq00}
    \sum_{\nu\text{~good}}N_\nu\log|(f,g)|_\nu\geq\frac{1}{2}h(f,g)-\frac{1}{2}.
\end{align}

\smallskip

Otherwise, since $\sum_{\nu\in M_K}N_\nu\log|(f,g)|_\nu=h(f,g)$, we obtain
\begin{align}
\label{eqq: badmanyy}
    \sum_{\nu\text{~bad}}N_\nu\log|(f,g)|_\nu\geq\frac{1}{2}h(f,g)+\frac{1}{2}.
\end{align}

 Let $\nu$ be a bad place.  Note that there exists $p_i$, such that \[\lambda_{V',\nu}(f)=-\log|p_i(f)|_\nu\leq M\max_{a_i\not=0}\log^+\{|a_i|^{-1}_\nu\}.\]
Therefore, 
\[2(d+1)\log|(f,g)|_{\nu}<\max_{a_i\not=0}\log^+\{|a_i|_\nu^{-1}\}.\]
Using~\eqref{eqq: badmanyy} and summing up over all bad places, one obtains
\begin{align*}
(d+1)(h(f,g)+1)&\leq\sum_{\nu~\text{bad}} N_\nu\max_{a_i\not=0}\{\log^{+}\{|a_i|_\nu^{-1}\}\}
\end{align*}
By the pigeonhole principle, there exists $a_i\not=0$, 
such that
\[h(f,g)+1< \sum_{\nu\in\mathrm{M}_K}N_\nu\log^+\{|a_i|_\nu^{-1}\}=h(a_i).\]
Since $h(f,g)\ge h(a_i)$, we obtain a contradiction.

\medskip

\noindent\emph{Step 3. Conclusion.}
If there exist positive integers $n_1,n_2$ such that $D=d_1^{n_1}=d_{2}^{n_2}$.
Let \(C_1 = 4mR_1(D+1)M\) and \(C_2 = 4mR_1(D+1)\), otherwise take $C_1=2(\max\{d_1,d_2\}+1)M$. Let \(C_5\) be the constant appearing in Theorem~\ref{IntroUniform3}. For any good place \(\nu\), we distinguish two cases according to whether the residue characteristic of \(\C_\nu\) is \(<C_5\) or \(\geq C_5\).

In the former case, since $\C_\nu$ has finitely many choices up to isometry, by Theorem~\ref{IntroUniform2}, there exist uniform constants $A_1,B_1>0$, such that 
\[\langle f,g\rangle_\nu\geq A_1\log|(f,g)|_\nu-B_1.\]
Note that for any prime number $p$, one has $\sum_{\nu \mid p }N_\nu=1.$ Therefore, 
\begin{align}
\label{eqq: ineq1}
    \sum_{\substack{
\nu\ \text{good} \\
\mathrm{char}(\tilde{\C}_\nu)< C_5
}}
N_\nu\langle f,g\rangle_{\nu}
\ge
- B_1 C_5
+
A_1 \sum_{\substack{
\nu\ \text{good} \\
\mathrm{char}(\tilde{\C}_\nu)< C_5
}}
N_\nu \log|(f,g)|_{\nu} .
\end{align}

In the latter case, by Theorem~\ref{IntroUniform3}, there  exist uniform constants $A_2>0$, such that 
\[\langle f,g\rangle_{\nu}\geq A_2\log|(f,g)|_{\nu}.\]
Therefore, 
\begin{align}
\label{eqq: ineq2}
   \sum_{\substack{
\nu\ \text{good} \\
\mathrm{char}(\tilde{\C}_\nu)\ge C_5
}}
N_\nu\langle f,g\rangle_\nu
\ge
 A_2 \sum_{\substack{
\nu\ \text{good} \\
\mathrm{char}(\tilde{\C}_\nu)\ge C_5
}}
N_\nu\log |(f,g)|_\nu . 
\end{align}
Let \(A=\tfrac{1}{2}\min\{A_1,A_2\}\) and \(B = B_1 C_5 + \tfrac{A}{2}\). Summing \eqref{eqq: ineq1} and \eqref{eqq: ineq2}, and using \eqref{eqq: ineq0} and \eqref{eqq: ineq00}, we obtain
\[
\langle f,g\rangle
\ge
\sum_{\nu\ \mathrm{good}} N_\nu \langle f,g\rangle_\nu
\ge
A\,h(f,g)-B.
\]
This completes the proof.
\end{proof}

We now prove the second uniform lower bound on our global energy pairing. 

\begin{theorem}
\label{thm: glbenergybd2}
Let  \(d_1,d_2\ge 2\) be fixed integers.   Then there exists a constant \(\delta=\delta(d_1,d_2)>0\), such that for any two  polynomials $(f, g) \in \mathrm{Poly}_{\mathrm{mc}}^{d_1}\times\mathrm{Poly}^{d_2}(\ovl{\bb{Q}})$ satisfying
\[
f^{n_1}\neq \lambda\, g^{n_2}
\quad
\text{for all } \lambda\in \Sigma_f
\text{ and all } n_1,n_2\in \mathbb Z_{>0},
\]
we have
\[
\langle f,g\rangle \ge \delta.
\]
\end{theorem}
\begin{proof}
The case where $f$ is monomial and $g$ is monic follows from \cite[Proposition 5.6]{MY25}. The proof easily extends to the non-monic case as $g$ is assumed to have a non-zero coefficient that is not the leading coefficient. 
\par 
We may then suppose that $f$ is not monomial.
Let \( V_m \) denote the set of polynomials whose group of dynamical symmetries is equal to \( \bU_m \). Since $
\mathrm{Poly}_{\mathrm{mc}}^{d_1}(\overline{\mathbb{Q}})\setminus \{z^{d_1}\}
= \bigsqcup_{m=1}^{d_1} \left( V_m \setminus \bigcup_{i=2}^{d_1} V_{im} \right),$
it suffices to consider the case where
$f \in V \setminus V'$,
with  $V = V_m$ and $V' = \bigcup_{i=2}^d V_{im}.$
By construction, any \( f \in V \setminus V' \) satisfies \( \Sigma_f = \bU_m \).

If there exist positive integers $n_1,n_2$ such that $d_1^{n_1}=d_{2}^{n_2}$, up to iteration, we may suppose $d_1=d_2=d.$

By Theorem~\ref{thm: glbenergybd1}, if $h(f,g)$ is sufficiently large, then the conclusion follows.
Hence, we may assume that there exists a constant $C>1$ such that
$h(f,g)<C.$

Let $|\cdot|_\infty$ denote the standard Euclidean norm on $\C$.
We consider pairs $(f,g)$ satisfying the following four conditions:
\begin{enumerate}
    \item $\log|f|_{\infty} \le 8C$;
    \item $\log|g|_{\infty} \le 8C$;
    \item $-\log\max_{1\le i\le m}|p_i(f)|_{\infty} \le 8CM$.
\end{enumerate}
When $d_1=d_2=d$, we further suppose
\begin{enumerate}[resume]
  \item $\min_{\lambda\in\bU_m}\log|f-\lambda g|_\nu \ge -8m(d+1)(2C+\log 2)$.
  \end{enumerate}

These conditions define a compact subset
$Y \subset (V\setminus V')\times \mathrm{Poly}^{d_2}(\C)\setminus E,$
where $E=\emptyset$  if $d_1\not=d_2$ and
$E=\{((V\setminus V')\times \mathrm{Poly}^{d_2}(\C) : f=\lambda g \text{ for some } \lambda\in\bU_m\}$ if $d_1=d_2.$

By the complex rigidity of Julia sets~\cite{SS95}, for $(f,g)\in X$ we have
$\langle f,g\rangle_\infty=0$ if and only if $(f,g)\in E$.
Therefore, by continuity of the local energy pairing and compactness of $Y$,
there exists a constant $\delta>0$ such that $\langle f,g\rangle_{\infty} \ge 2\delta$ for all $(f,g)\in Y.$

Now choose a number field $K$ containing the coefficients of both $f$ and $g$. 
Denote by $\sigma_\nu:K\hookrightarrow\C$ the embedding corresponding to an Archimedean place $v\mid\infty$.

We say that an Archimedean place $\nu$ is \emph{good} if $
(\sigma_\nu(f),\sigma_\nu(g))\in Y.$ We claim that if $h(f,g)<C$, then $
\sum_{\nu\ \mathrm{good}} N_\nu \ge \frac12.$

Assuming the claim, we obtain
\[
\langle f,g\rangle
\ge
\sum_{\nu\mid\infty} N_\nu \langle f,g\rangle_\nu
\ge
\sum_{\nu\ \mathrm{good}} N_\nu \langle f,g\rangle_\nu
\ge
\sum_{\nu\ \mathrm{good}} 2N_\nu\delta
\ge \delta,
\]
which completes the proof.

\smallskip

It remains to prove the claim. Since $\sum_{\nu\mid\infty} N_\nu = 1$, if the claim were false, we would have
$\sum_{\nu\ \mathrm{bad}} N_\nu > \frac12.$
Each bad place $\nu$ must violate at least one of the defining conditions of $Y$,
that is, at least one of the following holds:
\begin{enumerate}
\item $\log|f|_\nu > 8C$,
\item $\log|g|_\nu > 8C$,
\item $-\log\max_{1\le i\le l}|p_i(f)|_\nu > 8CM$.
\end{enumerate}
When $d_1=d_2=d$, we further suppose
\begin{enumerate}[resume]
  \item $\min_{\lambda\in\bU_m}\log|f-\lambda g|_\nu < -8m(d+1)(2C+\log 2)$.
\end{enumerate}

By the pigeonhole principle, for at least one of these four cases the total sum
of the corresponding $N_\nu$ is at least $1/8$.

If this occurs in case~(1), then
\[
h(f,g)\ge \sum_{\substack{\nu\mid\infty\\ |f|_\nu>8C}} N_\nu\log|f|_\nu
> C,
\]
contradicting $h(f,g)<C$.
Case~(2) is identical.

In case~(3), write $f(z)=\sum_{n=0}^{d_1} a_n z^n$ and  $g(z)=\sum_{n=0}^{d_2} b_n z^n.$ Then  there exists $i$ such that $a_i\not=0$ while
$\log^+|a_i^{-1}|_\nu \ge 8C.$ Hence
\[
h(f,g)\ge h(a_i)\ge\sum_ {\substack{\nu\mid\infty\\ \text{(3) holds}}}N_\nu\log^+|a_i^{-1}|_\nu \ge C,
\]
a contradiction.

Finally, suppose case~(4) occurs. Then there exists $i$ and $\lambda\in\bU_m$ such that $a_i\neq \lambda b_i$ and 
$\log|a_i-\lambda b_i|_\nu < -8(2C+\log 2).$
Thus \[h(a_i-\lambda b_i)=h((a_i-\lambda b_i)^{-1})\ge \sum_ {\substack{\nu\mid \infty\\ \text{(4) holds}}}N_\nu\log(|a_i-\lambda b_i|_\nu^{-1}) > 2C+\log 2.\]
By Lemma~\ref{lem: height}, one has
\[
C>h(f,g)\ge \frac12\bigl(h(a_i)+h(b_i)\bigr)
\ge \frac12\bigl(h(a_i-\lambda b_i)-\log 2\bigr)
> C,
\]
again a contradiction.

Therefore all cases lead to contradictions, and the claim follows.
\end{proof}

\subsection{Proof of Theorem~\ref{thm: uniformprep}} We may now combine Theorems \ref{thm: glbenergybd1} and \ref{thm: glbenergybd2} to prove Theorem \ref{thm: uniformprep}. By a specialization argument~\cite[\S 10.2]{DKY21}, we may assume that \(f,g \in \ovl{\bb{Q}}[z]\).
By conjugating both \(f\) and \(g\) by an affine transformation, we may further assume that
$f$ is monic centered. 

\par 
Suppose there exist \(\lambda \in \Sigma_f\) and positive integers $n_1,n_2$ such that \(f^{n_1} = \lambda g^{n_2}\). Let $D=d_1^{n_1}=d_2^{n_2}$.
By induction, for all \(n \in \mathbb{N}\), the iterations of $f$ and $g$ satisfy
\[
f^{n_1n}(z)
= \lambda^{1 + D + \cdots + D^{\,n-1}} g^{n_2n}(z).
\]
Since \(\lambda\) is a root of unity, it follows that \(z\) has finite orbit under \(f\) if and only if it has finite orbit under \(g\).
Therefore, $\Prep(f) = \Prep(g).$

\smallskip

Now suppose otherwise.
By \cite[Theorem B]{Gau24}, if
$N=\bigl|\Prep(f)\cap \Prep(g)\bigr|\geq 2,$
then there exists a constant $C=C(d_1,d_2)>0$ such that 
\begin{equation}\label{eq: AY}
\langle f,g\rangle
\le
C\left(\frac{1}{N}+\frac{\log N}{N}\right)\bigl(h(f,g)+1\bigr)\leq 2C\frac{\log N}{N}(h(f,g)+1).
\end{equation}

On the other hand, by Theorems~\ref{thm: glbenergybd1} and~\ref{thm: glbenergybd2}, there exist uniform constants \(A,B,\delta>0\), depending only on the degrees, such that $
\langle f,g\rangle \ge \delta$ and $\langle f,g\rangle \ge A\,h(f,g)-B.$

Suppose $h(f,g)\ge\frac{2(B+2C)}{A}.$
We claim that \[\frac{\log N}{N}\geq\frac{A}{4C},\] hence $N$ is uniformly bounded. If the claim does not hold, applying \eqref{eq: AY}, we have 
\[B+2C<\left(A-2C\frac{\log N}{N}\right)h(f,g)\leq B+2C,\] 
which leads to a contradiction.

Suppose $h(f,g)<\frac{2(B+2C)}{A}.$ Applying \eqref{eq: AY} again, we obtain 
\[\frac{\log N}{N}\geq\frac{\delta A}{2C(2B+4C+A)}.\]
Thus $N$ is uniformly bounded, which completes the proof.

\section{DeMarco--Mavraki Conjecture for Fibered Powers of $\Delta$}

\subsection{DeMarco--Mavraki conjecture}  We now prove the following special case of the relative Dynamical Bogomolov conjecture. Let $X \subseteq \mathrm{Poly}_{\mathrm{mc}}^{d_1} \times \mathrm{Poly}^{d_2}$ be a closed irreducible subvariety of dimension $k$ defined over $\ovl{\bb{Q}}$. Let $\Delta \subseteq (\bb{P}^1)^2$ be the diagonal and consider the fibered power $\Delta^{k+1} \times X \subseteq (\bb{P}^1)^{2(k+1)} \times \mathrm{Poly}_{\mathrm{mc}}^d \times \mathrm{Poly}^d$. For a point $x=(x_1,x_2) \in \bP^{1}(\C)\times \bP^{1}(\C)$, we will let $\h_{f,g}(x) = \h_f(x_1) + \h_g(x_2).$

\begin{theorem} \label{RelativeManinMumford1}
Assume that as $(f,g)$ varies over $X(\ovl{\bb{Q}})$, the map $f$ is not always $z^{d_1}$. Further assume that $\Prep(f) \not = \Prep(g)$ for some $(f,g) \in X$. Then there exists $\eps > 0$ the set 
$$\left\{(x_1,\ldots,x_{k+1}, f,g) \text{ } \bigg| \text{ } \sum_{i=1}^{k+1} \h_{f,g}(x_i) < \eps \right\} \subset (\Delta)^{k+1} \times X$$
is not Zariski dense in $(\Delta)^{k+1} \times X$. 
\end{theorem}

We have to exclude the case where $f$ is always a monomial map as there are counterexamples. For example, consider the family $(z^d, cz^d)$ as $c$ varies in $\bb{A}^1$. Then if $c$ is a root of unity, we have $\Prep(z^d) = \Prep(cz^d)$ which gives us a Zariski dense set of bad parameters in $\bb{A}^1$. 

Theorem \ref{RelativeManinMumford1} can be viewed as a special case to the DeMarco--Mavraki conjecture \cite[Conjecture 1.1]{DM26}, which we will recall the statement now. Let $S$ be a smooth and irreducible quasi-projective variety over $\bb{C}$ and let $\Phi: S \times \bb{P}^N \to S \times \bb{P}^N$ be an algebraic family of endomorphisms of degree $d$, i.e. $\Phi(s,x) = \Phi(s,f_s(x))$ where $f_s: \bb{P}^n \to \bb{P}^N$ is an endomorphism of degree $d$. The smoothness of $S$ is not important as one can always apply resolution of singularities to the base to obtain a smooth family. We will denote the function field of $S$ by $\bb{C}(S)$. Let $\cal{X} \subseteq S \times \bb{P}^N$ be a closed irreducible subvariety flat over $S$. Given an irreducible subvariety $\cal{Y} \subset S \times \bb{P}^N$ which is flat over a Zariski open, we say $\cal{Y}$ is $\Phi$-special if there exists a subvariety $\mbf{Z} \subseteq \bb{P}^N_{\ovl{\bb{C}(S)}}$ containing the generic fiber $\mbf{Y}$ of $\cal{Y}$ and a polarizable endomorphism $\mathbf{\Psi}: \mathbf{Z} \to \mbf{Z}$ over $\ovl{\bb{C}(S)}$ such that for some $n \in \bb{N}$, the following hold:

\begin{enumerate}
    \item $\mbf{\Psi}^n(\mbf{Z}) = \mbf{Z}$;
    \item $\mbf{\Psi}^n \circ \mbf{\Phi} = \mbf{\Phi} \circ \mbf{\Psi}^n$ on $\mbf{Z}$ where $\mbf{\Phi}$ is the generic fiber of $\Phi$;
    \item $\mbf{Y}$ is preperiodic for $\Psi$.
\end{enumerate}

We will let $r_{\Phi,\cal{X}}$ denote the relative special dimension of $\cal{X}$ over $S$, i.e. 
$$r_{\Phi,\cal{X}} := \min\{ \dim \cal{Y} - \dim S \mid \cal{Y} \text{ is } \Phi-\text{special} \text{ and } \cal{X} \subset \cal{Y}\}.$$
Now let $\wh{T}_{\Phi}$ denote the dynamical $(1,1)$-Green current associated to $\Phi$ \cite{GV25}. We can now state the DeMarco--Mavraki conjecture. 

\begin{conjecture}\label{DMConj1}
The following are equivalent:
\begin{enumerate}
    \item $\cal{X}$ contains a Zariski dense set of $\Phi$-preperiodic points; 
    \item $\wh{T}_{\Phi}^{\wedge r_{\Phi,\cal{X}}} \wedge [\cal{X}] \not = 0.$
\end{enumerate}
\end{conjecture}

The implication $(2 \Rightarrow (1)$ is \cite[Theorem 1.5]{DM26}. Hence the difficulty of the conjecture is to show that if $\wh{T}_{\Phi}^{\wedge r_{\Phi,\cal{X}}} \wedge[\cal{X}] = 0$, then $\cal{X}$ does not contain a Zariski dense set of preperiodic points. 
\par 
The DeMarco--Mavraki conjecture can be viewed as a dynamical generalization of the Relative Manin--Mumford conjecture, proven by Gao and Habegger \cite{GH23}. The reader may consult \cite{DM26} for some striking consequences of the conjecture. The relative Bogomolov conjecture was also proven for fibered powers of elliptic curves by Kühne \cite{Kuh23}. 
\par 
Combining Theorem \ref{RelativeManinMumford1} and implication $(2) \Rightarrow (1)$ for Conjecture \ref{DMConj1}, we obtain that the DeMarco--Mavraki conjecture is true for $\cal{X} = (\Delta)^{k} \times X$ for any $k$. We restrict ourselves to the case of equal degrees as the DeMarco--Mavraki conjecture is originally stated only for polarized dynamical systems.

\begin{theorem} \label{DMConj2}
Let $X \subseteq \mathrm{Poly}_{\mathrm{mc}}^d \times \mathrm{Poly}^d$ be an irreducible subvariety defined over $\ovl{\bb{Q}}$. Assume that as $(f,g)$ varies over $X(\ovl{\bb{Q}})$, the map $f$ is not always $z^d$. Fix a positive integer $k$ and let 
$$\cal{X} = (\Delta)^k \times X \subseteq (\bb{P}^1)^{2k} \times X,$$
where we view $\cal{X}$ as a family over the base $X$. 
Then the DeMarco--Mavraki conjecture is true for $\cal{X}$.
\end{theorem}

As remarked previously, Theorem \ref{RelativeManinMumford1} can be viewed as a relative version of the Dynamical Manin--Mumford. It is easy to deduce a uniform bound on $|\Prep(f) \cap \Prep(g)|$ over a Zariski open $U$ of $X$ from Theorem \ref{RelativeManinMumford1}. Conversely, if one knew the existence of a Zariski open $U$ and $N = N(d) > 0$ such that $|\Prep(f) \cap \Prep(g)| \leq N$ for all $(f,g) \in U$, then the set 
$$\{(x_1,\ldots,x_{N+1},f,g) \mid  x_i \in \Prep(f) \cap \Prep(g)\} \subset (\Delta)^{N+1} \times U$$
is contained in the union of $\{x_i = x_j\}$ over all pairs $(i,j)$. The extra content of Theorem \ref{RelativeManinMumford1} is that we only have to pass to a $(\dim X + 1)$-fibered power to have the special points lie in a Zariski closed subset. 
\par 
We will deduce Theorem \ref{RelativeManinMumford1} by combining \cite[Theorem 4.1]{Yua24} with our uniform lower bounds in a similar fashion that Yuan deduces his uniform Bogomolov-type statements. We now recall the theory of adelic line bundles over quasi-projective varieties. 

\subsection{Adelic line bundles over quasi-projective varieties}
We will be using the theory of adelic line bundles over quasi-projective varieties which the reader may refer to \cite{YZ24} for the details. We first briefly review the basic definitions. Let $X$ be a flat integral projective scheme over $\Spec \bb{Z}$. Then an arithmetic $\bb{Q}$-divisor $\cal{D}$ is a pair $(D,g_{\cal{D}})$ where $D$ is a $\bb{Q}$-Cartier divisor and $g_{\cal{D}}$ a continuous Green's function for the divisor $D$ on $X(\bb{C})$. Now if $U$ is a quasi-projective scheme flat over $\Spec \bb{Z}$, we define the space of arithemtic model divisors $\wh{\Div}_{\bb{Q}}(U)_{\mathrm{mod}}$ as 
$$\wh{\Div}_{\bb{Q}}(U)_{\mathrm{mod}} = \lim_{\rightarrow} \wh{\Div}_{\bb{Q}}(X)$$
where $X$ ranges over all compactifications of $U$. Fixing some compactification $\ovl{U}$, we may define a boundary divisor $\cal{E}$ which is a model divisor whose suppport is $\ovl{U} \setminus U$ and whose Green's function satisfies $g_{\cal{E}} > 0$ everywhere. We may then endow $\wh{\Div}_{\bb{Q}}(U)_{\mathrm{mod}}$ with a suitable boundary topology and form its completion, giving us the space $\wh{\Div}_{\bb{Q}}(U)$. Quotienting out by a suitable equivalence gives us the space of adelic line bundles $\wh{\Pic}(U/\bb{Z})$. 
\par 
Finally if $X$ is simply a quasi-projective variety over a number field $K$, we will let 
$$\wh{\Pic}(X/\bb{Z}):= \lim_{\rightarrow} \wh{\Pic}(\cal{U}/\bb{Z})$$
where the limit is taken over all quasi-projective models $\cal{U}$ of $X$ over $\Spec \bb{Z}$. 
\par 
The space $\wh{\Pic}(X/\bb{Z})$ is usually too big and we will work with a subset of it known as the integrable adelic line bundles for which a good notion of intersection theory can be developed. We will denote the set of integrable adelic line bundles by $\wh{\Pic}(X/\bb{Z})_{\intt}$. 
\par 
We now start recalling some basic results that we will now need. Given an integral scheme $S$, we say $(X,f,L)$ is a polarized dynamical system over $S$ if the following holds:
\begin{enumerate}
\item $X$ is an integral projective scheme over $S$;

\item $f: X \to X$ is a morphism over $S$;

\item $L$ is a $\bb{Q}$-line bundle on $X$, relatively ample over $S$ and there exists $q > 1$ such that $f^*L \simeq qL$. 
\end{enumerate}

\begin{proposition} \label{InvariantAdelic1}
Let $S$ be a quasi-projective variety over a number field $K$ with an affine quasi-projective model over $\bb{Z}$. Let $(X,f,L)$ be a polarized dynamical system over $S$. Then there exists an integrable nef adelic line bundle $\ovl{L}_f \in \wh{\Pic}(X/\bb{Z})$ such that $f^* \ovl{L}_f = q \ovl{L}_f$. 
\end{proposition}

\begin{proof}
This is \cite[Theorem 6.1.1]{YZ24}.     
\end{proof}

We now recall the Deligne pairing for adelic line bundles over quasi-projective varieties. Again let $K$ be a number field and let $Y$ be a normal quasi-projective variety over $K$. Let $f: X \to Y$ be a projective and flat morphism of relative dimension $n$ and assume that $X$ is integral. Then the Deligne pairing \cite[Theorem 4.1.3]{YZ24} induces a symmetric and multilinear functor 
$$\wh{\Pic}(X/\bb{Z})_{\intt}^{n+1} \to \wh{\Pic}(Y/\bb{Z})_{\intt}.$$
We denote the image of $(\ovl{L}_1,\ldots,\ovl{L}_{n+1})$ by $f_*\langle \ovl{L}_1,\ldots, \ovl{L}_{n+1} \rangle$. Furthermore, the Deligne pairing functor is compatible by base change, giving us the following interpretation of heights under the Deligne pairing. 

\begin{proposition} \label{HeightSubvariety1}
Let $\ovl{L}_1,\ldots,\ovl{L}_{n+1}$ be $n+1$ integrable adelic line bundles on $X$ and let $\ovl{M} = \langle \ovl{L}_1,\ldots,\ovl{L}_{n+1} \rangle \in \wh{\Pic}(Y/\bb{Z})_{\intt}$. Let $s \in Y(\ovl{\bb{Q}})$. Then
$$h_{\ovl{M}}(s) = h_{\ovl{L}_1,\ldots,\ovl{L}_{n+1}}(X_s).$$
\end{proposition}

\begin{proof}
This follows from functoriality of the Deligne pairing by base change, see the proof of \cite[Lemma 6.3.2]{YZ24}.     
\end{proof}

We next recall the definition of a big line bundle. For any adelic line bundle $\ovl{L}$, we may define an arithmetic volume function $\wh{\vol}(\ovl{L})$ \cite[Section 5.2]{YZ24}. We say that $\ovl{L}$ is big if $\wh{\vol}(\ovl{L}) > 0$. The bigness of $\ovl{L}$ directly implies the non-Zariski density of small points for $h_{\ovl{L}}$, as seen in \cite[Theorem 5.3.7]{YZ24}. We will need a weaker condition called potential bigness. Again for our morphism $f: X \to Y$ of pure relative dimension $n$, we say that an adelic line bundle $\ovl{L}$ on $X$ is potentially big if the adelic line bundle 
$$\ovl{L}^{\boxtimes m} = (p_1^* \ovl{L}) \otimes \cdots \otimes (p_{m}^* \ovl{L})$$
is big on the $m$-fold fiber product 
$$X_{/Y}^{m} = X \times_Y \cdots \times_Y X$$
for all $m \geq \dim Y+1$. 
The following proposition of Yuan gives a criterion for potential bigness in terms of the bigness of the Deligne pairing. 

\begin{proposition} \label{YuanBigness1}
Assume that the generic fiber of $f: X \to Y$ is geometrically integral. Let $\ovl{L}$ be an adelic line bundle on $X$ which is nef, and that the degree of $L$ on the fibers $X \to Y$ is strictly positive. Then $\ovl{L}$ is potentially big if and only if the Deligne pairing 
$f_*(\ovl{L},\ldots,\ovl{L})$ is big, 
\end{proposition}

\begin{proof}
This is \cite[Theorem 4.1]{Yua24}. Note that we have to take $m \geq \dim Y+1$ as the dimension in Yuan's theorem is the dimension over $\bb{Z}$ and thus we have to add one to it.     
\end{proof}

Finally we have a criterion for bigness of a line bundle $\ovl{L}$ on $X$. 

\begin{proposition} \label{YuanBigness2}
Let $\ovl{L}$ and $\ovl{M}$ be adelic line bundle on $X$ and assume that $\ovl{L}$ is nef and $\ovl{M}$ is big. Then if there exists a Zariski open $U \subseteq X$ and constants $\eps > 0$ such that 
$$h_{\ovl{L}}(x) \geq \eps h_{\ovl{M}}(x)\text{ for all } x \in U(\ovl{\bb{Q}}),$$
then $\ovl{L}$ is also big. 
\end{proposition}

\begin{proof}
We combine a few results in \cite[Section 5]{YZ24}. First by \cite[Theorem 5.3.8]{YZ24}, if $\tilde{L}$ denotes the image of $\ovl{L}$ in $\wh{\Pic}(X/K)$, then $\tilde{L}$ is big. By \cite[Lemma 5.3.4]{YZ24}, as $\wh{\vol}(\ovl{M}) > 0$, it follows that its first essential minimum $e_1(X,\ovl{M}) > 0$. Then our inequality implies that $e_1(X,\ovl{L}) > 0$ too. By \cite[Theorem 5.3.3]{YZ24}, as we know that $\tilde{L}$ is big, it follows that $h_{\ovl{L}}(X) > 0$. By \cite[Theorem 5.2.2]{YZ24}, it follows that $\wh{\vol}(\ovl{L}) > 0$ and hence $\ovl{L}$ is big as desired.
\end{proof}

\subsection{Proofs of Theorems \ref{RelativeManinMumford1} and \ref{DMConj2}} We now begin the proofs of Theorem \ref{RelativeManinMumford1} and \ref{DMConj2}. We first begin with the setup. Consider $Y = \bb{P}^1 \times \mathrm{Poly}_{\mathrm{mc}}^{d_1}$ as a variety over the base $\mathrm{Poly}_{\mathrm{mc}}^{d_1}$ which is a quasi-projective variety over $\ovl{\bb{Q}}$. Then we have a natural endomorphism $f: Y \to Y$ over $S$ which is polarizable of degree $d \geq 2$ using the line bundle $L = O(1)$. By Proposition \ref{InvariantAdelic1}, there exists an integrable adelic line bundle $\ovl{L}_f \in \wh{\Pic}(Y/\bb{Z})$ satisfying $f^* \ovl{L}_f \simeq d_1 \ovl{L}_f$. Similarly for $Z = \bb{P}^1 \times \mathrm{Poly}^{d_2}$, we obtain an adelic line bundle $\ovl{L}_g \in \wh{\Pic}(Z/\bb{Z})$ satisfying $g^* \ovl{L}_g \simeq  d_2 \ovl{L}_g$. 
\par 
Now on $Y \times Z = \bb{P}^1 \times \mathrm{Poly}_{\mathrm{mc}}^{d_1} \times \bb{P}^1 \times \mathrm{Poly}^{d_2}$, we may consider the adelic line bundle $\ovl{L} = \pi_1^*\ovl{L}_f - \pi_2^*\ovl{L}_g \in \wh{\Pic}(Y \times Z / \bb{Z})$ and we restrict it to $\Delta' = \Delta \times \mathrm{Poly}_{\mathrm{mc}}^{d_1} \times \mathrm{Poly}^{d_2}$ to obtain an adelic line bundle $\ovl{L} \in \wh{\Pic}(\Delta'/\bb{Z})$. Now $\Delta'$ is a flat family of relative dimension $1$ over the base $\mathrm{Poly}_{\mathrm{mc}}^d \times \mathrm{Poly}^d$ and the base is certainly normal. We may thus consider the Deligne pairing 
$$\ovl{M} = \langle \ovl{L}, \ovl{L} \rangle \in \wh{\Pic}(\mathrm{Poly}_{\mathrm{mc}}^{d_1} \times \mathrm{Poly}^{d_2}/\bb{Z}).$$
Now consider a subvariety $X \subset \mathrm{Poly}_{\mathrm{mc}}^{d_1} \times \mathrm{Poly}^{d_2}$. Let $\nu: X' \to X$ be its normalization. We aim to prove that $\ovl{L}$ is potentially big on $\Delta \times X$. 

\begin{proposition} \label{Relative1}
Assume that as $(f,g)$ varies over $X(\ovl{\bb{Q}})$, the map $f$ is not always the monomial map.  If not all $s \in X(\ovl{\bb{Q}})$ satisfy $\Prep(f_s) = \Prep(g_s)$, then the adelic line bundle $\ovl{L}$ is potentailly big on $\Delta \times X$.

\end{proposition}

\begin{proof}
We first show that the set of points $s \in X(\ovl{\bb{Q}})$ that do satisfy $\Prep(f_s) = \Prep(g_s)$ is contained in a proper Zariski closed subset.If $\Prep(f) = \Prep(g)$, then for some Archimedean place $\nu$, we must have $\mu_{f,v} = \mu_{g,v}$. Fix some $(n_1,n_2)$ such that $d_1^{n_1} = d_2^{n_2}$. If they don't exist then the locus where $\mu_{f,v} = \mu_{g,v}$ is just empty. Otherwise, we obtain $f^{n_1} = \lambda g^{n_2}$ for some $\lambda \in \Aut(J(f))$. On a Zariski open $U \subseteq X$, we have that $\Aut(J(f)) = \{ \zeta^m = 1 \}$ for some $m \geq 1$ as $f$ is generically not the monomial map Then $\{f^{n_1} = \lambda g^{n_2}\}$ for $\lambda \in \Aut(J(f))$ cuts out a Zariski closed subset of $X$ which is not the whole space as desired as otherwise $\Prep(f_s) = \Prep(g_s)$ for all $s \in X(\ovl{\bb{Q}})$.  
\par 
Let $W$ be a Zariski closed subset containing all $s$ for which $\Prep(f_s) = \Prep(g_s)$ and let $U$ be its complement. Now by Theorems \ref{thm: glbenergybd1} and \ref{thm: glbenergybd2}, along with Proposition \ref{HeightSubvariety1}, it follows that 
$$h_{\ovl{M}}(s) \geq A h(f,g)+B$$
for some constants $A,B > 0$, for all $s \in U(\ovl{\bb{Q}})$. Since $\nu: X' \to X$ is birational, the potential bigness of $\ovl{L}$ on $X \times \Delta$ is equivalent to that of $\nu^* \ovl{L}$ on $X' \times \Delta$. Using Proposition \ref{YuanBigness1}, it suffices to prove that $\nu^* \ovl{M}$ is big, which again is equivalent to $\ovl{M}$ being big. Now as the Weil height $h(f,g)$ comes from the height of a big adelic line bundle, by Proposition \ref{YuanBigness2} it follows that $\ovl{M}$ is big and hence $\ovl{L}$ is potentially big on $X \times \Delta$.  
\end{proof}

\begin{theorem} \label{RelativeManinMumford2}
Assume that as $(f,g)$ varies over $X(\ovl{\bb{Q}})$, the map $f$ is not always $z^{d_1}$. Further assume that $\Prep(f) \not = \Prep(g)$ for some $(f,g) \in X$. Then there exists $\eps > 0$ the set 
$$\left\{(x_1,\ldots,x_{k+1}, f,g) \text{ } \bigg| \text{ } \sum_{i=1}^{k+1} \h_{f,g}(x_i) < \eps \right\} \subset (\Delta)^{k+1} \times X$$
is not Zariski dense in $(\Delta)^{k+1} \times X$. 
\end{theorem}

\begin{proof}
By Proposition \ref{Relative1}, we know that $\ovl{L}$ is potentially big on $X \times \Delta$ over $X$. Hence if $k = \dim X$, it follows that $\ovl{L}^{\boxtimes k+1}$ is big on $\Delta^{k+1} \times X$. The height function induced by $\ovl{L}^{\boxtimes k+1}$ is equivalent to $\sum_{i=1}^{k+1} \h_{f,g}(x_i)$. It thus follows there exists an $\eps > 0$ points of height $< \eps$ is not Zariski dense in $(\Delta)^{k+1} \times X$ as desired.
\end{proof}

We now prove Theorem \ref{DMConj2}. 

\begin{theorem} \label{DMConj3}
Let $X \subseteq \mathrm{Poly}_{\mathrm{mc}}^d \times \mathrm{Poly}^d$ be an irreducible subvariety defined over $\ovl{\bb{Q}}$. Assume that as $(f,g)$ varies over $X(\ovl{\bb{Q}})$, the map $f$ is not always $z^d$. Fix a positive integer $k$ and let 
$$\cal{X} = (\Delta)^k \times X \subseteq (\bb{P}^1)^{2k} \times X,$$
where we view $\cal{X}$ as a family over the base $X$. 
Then the DeMarco--Mavraki conjecture is true for $\cal{X}$.
\end{theorem} 

\begin{proof}
Let's first handle the case where $\Prep(f_s) = \Prep(g_s)$ for every pair $(f,g) \in X(\ovl{\bb{Q}})$. Then there exists a constant $\lambda$ so that $f = \lambda g$ is true for all pairs $(f,g) \in X(\ovl{\bb{Q}})$ with $\lambda \in \Aut(\cal{J}(f_{\eta})$, where $f_{\eta}$ denotes the generic $f$. In particular, $\cal{X}$ itself is preperiodic for $\Phi$ and hence is $\Phi$-special. Thus $r_{\Phi,\cal{X}} = \dim \cal{X} - \dim X = k$.
\par 
We may now verify that 
$$\wh{T}_{\Phi}^{\wedge k} \wedge [\cal{X}] \not = 0.$$
Indeed by slicing, it suffices to prove that for $s \in X(\ovl{\bb{Q}})$, the current $\wh{T}_{\Phi,s}^{\wedge k}$ is non-zero on the fiber $(\Delta)^k$. In this case if we identify $\Delta$ with $\bb{P}^1$, then we are looking at the current $T = \sum_{i=1}^{k} \pi_i^*(\wh{T}_{f_s} + \wh{T}_{g_s})$ where $\pi_i: (\bb{P}^1)^k \to \bb{P}^1$ is the $i$th projection map and $\wh{T}_{f_s}, \wh{T}_{g_s}$ are the $(1,1)$-dynamical currents for the pair $(f_s,g_s)$. Now expanding out $T^{\wedge k}$ gives us a sum of non-negative currents, one of which is $\wedge_{i=1}^{k} \pi_i^* \wh{T}_{f_s}$, which is clearly positive as $\pi_i$'s are from independent coordinates. Hence $\wh{T}_{\Phi}^{\wedge k} \not = 0$, which verifies the DeMarco--Mavraki conjecture in this setting.
\par 
We now assume that there is at least one pair $(f_s,g_s)$ such that $\Prep(f_s) \not = \Prep(g_s)$. This puts us in the setting of Theorem \ref{RelativeManinMumford1}. Note that as $(2) \implies (1)$ is already known, if one knew that $\cal{X}$ does not contain a Zariski dense set of $\Phi$-preperiodic point then the DeMarco--Mavraki conjecture is automatically verified in this case. In particular, for $k \geq \dim X+1$, as Theorem \ref{RelativeManinMumford1} implies that $\cal{X}$ does not have a Zariski dense set of $\Phi$-preperiodic points and thus the DeMarco--Mavraki conjecture is true for $k \geq \dim X + 1$.
\par 
Now assume that $k \leq \dim X$. Then \cite[Theorem 3.1]{DM24} implies that $\cal{X}$ contains a Zariski dense set of preperiodic points. It thus suffices to show that $\wh{T}_{\Phi}^{r_{\Phi,\cal{X}}} \wedge [\cal{X}] \not = 0$. First, we have $r_{\Phi,\cal{X}} \leq 2k$ and we only need to show that $\wh{T}_{\Phi}^{\wedge 2k} \wedge [\cal{X}] \not = 0$. When $k = \dim X$, this is equivalent to non-degeneracy of $\cal{X}$. This is implied by the proof of Theorem \ref{RelativeManinMumford1}, as we have bigness of $\sum_{i=1}^{m} p_i^* \tilde{L}$ on $\cal{X}$ where $m = \dim X$, using again \cite[Theorem 4.1]{Yua24}, which gives us non-degeneracy of $\cal{X}$. 
\par 
For $k < \dim X$, let us assume for contradiction that $\wh{T}_{\Phi}^{\wedge 2k} \wedge [\cal{X}] = 0$. Identifying $\Delta$ with $\bb{P}^1$ again and letting $\pi_i: \Delta^k \times X \to \Delta \times X$ be the projection to the $i$th coordinate, we have that 
\begin{equation} \label{eq: Wedge1}
\left(\sum_{i=1}^{k} (\pi_i^*\wh{T}_f + \pi_i^* \wh{T}_g)\right)^{\wedge 2k} = 0 \text{ on } \Delta^k \times X.
\end{equation}
On the other hand for $m = \dim X$, we have 
\begin{equation} \label{eq: Wedge2}
\left(\sum_{i=1}^{m} (\pi_i^* \wh{T}_f + \pi_i^* \wh{T}_g) 
\right)^{\wedge 2m} \not = 0 \text{ on } \Delta^m \times X.
\end{equation}
We now expand out \eqref{eq: Wedge2}. Since $\wh{T}_f^2 = \wh{T}_g^2 = 0$ (see \cite{GV25}), the only surviving term is 
\begin{equation} \label{eq: Wedge3}
\pi_1^* \wh{T}_f \wedge \pi_1^* \wh{T}_g \wedge \cdots \wedge \pi_m^* \wh{T}_f \wedge \pi_m^* \wh{T}_g.
\end{equation}
Now let $\pi: \Delta^m \times X \to \Delta^k \times X$ be the projection to the first $k$ coordinates. Then we have 
$$\pi^*(\pi_1^* \wh{T}_f \wedge \cdots \wedge \pi_k^* \wh{T}_g) \wedge \pi^*_{k+1} \wh{T}_f \wedge \cdots \wedge \pi_m^* \wh{T}_g \not = 0.$$
This implies that 
$$\pi_1^* \wh{T}_f \wedge \cdots \wedge \pi_k^* \wh{T}_g \not = 0 \text{ on } \Delta^k \times X,$$
contradicting \eqref{eq: Wedge1}. Hence $\wh{T}_{\Phi}^{\wedge 2k} \wedge [\cal{X}] \not = 0$ as desired.
\end{proof}

\printbibliography

\end{document}